\documentclass[11pt]{amsart}
\usepackage{amssymb,amsmath,amsthm,amsfonts,mathrsfs,caption}
\usepackage[hmargin=3cm,vmargin=3.5cm]{geometry}
\usepackage[dvipsnames,table,xcdraw]{xcolor}
\usepackage[colorlinks = true,
            linkcolor = Fuchsia,
            urlcolor  = ForestGreen,
            citecolor = WildStrawberry,
            anchorcolor = blue]{hyperref}
\usepackage{graphicx,psfrag}
\usepackage{amscd}  
\usepackage[shellescape]{gmp}
\usepackage{stmaryrd}  %% double square brackets
\usepackage[all,2cell]{xy} \UseAllTwocells \SilentMatrices
\usepackage{tikz-cd}
\usepackage[dvips]{epsfig}
\usepackage{MnSymbol}
\usepackage{comment}
\usepackage{enumitem}
\usepackage{kbordermatrix}
\usepackage[colorinlistoftodos]{todonotes}

\renewcommand{\kbldelim}{(}% Left delimiter
\renewcommand{\kbrdelim}{)}% Right delimiter
\usetikzlibrary{arrows,automata}
\usetikzlibrary{decorations.markings}

\makeatletter
\def\@adminfootnotes{%
  \let\@makefnmark\relax  \let\@thefnmark\relax
  \ifx\@empty\@date\else \@footnotetext{\@setdate}\fi%%   <------ added
  \ifx\@empty\@subjclass\else \@footnotetext{\@setsubjclass}\fi
  \ifx\@empty\@keywords\else \@footnotetext{\@setkeywords}\fi
  \ifx\@empty\thankses\else \@footnotetext{%
    \def\par{\let\par\@par}\@setthanks}%
  \fi
}
\makeatother

\newtheorem{theorem}{Theorem}[section]

\newtheorem{prop}[theorem]{Proposition}

\newtheorem{corollary}[theorem]{Corollary}

\theoremstyle{definition}

\newtheorem{example}[theorem]{Example}

\theoremstyle{remark}
\newtheorem{remark}[theorem]{Remark}

\numberwithin{equation}{section}

\let\oldtocsection=\tocsection
\let\oldtocsubsection=\tocsubsection
\renewcommand{\tocsection}[2]{\hspace{0em}\oldtocsection{#1}{#2}}
\renewcommand{\tocsubsection}[2]{\hspace{1em}\oldtocsubsection{#1}{#2}}
\usepackage{chngcntr}
\counterwithin{figure}{subsection}

\makeatletter
\@namedef{subjclassname@2020}{%
  \textup{2020} Mathematics Subject Classification}

\makeatother
 
\title[TQFT and sofic systems]{Boolean TQFTs with accumulating defects, sofic systems, and automata for infinite words}

\author{Paul Gustafson}
\address{The Ambrus Group, New York, NY 10036, USA}
\email{\href{mailto:paul.gustafson@theambrusgroup.com}{paul.gustafson@theambrusgroup.com}}
\author{Mee Seong Im}
\address{Department of Mathematics, United States Naval Academy, Annapolis, MD 21402, USA}
\email{\href{mailto:meeseongim@gmail.com}{meeseongim@gmail.com}}
\author{Mikhail Khovanov}
\address{Department of Mathematics, Columbia University, New York, NY 10027, USA}
\email{\href{mailto:khovanov@math.columbia.edu}{khovanov@math.columbia.edu}}
\address{Department of Mathematics, Johns Hopkins University, Baltimore, MD 21218, USA}
\email{\href{mailto:khovanov@jhu.edu}{khovanov@jhu.edu}}
\date{December 25, 2023}

\providecommand{\keywords}[1]{\textbf{\textit{Key words and phrases.}} #1}


\begin{document}

\def\ac{\mathsf{ac}}
\def\concatenate{\mathsf{concatenate}}
\def\gen{\mathsf{generators}}
\def\init{\mathsf{in}}
\def\t{\mathsf{t}}
\def\out{\mathsf{out}}
\def\Left{\mathsf{L}}
\def\Right{\mathsf{R}}
\def\I{\mathsf I}
\def\R{\mathbb R}
\def\Q{\mathbb Q}
\def\Z{\mathbb Z}
\def\mc{\mathcal{c}}
\def\finite{\mathsf{finite}}
\def\infinite{\mathsf{infinite}}
\def\N{\mathbb N} 
\def\C{\mathbb C}
\def\S{\mathbb S}
\def\SS{\mathbb S} 
\def\CP{\mathbb P}
\def\Ob{\mathsf{Ob}}
\def\op{\mathsf{op}}
\def\new{\mathsf{new}}
\def\old{\mathsf{old}}
\def\rat{\mathsf{rat}}
\def\rec{\mathsf{rec}}
\def\tail{\mathsf{tail}}
\def\coev{\mathsf{coev}}
\def\ev{\mathsf{ev}}
\def\id{\mathsf{id}}
\def\s{\mathsf{s}}
\def\t{\mathsf{t}}
\def\start{\textsf{starting}}
\def\Notation{\textsf{Notation}}
\def\circleft{\raisebox{-.68ex}{\scalebox{1}[3.00]{\rotatebox[origin=c]{180}{$\curvearrowright$}}}}
\renewcommand\SS{\ensuremath{\mathbb{S}}}
\newcommand{\kllS}{\kk\llangle  S \rrangle} %% power ser
\newcommand{\kllSS}[1]{\kk\llangle  #1 \rrangle}
\newcommand{\klS}{\kk\langle S\rangle}  % nc polynomials
\newcommand{\aver}{\mathsf{av}}  % average 
\newcommand{\ophana}{\overline{\phantom{a}}}
\newcommand{\Bool}{\mathbb{B}}
\newcommand{\dmod}{\mathsf{-mod}}
\newcommand{\lang}{\mathsf{lang}}
\newcommand{\pfmod}{\mathsf{-pfmod}}
\newcommand{\primitive}{\mathsf{irr}}
\newcommand{\Bmod}{\Bool\mathsf{-mod}}  % B-module 
\newcommand{\Bmodo}[1]{\Bool_{#1}\mathsf{-mod}}  
\newcommand{\Bfmod}{\Bool\mathsf{-fmod}} % finite B-modules 
\newcommand{\Bfpmod}{\Bool\mathsf{-fpmod}} % finite projective B-modules
\newcommand{\Bfsmod}{\Bool\mathsf{-}\underline{\mathsf{fmod}}}  % stable category 
\newcommand{\undvar}{\underline{\varepsilon}} %sequence of varepsilons, not using anymore
\newcommand{\RLang}{\mathsf{RLang}}
\newcommand{\undotimes}{\underline{\otimes}}
\newcommand{\sigmaacirc}{\Sigma^{\ast}_{\circ}} % equiv classes of words under rotation 
\newcommand{\cl}{\mathsf{cl}}
\newcommand{\PP}{\mathcal{P}} % powerset 
\newcommand{\wedgezero}{\{ \vee ,0\} } % semilattices
\newcommand{\whA}{\widehat{A}}
\newcommand{\whC}{\widehat{C}}
\newcommand{\whM}{\widehat{M}}
\newcommand{\Sigmalr}{\Sigma^{\Z}}
\newcommand{\Sigmal}{\Sigma^{\t}}
%tail or left 
%\newcommand{\Sigmar}{\Sigma^{r}}
\newcommand{\Sigmar}{\Sigma^{\h}}
%head or right 
\newcommand{\Sigmaa}{\Sigma^{\ast}}
\newcommand{\SigmaZ}{\Sigma^{\Z}}  
\newcommand{\Sigmac}{\Sigma^{\circ}}
\newcommand{\mcF}{\mathcal{F}} 
\newcommand{\mInf}{\mathsf{Inf}}
\newcommand{\mcCinfS}{\Cob^{\infty}_{\Sigma}}

\newcommand{\alphai}{\alpha_I}  % alpha vertical
\newcommand{\alphac}{\alpha_{\circ}}  % alpha circle 
\newcommand{\alphap}{(\alphai,\alphac)} % alpha pair 
\newcommand{\alphalr}{\alpha_{\leftrightarrow}}
\newcommand{\alphaZ}{\alpha_{\Z}}
\newcommand{\mcCinfalpha}{\Cob^{\infty}_{\alpha}}

% redefine emptyset symbol 
\let\oldemptyset\emptyset
\let\emptyset\varnothing

\newcommand{\undempty}{\underline{\emptyset}}
\def\basis{\mathsf{basis}}
\def\irr{\mathsf{irr}} % recognizable series 
\def\spanning{\mathsf{spanning}}
\def\elmt{\mathsf{elmt}}

\def\f{\mathsf{f}}
\def\h{\mathsf{h}}
\def\t{\mathsf{t}}
\def\l{\lbrace}
\def\r{\rbrace}
\def\o{\otimes}
\def\lra{\longrightarrow}
\def\Ext{\mathsf{Ext}}
\def\mf{\mathfrak} 
\def\mcC{\mathcal{C}}
\def\mcO{\mathcal{O}}
\def\Fr{\mathsf{Fr}}

\def\ovb{\overline{b}}
\def\tr{{\sf tr}} 
\def\det{{\sf det }} 
\def\tral{\tr_{\alpha}}
\def\one{\mathbf{1}}   % unit  object of category 

\def\lra{\longrightarrow}
\def\twoheadlra{\longrightarrow\hspace{-4.6mm}\longrightarrow}
\def\hooklra{\raisebox{.2ex}{$\subset$}\!\!\!\raisebox{-0.21ex}{$\longrightarrow$}}
\def\kk{\mathbf{k}}  %% base field  
\def\gdim{\mathsf{gdim}}  %% graded dimension 
\def\rk{\mathsf{rk}}
\def\undep{\underline{\epsilon}}
\def\mathM{\mathbf{M}}  % boolean matrix 

% cobordism categories 
%\def\CCC{\mathcal{C}} % cat of cobordisms 
%\def\wCCC{\widehat{\CCC}}  % completed category

\def\complement{\mathsf{comp}}

\def\Cob{\mathcal{C}} 
\def\Kar{\mathsf{Kar}}   % Karoubi envelope 

\def\dmod{\mathsf{-mod}}   % modules  
\def\pmod{\mathsf{-pmod}}    % projective modules 

\newcommand{\brak}[1]{\ensuremath{\left\langle #1\right\rangle}}
\newcommand{\oplusop}[1]{{\mathop{\oplus}\limits_{#1}}}
\newcommand{\ang}[1]{\langle #1 \rangle } 
\newcommand{\ppartial}[1]{\frac{\partial}{\partial #1}} %partial derivative 

\newcommand{\mcA}{{\mathcal A}}
\newcommand{\cZ}{{\mathcal Z}}
\newcommand{\sq}{$\square$}
\newcommand{\bi}{\bar \imath}
\newcommand{\bj}{\bar \jmath}

\newcommand{\undn}{\mathbf{n}}
\newcommand{\undm}{\mathbf{m}}
\newcommand{\cob}{\mathsf{cob}} % cobordism 
\newcommand{\comp}{\mathsf{comp}} % complementary

\newcommand{\Aut}{\mathsf{Aut}}
\newcommand{\Hom}{\mathsf{Hom}}
\newcommand{\Idem}{\mathsf{Idem}}
\newcommand{\Ind}{\mbox{Ind}}
\newcommand{\Id}{\textsf{Id}}
\newcommand{\End}{\mathsf{End}}
\newcommand{\iHom}{\underline{\mathsf{Hom}}}
\newcommand{\omcC}{\overline{\Cob}}

\newcommand{\drawing}[1]{
\begin{center}{\psfig{figure=fig/#1}}\end{center}}

\def\MS#1{{\color{blue}[MS: #1]}}
\def\MK#1{{\color{red}[MK: #1]}}
\def\PG#1{{\color{magenta}[PG: #1]}}

\begin{abstract} 
Any finite state automaton gives rise to a Boolean one-dimensional TQFT with defects and inner endpoints of cobordisms. This paper extends the correspondence to Boolean TQFTs where defects accumulate toward inner endpoints, relating such TQFTs and topological theories to sofic systems and $\omega$-automata.
\end{abstract}

\maketitle
\tableofcontents

%%%%%%%%%%%%%%%%%%%%%%%
%
%  Intro  
%
%%%%%%%%%%%%%%%%%%%%%%%

\section{Introduction}
\label{section:intro}

This paper continues a study of Boolean one-dimensional TQFTs and topological theories, complementing the papers~\cite{IK-top-automata,IKV23,GIKKL23}, see also~\cite{IK_TQFTjourney23} (and \cite{IK_22_linear,IKV23,IZ} for the corresponding theories over a field). We consider one-dimensional cobordisms decorated by point defects labelled by elements of a finite set $\Sigma$ and allow defects to accumulate toward inner boundary points of a cobordism. Near an inner boundary point, a cobordism is encoded by a semi-infinite sequence of labels $a_1a_2a_3\cdots$, $a_i\in \Sigma$. We then study Boolean topological theories and TQFTs for two versions of this cobordism category and link them to sofic systems in Section~\ref{section:sofic-systems}  and to automata on infinite words ($\omega$-automata) in Section~\ref{sec_one_side}. Examples are worked out in Sections~\ref{sec_examples_sofic} and~\ref{sec_examples_automata}. 

Section~\ref{sec_review} contains a brief review of the relation between automata and one-dimensional Boolean TQFTs as well as that of the universal construction for Boolean topological theories in one dimension. 
In Section~\ref{section:sofic-systems} we consider the category of one-dimensional oriented cobordisms with $\Sigma$-labelled defects and inner endpoints with the defects accumulating towards each inner endpoint. We relate Boolean TQFTs and topological theories for this cobordism category to sofic systems in  
Theorems~\ref{theorem:regular_sofic} and~\ref{thm_inficirc}. Section~\ref{sec_examples_sofic} contains examples for this correspondence. 

%encapsulates the connection between sofic systems and Boolean topological theories for cobordisms with accumulating defects.

\vspace{0.07in} 

In Section~\ref{sec_one_side} 
we consider a variation of the 
cobordism category, 
where defects only accumulate towards inner endpoints with a particular orientation.  Boolean TQFTs for this category are provided, in particular, by $\omega$-automata, which are various types of automata on semi-infinite words, including B\"uchi and Muller automata. Examples of TQFTs associated to such automata are provided in Section~\ref{sec_examples_automata}.

\vspace{0.07in}

{\bf Acknowledgements.} 
The second and the third author would like to thank retreat-style ``Merging Categorification, Gauge Theory, and Physics" conference at the Swiss Alps in Switzerland for a productive working atmosphere. P.G. was supported by AFRL grant FA865015D1845 (subcontract 669737-1) and ONR grant N00014-16-1-2817, a Vannevar Bush Fellowship held by Dan Koditschek and sponsored by the Basic Research Office of the Assistant Secretary of Defense for Research and Engineering. M.K. was partially supported by NSF grant DMS-2204033 and Simons Collaboration Award 994328 while working on this paper. 

%%%%%%%%%%%%%%%%%%%
%
% Quick review 
%
%%%%%%%%%%%%%%%%%%%

\section{A review of the automata and Boolean TQFT correspondence}\label{sec_review}

%%%%%%%%%%%%%%%%%%%
% Quick review 1D 
%%%%%%%%%%%%%%%%%%%

\subsection{One-dimensional Boolean TQFT and finite state automata}
\label{subsec_qreview}

\quad 

A one-dimensional oriented topological quantum field theory (TQFT) is a symmetric monoidal functor 
\[
\mcF \ : \ \Cob \lra R\mathsf{-mod}
\]
from the category $\Cob$ of oriented one-dimensional cobordisms to the category $R\mathsf{-mod}$ of modules over a commutative ring $R$.
 It is determined by its value $\mcF(+)\cong P$ on a positive-oriented  
$0$-manifold (on a point labelled $+$), and that value can be any finitely-generated projective $R$-module $P$. This module determines the theory up to isomorphism, with 
\[\mcF(-)\, \cong \, P^{\ast}:=\Hom_R(P,R), 
\]   
and \emph{cup} and \emph{cap} morphisms, see Figure~\ref{figure-2}, given by coevaluation and evaluation morphisms between $R$-modules $R$,  $P\otimes P^{\ast}$ and $P^{\ast}\otimes P$. This gives a classification of one-dimensional topological theories. 

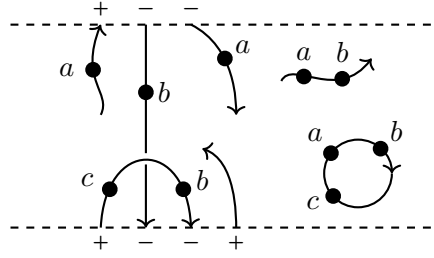
\begin{figure}
    \centering
\begin{tikzpicture}[scale=0.6]
\begin{scope}[shift={(0,0)}]
%\draw[thin,yellow] (0,0) grid (6.5,4);
\draw[thick,dashed] (-1.5,4.5) -- (8.0,4.5);
\draw[thick,dashed] (-1.5,0) -- (8.0,0);

\node at (0.5,-0.35) {$+$};
\node at (1.5,-0.35) {$-$};
\node at (2.5,-0.35) {$-$};
\node at (3.5,-0.35) {$+$};

\node at (0.5,4.85) {$+$};
\node at (1.5,4.85) {$-$};
\node at (2.5,4.85) {$-$};

\draw[thick,->] (0.5,0) .. controls (0.6,2) and (2.4,2) .. (2.5,0);
\draw[thick,fill] (0.85,0.85) arc (0:360:1.5mm);
\node at (0.20,1.05) {$c$};
\draw[thick,fill] (2.48,0.85) arc (0:360:1.5mm);
\node at (2.75,1.05) {$b$};

\draw[thick,<-] (0.5,4.5) .. controls (0,3.25) and (0.75,3.0) .. (0.5,2.5);
\draw[thick,fill] (0.48,3.5) arc (0:360:1.5mm); 
\node at (-0.25,3.5) {$a$};

%
%
% Old version
% \draw[thick,->] (1.5,4) .. controls (1.75,3) and (1.75,2.75) .. (1.75,2.49);
% \draw[thick,fill] (1.80,3.30) arc (0:360:1.5mm); 
% \node at (2.10,3.35) {$b$};
% \draw[thick] (1.75,2.5) .. controls (1.75,2.25) and (1.75,2.00) .. (1.5,1.75);
% \draw[thick] (1.4,1.15) .. controls (1.25,0.85) and (1.25,0.25) .. (1.5,0);
%
%

% New version
\draw[thick] (1.5,4.5) -- (1.5,1.65);
\draw[thick,->] (1.5,1.4) -- (1.5,0);
\draw[thick,fill] (1.65,3.0) arc (0:360:1.5mm); 
\node at (1.9,3.0) {$b$};

\draw[thick,->] (2.5,4.5) .. controls (3.5,4) and (3.5,2.75) .. (3.5,2.5);
\draw[thick,fill] (3.40,3.75) arc (0:360:1.5mm); 
\node at (3.65,4.0) {$a$};

\draw[thick,<-] (2.75,1.75) .. controls (3.5,1.50) and (3.5,0.25) .. (3.5,0);
%\draw[thick,fill] (3.65,0.75) arc (0:360:1.5mm);
%\node at (3.88,1) {$c$};

\draw[thick,->] (4.5,3.25) .. controls (4.75,3.75) and (5.75,2.75) .. (6.50,3.75);
\draw[thick,fill] (5.15,3.35) arc (0:360:1.5mm);
\node at (5,3.85) {$a$};
\draw[thick,fill] (6.00,3.30) arc (0:360:1.5mm);
\node at (5.85,3.90) {$b$};

\draw[thick,<-] (6.95,1.2) arc (0:360:0.75); 
\draw[thick,fill] (6.85,1.75) arc (0:360:1.5mm);
\node at (7.05,2.15) {$b$};
\draw[thick,fill] (5.75,1.65) arc (0:360:1.5mm);
\node at (5.25,2.05) {$a$};
\draw[thick,fill] (5.80,0.70) arc (0:360:1.5mm);
\node at (5.20,0.55) {$c$};

\end{scope}

\end{tikzpicture}
    \caption{A morphism from $(+--+)$ to $(+--)$ in the category $\Cob_{\Sigma}$. Labels $a,b,c$ are elements of $\Sigma$. This morphism has seven outer and five inner boundary points. Objects of $\Cob_{\Sigma}$ are sign sequences. Outer endpoints induce sequences $(+--+)$ and $(+--)$ of orientations at the bottom and top boundary of the cobordism, which are the source and target objects of this morphism.  The morphism has two floating components, one interval and one circle, which define a word $ab$ and a circular word $cab$, respectively.  }
    \label{figure-1}
\end{figure}

To enrich the theory, one introduces defects (dots sliding along components of a cobordism) labelled by elements of a finite set $\Sigma$ and allow one-manifolds to end ``inside'' the cobordism, see Figure~\ref{figure-1}. Endpoints of components of a cobordism that end ``inside'' the cobordism are called \emph{inner}, while those that end on the ``boundary'' of a cobordism are called \emph{outer}. The resulting category $\Cob_{\Sigma}$ is symmetric monoidal and isomorphism classes of TQFTs for it, that is, symmetric monoidal functors 
\[
\mcF \ : \ \Cob_{\Sigma} \lra R\mathsf{-mod}, 
\]
are classified by isomorphism classes of data $(P,\{m_a\}_{a\in \Sigma},v_0,v^{\ast})$, where $P=\mcF(+)$ is a finitely-generated projective $P$-module, $m_a:P\lra P$ is an endomorphism of $P$ for each $a\in \Sigma$, $v_0\in P$ and $v^{\ast}:P\lra R$ is an $R$-module map~\cite{GIKKL23,IK_TQFTjourney23}.

Furthermore, one can replace a commutative ring $R$ by a commutative semiring, with the same classification of TQFTs as for rings. A projective finitely-generated $R$-semimodule $P$ is then defined as a retract of a free semimodule, with semimodule maps $P\stackrel{\iota}{\lra} R^n\stackrel{p}{\lra} P$ such that $p\circ \iota=\id_P$.

All of that is explained in~\cite{GIKKL23}, which then specializes to the Boolean semiring $R=\Bool=\{0,1|1+1=1\}$ and considers TQFTs (symmetric monoidal functors) 
\begin{equation}\label{eq_TQFT_B}
\mcF \ : \ \Cob_{\Sigma}\lra \Bool\mathsf{-fmod}
\end{equation}
from this category of $\Sigma$-decorated one-dimensional cobordisms to  the category of free $\Bool$-modules. The state space $\mcF(+)\cong \Bool^n$ of a $+$ point is necessarily of finite rank, and the free $\Bool$-module $\mcF(+)$ has a unique basis, up to permutations of terms. Writing $\mcF(+)\cong \Bool Q$, for a finite set $Q$, allows to associate to the functor $\mcF$ a nondeterministic finite state automaton with $\mcF$ on the set of states $Q$ as follows.  

Note that  the powerset $\mathcal{P}(Q)$ with the union of sets as the addition operation can be identified with the $\Bool$-semimodule $\Bool Q$ by taking a subset $S\subset Q$ to be the sum $\sum_{s\in S}s \in \Bool Q$. 

  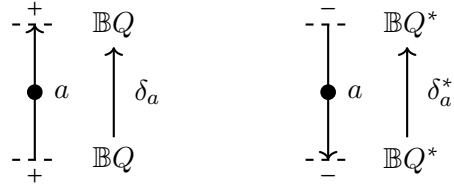
\begin{figure}
    \centering
\begin{tikzpicture}[scale=0.6]
\begin{scope}[shift={(0,0)}]
%\draw[thin,yellow] (0,0) grid (4,4);
\node at (0.5, 3.35) {$+$};
\node at (0.5,-0.35) {$+$};
\draw[thick,dashed] (0,3) -- (1,3);
\draw[thick,dashed] (0,0) -- (1,0);
\draw[thick,->] (0.5,0) -- (0.5,3); 
\draw[thick,fill] (0.65,1.5) arc (0:360:1.5mm);
\node at (1.1,1.5) {$a$};

\node at (2.25,3) {$\Bool Q$};
\draw[thick,<-] (2.25,2.5) -- (2.25,0.5);
\node at (3,1.5) {$\delta_a$};
\node at (2.25,0) {$\Bool Q$};
\end{scope}

\begin{scope}[shift={(6.5,0)}]
%\draw[thin,yellow] (0,0) grid (4,4);
\node at (0.5, 3.35) {$-$};
\node at (0.5,-0.35) {$-$};
\draw[thick,dashed] (0,3) -- (1,3);
\draw[thick,dashed] (0,0) -- (1,0);
\draw[thick,<-] (0.5,0) -- (0.5,3); 
\draw[thick,fill] (0.65,1.5) arc (0:360:1.5mm);
\node at (1.1,1.5) {$a$};

\node at (2.35,3) {$\Bool Q^*$};
\draw[thick,<-] (2.25,2.5) -- (2.25,0.5);
\node at (3,1.5) {$\delta_a^*$};
\node at (2.35,0) {$\Bool Q^*$};
\end{scope}

\end{tikzpicture}
    \caption{To a labelled dot on an upward-oriented interval functor $\mcF$ associates endomorphism $\delta_a$ of $\Bool Q$. These endomorphisms, over all $a\in \Sigma$, encode the oriented labelled graph of the automaton $(Q)$. To a dot on a downward-oriented interval functor $\mcF$  associates the dual operator $\delta_a^{\ast}$ on $\Bool Q^{\ast}$.}
    \label{figure-3}
\end{figure} 

Functor $\mcF$ associates an endomorphism $\delta_a$ of the free $\Bool$-module $\Bool Q$ to a point defect labelled $a$ on an upward-oriented line, see Figure~\ref{figure-3}. This endomorphism is determined by its values $\delta_a(q)\subset Q$  on the basis elements $q\in Q$, via the above isomorphism $\Bool Q\cong \mathcal{P}(Q)$. To endomorphisms $\delta_a,$ $a\in \Sigma$, associate a decorated oriented graph with $Q$ as the set of vertices and an oriented edge labelled $a$ out of $q$ and into $q'$ if and only if $q'\in \delta_a(q)$, over all $a,q$. 

   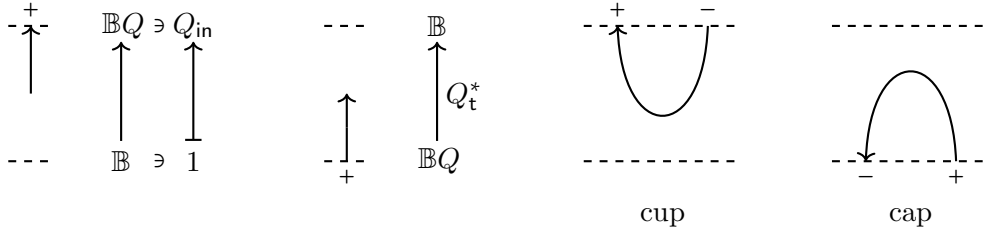
\begin{figure}
    \centering
\begin{tikzpicture}[scale=0.6]

    \begin{scope}[shift={(13,0)}]
    %\draw[thin,yellow] (0,0) grid (4,4);
    \draw[thick,dashed] (-0.25,3) -- (3.25,3);
    \draw[thick,dashed] (-0.25,0) -- (3.25,0);
    
    \node at (0.5,3.35) {$+$};
    \node at (2.5,3.35) {$-$};
    \draw[thick,<-] (0.5,3) .. controls (0.6,0.35) and (2.4,0.35) .. (2.5,3);
    \node at (1.5,-1.25) {cup};
    \end{scope}  

    \begin{scope}[shift={(18.5,0)}]
    %\draw[thin,yellow] (0,0) grid (4,4);
    \draw[thick,dashed] (-0.25,3) -- (3.25,3);
    \draw[thick,dashed] (-0.25,0) -- (3.25,0);
    
    \node at (0.5,-0.35) {$-$};
    \node at (2.5,-0.35) {$+$};
    
    \draw[thick,<-] (0.5,0) .. controls (0.6,2.65) and (2.4,2.65) .. (2.5,0);
    \node at (1.5,-1.25) {cap};
    \end{scope}

\begin{scope}[shift={(0,0)}]
%\draw[thin,yellow] (0,0) grid (4,4);
\draw[thick,dashed] (0,3) -- (1,3);
\draw[thick,dashed] (0,0) -- (1,0);
\node at (0.5,3.35) {$+$};
\draw[thick,->] (0.5,1.5) -- (0.5,3);
\end{scope}

\begin{scope}[shift={(0.8,0)}]
\node at (1.75,3) {$\Bool Q$};
\draw[thick,->] (1.7,0.45) -- (1.7,2.65);
\node at (1.7,0) {$\Bool$};

\node at (3.30,3) {$Q_{\init}$};
\node at (2.55,3) {$\ni$};
\draw[thick,|->] (3.30,0.45) -- (3.30,2.65);
\node at (2.55,0) {$\ni$};
\node at (3.30,0) {$1$};
\end{scope}

\begin{scope}[shift={(7,0)}]
%\draw[thin,yellow] (0,0) grid (4,4);
\draw[thick,dashed] (0,3) -- (1,3);
\draw[thick,dashed] (0,0) -- (1,0);
\node at (0.5,-0.35) {$+$};
\draw[thick] (0.5,0) -- (0.5,0.75);
\draw[thick,->] (0.5,0.75) -- (0.5,1.5);
\end{scope}

\begin{scope}[shift={(7.8,0)}]
\node at (1.7,3) {$\Bool$};
\draw[thick,->] (1.7,0.45) -- (1.7,2.65);
\node at (1.75,0) {$\Bool Q$};
\node at (2.30,1.45) {$Q_{\t}^{\ast}$};
\end{scope}

\end{tikzpicture}
    \caption{Left: maps assigned to the half-intervals with a $+$ boundary points. Right: ``cup'' and ``cap'' cobordisms.
    }
    \label{figure-2}
\end{figure} 

To the negative point $-$ functor $\mcF$ associates a free semimodule which can be canonically identified with the dual of $\Bool Q$: 
\[\mcF(-)\cong (\Bool Q)^{\ast}= \Bool Q^{\ast},
\]
where $Q^{\ast}$ is just a copy of $Q$ with elements denoted $q^{\ast}$, $q\in Q$. 
The identification is such that to ``cup'' and ``cap'' cobordisms shown in Figure~\ref{figure-2} on the right functor $\mcF$ associates the standard duality morphisms: \emph{coevaluation} morphism to the cup and \emph{evaluation} morphism to the cap 

\begin{equation}
\label{align_1}
\begin{aligned}[c]
    &\mcF(+-) \cong \Bool Q\otimes \Bool Q^{\ast}, \\
    &\mathsf{coev} : \ \Bool \lra \Bool Q \otimes \Bool Q^{\ast}, \\ 
    &\mathsf{coev}(1) = \sum_{q\in Q}q\otimes q^{\ast}, \\
\end{aligned}
\hspace{2cm}
\begin{aligned}[c]
    &\mcF(-+) = \Bool Q^{\ast}\otimes \Bool Q, \\
    &\mathsf{ev} : \ \Bool Q^{\ast} \otimes \Bool Q\lra \Bool, \\  
    &\mathsf{ev}(q^{\ast}\otimes q') = \delta_{q,q'}. \\ 
\end{aligned}
\end{equation}

Consider two half-intervals, shown in Figure~\ref{figure-2} on the left. To the half-interval oriented ``in'' at its inner endpoint functor $\mcF$ associates a semimodule 
homomorphism $\Bool\lra \Bool Q$, determined by the image of $1$, which is a subset $Q_{\init}\subset Q$, the set of initial states of the automaton associated to $\mcF$. To the ``out'' oriented half-interval functor $\mcF$ associates a linear map $\Bool Q\lra \Bool$. Denote by $Q_{\t}$ the set of states $q\in Q$ taken to $1$ by this map. This is the subset of \emph{terminal} or \emph{accepting} states of the automaton. 

At this point the construction is complete. The nondeterministic finite state automaton associated to $\mcF$ and denoted $(Q)$ has $Q$ as the set of states, where $\mcF(+)\cong \Bool Q$, with the transition function $\delta: \Sigma\times Q\lra \mathcal{P}(Q)$ determined by applying $\mcF$ near defect points of cobordisms, and sets of initial states $Q_{\init}$ and accepting states $Q_{\t}$ encoded by applying $\mcF$ to half-interval cobordisms.

This results in a bijection between isomorphism classes of Boolean TQFTs as in \eqref{eq_TQFT_B} and isomorphism classes of nondeterministic finite state automata with alphabet $\Sigma$. 

Regular language $L_{(Q)}$ associated with an automaton $(Q)$ describes evaluation of \emph{floating intervals}, which are one-manifolds homeomorphic to an interval with both endpoints inner. Figure~\ref{figure-1} contains one floating interval, decorated by word $ab$. 
Additionally $(Q)$ determines a regular circular language $L_{\circ,(Q)}$ which consists of circular words $\omega$ for which there exists an oriented loop in $(Q)$, see~\cite{GIKKL23,IK_TQFTjourney23}. Equivalently, a word $\omega$ on a circle evaluates to the trace of the operator on $\Bool Q$ that it induces.  

\begin{remark}
One possible generalization of the above correspondence is that to \emph{quasi-automata}, as defined in~\cite{GIKKL23}, by considering symmetric monoidal functors similar to those in \eqref{eq_TQFT_B} but with the target the entire category of $\Bool$-modules:  
\begin{equation}\label{eq_TQFT_Ball}
\mcF \ : \ \Cob_{\Sigma}\lra \Bool\dmod.
\end{equation}
Then $\mcF(+)$ is a projective finite $\Bool$-module which is not necessarily free. Any such module is isomorphic to the module $\mathcal{U}(Y)$ of open sets of a finite topological space $Y$, with the sum given by the union of open sets and $0$ given by the empty set. An element $a\in \Sigma$ must act on  $\mathcal{U}(Y)\cong \mcF(+)$ taking open sets to open sets preserving the union operation and the empty set. The set of initial states is replaced by an element $Q_{\init}\in \mcF(+)$ and the set of terminal (accepting) states is replaced by a $\Bool$-module map $Q_{\t}:\mcF(+)\lra \Bool$. 

Isomorphism classes of such \emph{quasi-automata} are in a bijection with isomorphism classes of symmetric monoidal functors \eqref{eq_TQFT_Ball}. 
\end{remark} 

%%%%%%%%%%%%%%%%%%%
% Quick review 1D 
%%%%%%%%%%%%%%%%%%%

\subsection{Boolean topological theories in one dimension}
\label{subsec_treview}

An earlier approach in~\cite{IK-top-automata} considers Boolean-valued lax TQFTs which may fail the tensor product axiom. One starts with a Boolean-valued evaluation of closed cobodisms in $\Cob_{\Sigma}$, that is, floating intervals and circles decorated by words in $\Sigma^{\ast}$. 
This requires a pair of Boolean functions $\alpha=(\alphai,\alpha_{\circ})$:
\begin{equation}\label{eq_pair} 
    \alphai: \Sigma^{\ast}\lra \Bool, \hspace{1cm}
    \alpha_{\circ}: \Sigma^{\circ} \lra \Bool, 
\end{equation}
where $\Sigma^{\ast}$ is the set of finite words in the alphabet $\Sigma$. 
Function $\alphai$ on words $\omega\in \Sigma^{\ast}$ determines how to evaluate floating intervals decorated by $\omega$. Function $\alpha_{\circ}$ is rotation-invariant, $\alphac(\omega_1\omega_2)=\alphac(\omega_2\omega_1)$, where $\Sigma^{\circ}$ above is the set of \emph{circular} words, that is, the quotient of $\Sigma^{\ast}$ by the rotation equivalence relation.  

Given $\alpha$ as above, pick a sign sequence $\varepsilon$ and consider the free $\Bool$-module $\Fr(\varepsilon)$ with the basis the set of cobordisms from $\emptyset$ to $\varepsilon$ in $\Cob_{\Sigma}$ without floating components. This module comes with a $\Bool$-bilinear form 
\[
(\:\:, \:\: )_{\varepsilon} \ : \ \Fr(\varepsilon)\otimes \Fr(\varepsilon) \lra \Bool, 
\hspace{1cm}
(x,y)_{\varepsilon} := \alpha(\overline{x}y). 
\]
Here $x,y$ are cobordisms from the empty 0-manifold $\emptyset_0$ to $\varepsilon$. Reflect $x$ and reverse its orientation to view it as a cobordism from $\varepsilon$ to $\emptyset_0$. The composition $\overline{x}y$ is an endomorphism of the empty $0$-manifold (a floating cobordism). It is a union of floating intervals and circles and can be evaluated via $\alpha$, giving the bilinear form $(\:\:,\:\:)_{\varepsilon}$. 

Define the state space 
\[A(\varepsilon)=A_{\alpha}(\varepsilon):=\Fr(\varepsilon)/\mathsf{ker}((\:\:,\:\:)_{\varepsilon})
\]
as the quotient of the free module $\Fr(\varepsilon)$ by the kernel of the form. Since we are working over a semiring, the quotient must be understood in the set-theoretic sense. Two finite $\Bool$-linear combinations $\sum_i x_i$, $\sum_{i'}x_{i'}$ are equal in $A(\varepsilon)$ if for any cobordism $y$ from $\emptyset_0$ to $\varepsilon$ we have 
\[
\sum_i \alpha(\overline{y}x_i) = \sum_{i'}\alpha(\overline{y}x_{i'}).
\]
The assignment $\varepsilon \longmapsto A(\varepsilon)$ extends to a functor 
\[
A=A_{\alpha} \ : \ \Cob_{\Sigma} \lra \Bool\dmod
\]
that to a sign sequence $\varepsilon$ assigns $A(\varepsilon)$ and to a cobordism $z$ from 
$\varepsilon$ to $\varepsilon'$ assigns the $\Bool$-module map $A(z):A(\varepsilon)\lra A(\varepsilon')$ given by composing cobordisms from $\emptyset_0$ to $\varepsilon$ with $z$. 

\begin{theorem}\cite{IK-top-automata} The following conditions are equivalent: 
\begin{enumerate}
    \item Languages $\alphai^{-1}(1), \alphac^{-1}(1)\subset \Sigma^{\ast}$ are regular, 
    \item State spaces $A(\varepsilon)$ are finite for all sequences $\varepsilon$, 
    \item State spaces $A(+),A(+-)$ are finite. 
\end{enumerate}
\end{theorem}

If one of these equivalent conditions holds, we say that $\alpha$ is a \emph{rational} evaluation. Recall that a language $L\subset \Sigma^{\ast}$ is called \emph{regular} is it is given by a regular expression, equivalently, if it can be described by a finite state automaton. 

Assume from now on that $\alpha$ is rational. Then $A(+)$ and $A(-)$ are dual semimodules, of the same cardinality, with the duality given by the pairing that glues a pair of cobordisms from $\emptyset_0$ into $+$ and $-$, respectively, into a floating cobordism and evaluates it via $\alpha$. In particular, $|A(-)|=|A(+)|$. 

State space $A(+)$ depends only on $\alphai$. 
For some rational evaluations $\alphai$ this state space is a \emph{projective} $\Bool$-semimodule. When this happen there is a canonical circular regular language $\alpha_{\circ}\subset \Sigma^{\circ}$ with $\omega\in \alpha_{\circ}$ if and only if $\tr_{A(+)}(\omega)=1$, see~\cite[Section 4.4]{IK-top-automata}. The topological theory for such a pair $\alpha=(\alphai,\alphac)$ is a TQFT, with $A(\varepsilon)\cong A(\varepsilon_1)\otimes \cdots \otimes A(\varepsilon_n)$ for a sign sequence $\varepsilon=(\varepsilon_1,\ldots, \varepsilon_n)$. In this case there is a decomposition of the identity, where the cup cobordism from $\emptyset_0$ to $+-$, see Figure~\ref{figure-2}, is in the image of the natural homomorphism $A(+)\otimes A(-)\lra A(+-)$ and the latter map is an isomorphism of $\Bool$-modules.   

Note that for general rational pairs $\alpha$ the map $A(+)\otimes A(-)\lra A(+-)$ is not an isomorphism and the cup cobordism is not in its image. 

\vspace{0.1in} 

 Cobordisms which are morphisms in $\Cob_{\Sigma}$ carry finitely many defects, which are dots decorated by elements of $\Sigma$. In the next section we allow defects to accumulate towards inner boundary points and, in fact, require to have such an infinite countable chain of defects near each inner boundary point. In particular, any half-interval contains an inner and an outer boundary point and necessarily carries infinitely-many defects.

%%%%%%%%%%%%%%%%%%%%
%
% sofic systems
%
%%%%%%%%%%%%%%%%%%%%

\section{Categories from sofic systems}
\label{section:sofic-systems}

\subsection{Summary of key notations used in this section}

The summary is provided here for convenience, with definitions and discussions in the follow-up subsections. 

\begin{itemize}
\item $\SigmaZ$ is the set of doubly-infinite sequences, $\Sigmal$ is the set left-infinite sequences, $\Sigmar$ -- the set of right-infinite sequences, $\Sigma^{\ast}$ --the set of finite sequences, $\Sigma^{\circ}$ -- the set of finite sequences up to rotation.
\item $\alphaZ:\SigmaZ\lra \Bool$ is a $\Bool$-valued evaluation of doubly-infinite words, with $X:=\alphaZ^{-1}(1)$ a closed subset of $\SigmaZ$, also denoted $L(\alphaZ)$ and called \emph{a shift}. 
\item $W\subset \Sigma^{\ast}$ the language of prohibited finite words, either initially given or determined by  $\alphaZ$,  see \eqref{eq_Wprime}. It consists of finite words that do not appear as subwords of any infinite word in $X$.
\item A shift can also be defined starting with a subset $W\subset \Sigma^{\ast}$ by $X:=W^{\perp}, X\subset \SigmaZ$. Here $W^{\perp}\subset \SigmaZ$ is the set of infinite words that do not contain any subword in $W$, 
\item 
A shift is sofic if $W$ is a regular language, where we start with $W\subset \Sigma^{\ast}$, then form $X=W^{\perp}\subset \SigmaZ$. 
\item $W^{\perp}_{\f}$ is the set of finite words that do not contain any subword in $W$, see \eqref{eq-Wperp}. 
\item Some of these notations appear in Table~\ref{tab:my_label}. 
\end{itemize}

%%%%%%%%%%%%%%%%
% infinite words 
%%%%%%%%%%%%%%%%

\subsection{Infinite words and shifts}
\label{subsec_infinite}
For a finite set $\Sigma$ of letters, denote by 
\[\SigmaZ=\{ \cdots a_{i-1}a_ia_{i+1}\cdots \mid \,  a_i \in \Sigma\},
\]
the set of doubly-infinite sequences of elements of $\Sigma$. Set $\SigmaZ$ carries the direct product topology, where $\Sigma$ is given the discrete topology.  The infinite cyclic group $\Z$ acts continuously on $\SigmaZ$ by shifting  each term of a sequence $n$ steps to the left, for $n\in \Z$ (for negative $n$, this means shifting each term by $-n$ steps to the right). 
The action of $\Z$ on $\Sigma^{\Z}$ is called the \emph{shift action}, and shift by one step to the left is also known as the \emph{shift transformation}.
The orbits under this action are elements of $\SigmaZ/\Z$.

\vspace{0.07in} 

In symbolic dynamics, $\SigmaZ$ is called the \emph{full shift}. A shift $X$ is a closed $\Z$-invariant subspace of $\SigmaZ$ \cite{Kit98, lind2021introduction,Williams04, adler1998symbolic}.

We call a $\Z$-invariant subset of $\SigmaZ$ an \emph{infinite language} (a language of infinite words). An infinite language is called \emph{closed} if the corresponding subset is closed in $\SigmaZ$. Closed infinite languages are in a bijection with shift spaces $X$. 

Subsets of a set $Y$ are in a bijection with maps $Y\lra \{0,1\}$ into a 2-element set. 
To a Boolean evaluation of infinite words $\alphaZ:\SigmaZ/\Z\lra \Bool=\{0,1\}$ we associate the infinite language (a set of infinite words) 
 \begin{equation}
 L(\alphaZ)=\{ \omega \in \SigmaZ \mid \alpha(\omega)=1\}.
 \end{equation}
 Here and later we view $\alphaZ$ as both a function on $\Sigmalr/\Z$ and a $\Z$-invariant function on $\Sigmalr$. 

A finite subword of an infinite word $\omega=\cdots a_{-1}a_0a_1\cdots $ is any finite sequence $a_ia_{i+1}\cdots a_j$, $i\le j+1$, of consecutive letters in $\omega$ (case $i=j+1$ corresponds to the empty subword $\emptyset$). 

\vspace{0.07in}
  
Any shift $X$ can be described as follows. Choose a subset $W\subset \Sigma^{\ast}$ of the set of finite words. 
Consider the infinite language $W^{\perp}\subset \SigmaZ$ which consists of infinite words that do not contain any finite word from $W$ as a subword. Thus, the complement $\SigmaZ\setminus W^{\perp}$ consists of words of the form $\omega_0\omega_1\omega_2$, where $\omega_0$ and $\omega_2$ are left-infinite, respectively right-infinite words, and $\omega_1\in W$, so that 
\begin{equation}\label{eq_perp_W}
    W^{\perp} \ = \SigmaZ \setminus \Sigmal W \Sigmar ,
\end{equation}
see formula \eqref{eq_sigmas} below for these notations. 
It is convenient to also associate to $W$ the finite language 
\begin{equation}\label{eq-Wperp}
    W^{\perp}_{\f} \ := \Sigma^{\ast} \setminus \Sigma^{\ast}W\Sigma^{\ast}
\end{equation}
consisting of all finite words that do not contain a word in $W$. An infinite word is in $W^{\perp}$ if and only if any finite subword of it is in $W^{\perp}_{\f}$. Note that both $W^{\perp}$ and $W^{\perp}_{\f}$ only depend on the closure $\Sigma^{\ast}W\Sigma^{\ast}$, so that different $W$'s can give rise to the same languages $W^{\perp}$ and $W^{\perp}_{\f}$.

For any shift $X$ there exists a subset $W\subset \Sigma^{\ast}$ so that $X=W^{\perp}$. 
Languages $W,W^{\perp}$ and $W^{\perp}_{\f}$ are summarized in Table~\ref{tab:my_label}. 

\vspace{0.07in} 
 
\begin{table}[hbt!]
    \centering
    \begin{tabular}{|c|c|}
    \hline 
     Notation & Description of language \\ 
    \hline 
        $W$ &  $\text{finite subwords  that  do not  appear   in  words  in}$ $X$  \\
    \hline 
        $X=W^\perp $ & $\text{{\it infinite}  words  with no  subwords in}$  $W$ \\
    \hline 
        $W^{\perp}_{\f}$ &  $\text{{\it finite}  words  with no  subwords in}$  $W$ \\
    \hline         
    \end{tabular} 
    \caption{Languages associated with $X$.}
    \label{tab:my_label}
\end{table} 
\begin{itemize}
\item A shift is called \emph{finite} if $X=W^{\perp}$ for some finite subset $W$. 
\item A shift is called \emph{sofic} if $X=W^{\perp}$ for some regular subset $W$, where $W\subset \Sigma^{\ast}$ is viewed as a language \cite{weiss1973subshifts,Kit98,lind2021introduction}. A regular language is one accepted by a finite state automaton. 
\end{itemize} 
A finite shift is necessarily sofic. An example of a sofic shift which is not finite is given by the language $W=a(bb)^{\ast}ba$, with $\Sigma=\{a,b\}$. Shift $X$ then consists of all infinite words which do not contain any odd-length finite runs of consecutive $b$'s. 

\begin{remark}
    Many naturally occurring invertible dynamical systems on manifolds and measure spaces can be encoded or approximated by the full shift or its finite and sofic subshifts, see~\cite{adler1998symbolic,Kit98,Williams04}, \cite[Section 6.5]{lind2021introduction}, \cite[Section 5]{GH83}.

\end{remark}

Denote by \begin{equation}
   \label{eq_sigmas} \Sigmal = \{ \cdots a_{-3}a_{-2}a_{-1} |\,  a_{-i}\in \Sigma \} , 
     \ \ \ \ 
     \Sigmar = \{ a_0a_1a_2\cdots |\,  a_i \in \Sigma\} 
\end{equation}
the sets of left-infinite and right-infinite sequences of elements of $\Sigma$. Concatenation gives a natural bijection 
$\Sigmal \times \Sigmar \lra \Sigmalr$. We use $\t$ for \emph{tail} and $\h$ for \emph{head}. 

Recall that $\Sigma^{\ast}$ denotes the set of finite words in the alphabet $\Sigma$ and $\Sigma^{\circ}=\Sigma^{\ast}/\sim$ is the set of circular words (the quotient of $\Sigma^{\ast}$ by the equivalence relation $\omega_1\omega_2\sim \omega_2\omega_1$ for $\omega_1\omega_2 \in \Sigma^{\ast}$). Concatenation gives maps 
\[ \Sigmal \times \Sigma^{\ast}\lra \Sigmal, \ \ \Sigma^{\ast}\times \Sigmar \lra \Sigmar.
\] 
To summarize, sequence notations are shown in Table~\ref{tab:sofic-summary}:

\begin{table}[h!]
    \centering
    \begin{tabular}{|c|c|c|c|c|}
        \hline 
         $\SigmaZ$ & $\Sigmal$ & $\Sigmar$ & $\Sigmaa$ &  $\Sigmac$ \\
         \hline 
         infinite words & left-infinite words & right-infinite words & finite words & circular words \\
         \hline 
         $\cdots a_{-1}a_0 a_1 \cdots$ & $\cdots a_{-2}a_{-1}a_0 $ & $a_1a_2a_3\cdots $ & $a_1\cdots a_n$ & $\omega_1\omega_2\sim \omega_2\omega_1$ \\
         \hline 
    \end{tabular}
    \caption{
    Types of infinite and finite words, with $\omega_i\in \Sigma^{\ast}$. }
    \label{tab:sofic-summary}
\end{table}
Sets $\Sigmaa,\Sigmac$ are countably infinite, the other three sets are uncountable.

For $a\in \Sigma$, denote by $a^{\t}$ or $a^{-\infty}\in \Sigmal$ the left-infinite word $\cdots aaa$, by $a^{\h}$ or  $a^{+\infty}\in \Sigmar$ the right-infinite word $aaa\cdots $, and by $a^{\Z}$ the infinite word $\cdots aaa \cdots \in \Sigma^{\Z}$.

\vspace{0.1in} 

In the description above, one starts with a finite or regular language $W\subset \Sigma^{\ast}$ and assigns to it a finite or sofic shift $X$ and the associated infinite evaluation function $\alphaZ:\SigmaZ/Z\lra \Bool$ with $\alphaZ^{-1}(1)=X$. We also denote $X$ by $L=L(\alphaZ)$.

\vspace{0.05in} 

Going in the opposite direction, 
to a \emph{closed} evaluation $\alphaZ:\Sigmalr/\Z\lra \Bool$ associate the set of finite words \begin{equation}\label{eq_Wprime}
W:=W(\alphaZ)\subset \Sigma^{\ast}
\end{equation}
with $\omega\in W(\alphaZ)$ if and only if $\omega$ is not a subword of any infinite word $\omega\in L(\alphaZ)$. In other words, $\omega\in W$ if and only if for any words $\omega_-\in\Sigmal,\omega_+\in \Sigmar$ the equation $\alphaZ(\omega_-\omega\omega_+)=0$ holds. 
Thus, language $W=W(\alphaZ)$ consists of words that are avoided by infinite words in $L(\alphaZ)$.
Language $W$ is prefix- and suffix-closed, in the sense that $\Sigma^{\ast}W\Sigma^{\ast}=W$. Then $L=W^{\perp}$ can be reconstructed from $W$. Indeed, the condition that $\alphaZ$ is closed says that $\alphaZ^{-1}(0)$ is open in $\SigmaZ$. Then $\alphaZ^{-1}(0)$ is the union of base open sets, which correspond to finite words, exactly those avoided by infinite words in $L=\alphaZ^{-1}(1)$.

%%%%%%%%%%%%%%%%%%%
% cobordisms accumulating defects 
%%%%%%%%%%%%%%%%%%%

\subsection{Cobordisms with accumulating defects}
\label{subsec_accumulating}

We next introduce the category $\Cob^{\infty}_{\Sigma}$ of cobordisms with $\Sigma$-defects where defects accumulate towards each floating (inner) endpoint of a  cobordism. 
   
Category $\Cob^{\infty}_{\Sigma}$ is defined similarly to $\Cob_{\Sigma}$, see Section~\ref{subsec_qreview}. Its objects are oriented 0-manifolds, parameterized by finite sequences of signs $\varepsilon=(\varepsilon_1,\ldots, \varepsilon_n)$. A morphism from $\varepsilon$ to $\varepsilon'$ is an oriented 1-manifold $U$ with \emph{outer boundary} $\partial_o U = \varepsilon'\cup -\varepsilon$. Manifold $U$ is allowed to have ``inner'' boundary points, so its boundary decomposes $\partial U = \partial_o U\sqcup \partial_i U$ into the disjoint union of ``outer'' boundary points (elements of $\partial_o U$) and ``inner'' boundary points (elements of $\partial_i U$). Same boundary decomposition holds for morphisms in $\Cob^{\infty}_{\Sigma}$. The difference is that at each floating endpoint of a morphism $U$ in $\Cob_{\Sigma}^{\infty}$ we require an accumulation of defects, a countably infinite sequence of defects approaching the floating endpoint.   
  
Depending on the orientation near an inner (floating) endpoint, we parametrize possible sequences of defects by elements of $\Sigmal$, respectively $\Sigmar$, see Figure~\ref{sofic-0001}. Such parametrization requires a \emph{cut-off} position between two consecutive dots, where the semi-infinite sequence of dots terminates. If an interval $I$ with a floating endpoint has the other endpoint on the outer boundary, the cut-off point can be chosen to be the outer boundary point. Such semi-floating intervals necessarily carry accumulation of defects towards their floating endpoints and, depending on orientation, can be parametrized by elements of $\Sigmal$ or $\Sigmar$, respectively, see Figure~\ref{sofic-0001}.  

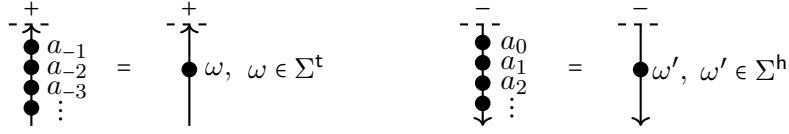
\begin{figure}
    \centering
\begin{tikzpicture}[scale=0.6]

\begin{scope}[shift={(-0.5,0)}]
%\draw[thin,yellow] (0,0) grid (4,4);
\node at (0.5,2.35) {$+$};
\draw[thick,dashed] (0,2) -- (1,2);
\draw[thick,<-] (0.5,2) -- (0.5,-0.25);

\draw[thick,fill] (0.65,1.5) arc (0:360:1.5mm);
\node at (1.3,1.45) {$a_{-1}$};
\draw[thick,fill] (0.65,1.05) arc (0:360:1.5mm);
\node at (1.3,1.0) {$a_{-2}$};
\draw[thick,fill] (0.65,0.6) arc (0:360:1.5mm);
\node at (1.3,0.55) {$a_{-3}$};
\draw[thick,fill] (0.65,0.15) arc (0:360:1.5mm);
\node at (1.15,0.10) {$\vdots$};

\node at  (2.55,1) {$=$};
\end{scope}

\begin{scope}[shift={(3,0)}]
\node at (0.5,2.35) {$+$};
\draw[thick,dashed] (0,2) -- (1,2);
\draw[thick,<-] (0.5,2) -- (0.5,-0.25);

\draw[thick,fill] (0.65,1.0) arc (0:360:1.5mm);
\node at (1.15,0.95) {$\omega$,};
\node at (2.65,1.07) {$\omega \in \Sigmal $}; 
\end{scope}

\begin{scope}[shift={(9.5,0)}]
\node at (0.5,2.35) {$-$};
\draw[thick,dashed] (0,2) -- (1,2);
\draw[thick,->] (0.5,2) -- (0.5,-0.25);

\draw[thick,fill] (0.65,1.6) arc (0:360:1.5mm);
\node at (1.2,1.55) {$a_{0}$};
\draw[thick,fill] (0.65,1.15) arc (0:360:1.5mm);
\node at (1.2,1.1) {$a_{1}$};
\draw[thick,fill] (0.65,0.7) arc (0:360:1.5mm);
\node at (1.2,0.65) {$a_{2}$};
\draw[thick,fill] (0.65,0.25) arc (0:360:1.5mm);
\node at (1.15,0.15) {$\vdots$};

\node at  (2.55,1) {$=$};
\end{scope}

\begin{scope}[shift={(13,0)}]
\node at (0.5,2.35) {$-$};
\draw[thick,dashed] (0,2) -- (1,2);
\draw[thick,->] (0.5,2) -- (0.5,-0.25);

\draw[thick,fill] (0.65,1.0) arc (0:360:1.5mm);
\node at (1.15,0.97) {$\omega'$,};
\node at (2.8,1.05) {$\omega' \in \Sigmar $}; 
\end{scope}

\end{tikzpicture}
    \caption{Homeomorphism (or diffeomorphism) classes of semifloating intervals are in a bijection with semiinfinite words $\omega=\cdots a_{-2}a_{-1}\in \Sigmal$ (left-infinite) and $\omega'=a_0a_1\cdots\in \Sigmar$ (right-infinite).  }
    \label{sofic-0001}
\end{figure}

An arc in $U$ without floating endpoints (so both of its endpoints are at the outer boundary) carries a finite set of defects, possibly the empty set, see the top row in Figure~\ref{sofic-0002}.
Floating components in $U$ are of two types, see Figure~\ref{sofic-0002}, bottom row:  
\begin{itemize}
    \item A circle with finitely many defects, decorated by an element of $\Sigma^{\circ}$.
    \item An interval with defects accumulating towards both endpoints. Such an interval is described by a sequence in $\Sigmalr$ up to an overall shift, thus an element of $\Sigmalr/\Z$. 
\end{itemize}

\begin{figure}
    \centering
\begin{tikzpicture}[scale=0.6]

\begin{scope}[shift={(5,0)}]
%\draw[thin,yellow] (0,0) grid (5,5);

\node at (0.5,3.35) {$-$};
\node at (2.5,3.35) {$+$};
\draw[thick,dashed] (0,3) -- (3,3);
\draw[thick,->] (0.5,3) .. controls (0.6,0.5) and (2.4,0.5) .. (2.5,3);
\draw[thick,fill] (0.7,2.30) arc (0:360:1.5mm);
\node at (0,2.1) {$a_1$};
\draw[thick,fill] (1.00,1.50) arc (0:360:1.5mm);
\node at (0.4,1.15) {$a_2$};
\draw[thick,fill] (1.65,1.1) arc (0:360:1.5mm);
\node at (1.5,0.6) {$a_3$};
\draw[thick,fill] (2.4,1.8) arc (0:360:1.5mm);
\node at (2.6,1.4) {$\rotatebox[origin=c]{53}{$\cdots$}$};
\draw[thick,fill] (2.65,2.5) arc (0:360:1.5mm);
\node at (3.10,2.5) {$a_n$};

\node at (4.2,1.75) {$=$};
\end{scope}

\begin{scope}[shift={(10,0)}]
%\draw[thin,yellow] (0,0) grid (5,5);

\node at (0.5,3.35) {$-$};
\node at (2.5,3.35) {$+$};
\draw[thick,dashed] (0,3) -- (3,3);
\draw[thick,->] (0.5,3) .. controls (0.6,0.5) and (2.4,0.5) .. (2.5,3);

\draw[thick,fill] (1.65,1.1) arc (0:360:1.5mm);
\node at (1.7,0.5) {$\omega'$};
\node at (5.00,1.75) {$\omega'\in \Sigma^*$};
\end{scope}

\begin{scope}[shift={(-0.25,-2.25)}]
%\draw[thin,yellow] (0,0) grid (5,4);
\draw[thick,<-] (0,1) arc (135:-225:1cm);

\draw[thick,fill] (-0.13,0.3) arc (0:360:1.5mm);
\node at (-0.9,0.25) {$a_1$};
\draw[thick,fill] (0.24,-0.5) arc (0:360:1.5mm);
\node at (-0.4,-0.9) {$a_2$};
\draw[thick,fill] (1.24,-0.6) arc (0:360:1.5mm);
\node at (1.3,-1.12) {$a_3$};
\draw[thick,fill] (1.86,0.3) arc (0:360:1.5mm);
\node at (2.05, 0.53) {$\rotatebox[origin=c]{20}{\vdots}$};
\node at (2.05,-0.13) {$\rotatebox[origin=c]{-20}{\vdots}$};

\draw[thick,fill] (1.03,1.25) arc (0:360:1.5mm);
\node at (1.1,1.7) {$a_n$};

\node at (3,0) {$=$};
\end{scope}

\begin{scope}[shift={(4.25,-2.25)}]
%\draw[thin,yellow] (0,0) grid (5,4);
\draw[thick,<-] (0,1) arc (135:-225:1cm);

\draw[thick,fill] (0.24,-0.5) arc (0:360:1.5mm);
\node at (-0.2,-1.0) {$\omega$};

\node at (3.75,0.4) {$\omega\in \Sigma^{\circ}$};
\end{scope}

\begin{scope}[shift={(11.5,-2)}]
%\draw[thin,yellow] (0,0) grid (4,4);
\draw[thick,->] (0,0) -- (4.5,0);

\draw[thick,fill] (0.65,0) arc (0:360:1.5mm);
\node at (0.5,-0.5) {$\cdots$};
\draw[thick,fill] (1.25,0) arc (0:360:1.5mm);
\node at (1.25,-0.5) {$a_{-1}$};
\draw[thick,fill] (2.0,0) arc (0:360:1.5mm);
\node at (2.0,-0.5) {$a_0$};
\draw[thick,fill] (2.7,0) arc (0:360:1.5mm);
\node at (2.7,-.5) {$a_1$};
\draw[thick,fill] (3.3,0) arc (0:360:1.5mm);
\node at (3.3,-.5) {$a_2$};
\draw[thick,fill] (3.85,0) arc (0:360:1.5mm);
\node at (3.85,-.5) {$\cdots$};

\node at (5,0) {$=$};

\end{scope}

\begin{scope}[shift={(17.0,-2)}]
%\draw[thin,yellow] (0,0) grid (4,4);
\draw[thick,->] (0,0) -- (2,0);

\draw[thick,fill] (1.15,0) arc (0:360:1.5mm);
\node at (1.15,-0.5) {$\omega$};
 
\node at (4,0) {$\omega\in \Sigma^{\Z}$}; 
\end{scope}

\end{tikzpicture}
    \caption{Top: finite word $\omega'=a_1\dots a_n$ along an arc with two outer endpoints. Bottom left: floating circle component with a finite circular word. Bottom right: floating interval component with an infinite word $\omega$.  }
    \label{sofic-0002}
\end{figure}
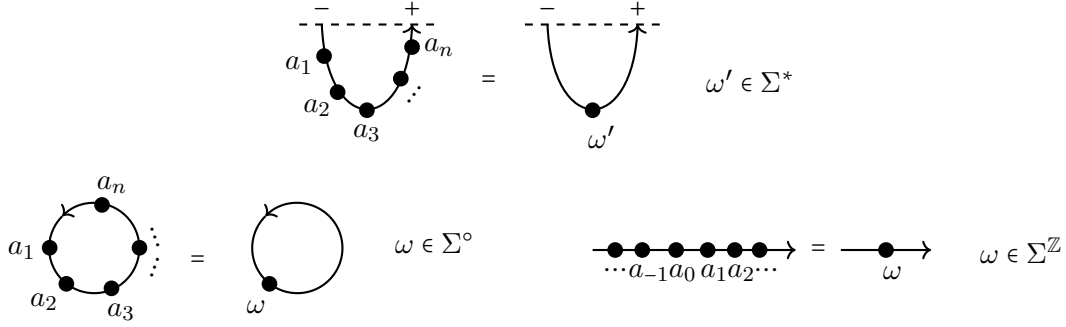

%%%%%%%%%%%%%%%%%%%%%%
% Evaluations infinite finite circular
%%%%%%%%%%%%%%%%%%%%%%

\subsection{Evaluations of infinite and finite circular words}

By an \emph{inficircular evaluation} or \emph{$\Z_{\circ}$-evaluation function} $\alpha=(\alphaZ,\alphac)$ we mean a pair of maps 
\begin{equation}\label{eq_alpha_Z}
    \alphaZ : \Sigmalr \lra \Bool, \hspace{1cm}
    \alphac : \Sigmac \lra \Bool,
\end{equation}
with a $\Z$-invariant $\alphaZ$, 
evaluating $\Z$-orbits on doubly-infinite words, respectively $\alphac$ evaluating circular words, to elements of $\Bool$.  

Choose an evaluation function $\alpha$ as in \eqref{eq_alpha_Z}. 
%\begin{equation}
%    \alphalr\ : \ \Sigmalr/\Z \lra \Bool , \ \ 
%    \alphac \ : \ \Sigmac \lra \Bool.
%\end{equation}
Equivalently, $\alpha$ can be described by a pair of languages 
\begin{equation}
    L=(L_{\Z},L_{\circ}), \hspace{1cm}
    L_{\Z} = \alphaZ^{-1}(1),  \hspace{1cm}
    L_{\circ} = \alphac^{-1}(1). 
\end{equation}
Here $L_{\Z}\subset \Sigmalr$ is a $\Z$-invariant subset that consists of all doubly-infinite words which evaluate to $1$ under $\alphaZ$, and $L_{\circ}$ consists of circular words evaluating to $1$ under $\alphac$.  

Using inficircular evaluation $\alpha$, we pass to the intermediate category ${\omcC}_{\alpha}^{\infty}$. It has the same objects as $\Cob^{\infty}_{\Sigma}$. Any floating component $C$ of a cobordism $U$ evaluates to $\alpha(C)\in \Bool$, and finite $\Bool$-linear combinations of cobordisms are allowed as morphisms. Hom spaces $\Hom(\varepsilon, 
\varepsilon')$ in $\Cob^{\infty}_{\alpha}$ are free $\Bool$-modules with a basis of all (diffeomorphism classes of) cobordisms with defects from $\varepsilon$ to  $\varepsilon'$ and without floating components. For instance, the basis of $\Hom(+,+)$ is parametrized by pairs of infinite words $(\omega_1,\omega_2),$ $\omega_1\in \Sigmal,\omega_2\in \Sigmar$ and finite words $\omega\in \Sigma^{\ast}$, see Figure~\ref{sofic-0003}.  

\vspace{0.07in} 

\begin{figure}
    \centering
\begin{tikzpicture}[scale=0.6]

\begin{scope}[shift={(0,0)}]
%\draw[thin,yellow] (0,0) grid (4,4);
\draw[thick,<-] (0,1) arc (135:-225:1cm);
\draw[thick,fill] (0.24,-0.5) arc (0:360:1.5mm);
\node at (-0.2,-1.0) {$\omega$};

\draw[thick,->] (2.25,0.25) -- (3.5,0.25);
\node at (2.75,0.75) {$\alphac$};
\node at (4,0.25) {$\Bool$};
\node at (6.5,0.25) {$\omega\in \Sigma^{\circ}$};
\end{scope}

\begin{scope}[shift={(10,0.25)}]
%\draw[thin,yellow] (0,0) grid (4,4);
\draw[thick,->] (0,0) -- (2,0);

\draw[thick,fill] (1.15,0) arc (0:360:1.5mm);
\node at (1.00,-0.5) {$\omega$};
 
\draw[thick,->] (2.5,0) -- (3.75,0);
\node at (3,0.5) {$\alphaZ$};
\node at (4.25,0) {$\Bool$};
\node at (6.75,0) {$\omega\in \Sigma^{\mathbb{Z}}$};
\end{scope}

\begin{scope}[shift={(0.5,-4.5)}]
\node at (0.5,2.35) {$+$};
\draw[thick,dashed] (0,2) -- (1,2);
\draw[thick,<-] (0.5,2) .. controls (0.5,1.4) and (1.25,0.5) ..  (1,0);

\draw[thick,fill] (0.95,1.0) arc (0:360:1.5mm);
\node at (1.45,0.97) {$\omega_2$};
\node at (4.95,1.05) {$\omega_2 =\cdots a_{-3}a_{-2}a_{-1} $}; 
\end{scope}

\begin{scope}[shift={(0.0,-6.5)}]
\node at (0.5,-0.35) {$+$};
\draw[thick,dashed] (0,0) -- (1,0);
\draw[thick,<-] (0.5,2) .. controls (1.5,1.5) and (0.5,0.5) .. (0.5,0);

\draw[thick,fill] (1.05,1.0) arc (0:360:1.5mm);
\node at (0.27,0.97) {$\omega_1$};
\node at (4.1,0.75) {$\omega_1 =a_0a_1a_2\cdots$}; 
\end{scope}

\begin{scope}[shift={(10.5,-5.75)}]
\node at (0.5,-0.35) {$+$};
\draw[thick,dashed] (0,0) -- (1,0);
\node at (0.5,3.35) {$+$};
\draw[thick,dashed] (0,3) -- (1,3);

\draw[thick,->] (0.5,0) -- (0.5,3);
\draw[thick,fill] (0.65,1.5) arc (0:360:1.5mm);
\node at (1.05,1.5) {$\omega$};
\node at (4.75,1.5) {$\omega = a_1 \cdots a_n\in \Sigma^*$};

\end{scope}

\end{tikzpicture}
    \caption{Top row: two types of connected floating components: circles and intervals, and their evaluations. Bottom right: two types of diagrams in $\Hom(+,+)$ without floating components, including a pair of half-intervals (left) and an arc (right).
    }
    \label{sofic-0003}
\end{figure}
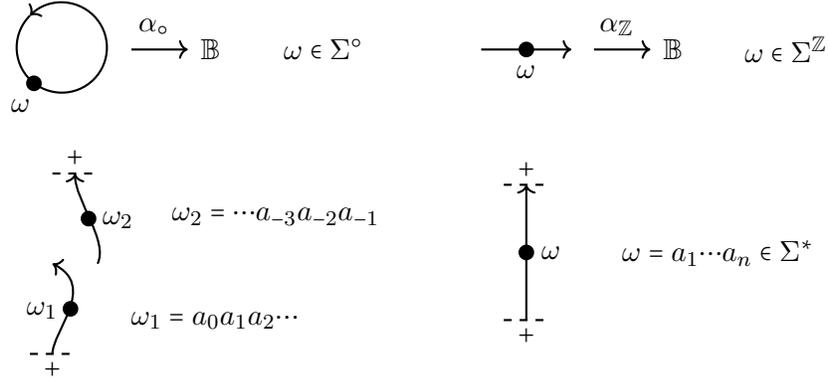

From the intermediate category $\omcC^{\infty }_{\alpha}$ we pass to the category $\Cob^{\infty}_{\alpha}$ via the universal construction for Boolean evaluations~\cite{IK-top-automata}. The new category has the same objects as the earlier two categories -- sign sequences. Hom spaces in  $\Cob^{\infty}_{\alpha}$ are quotient semimodules of the corresponding homs in $\omcC^{\infty}_{\alpha}$. Two elements $f_1,f_2$ of the hom space   $\Hom(\varepsilon,\varepsilon')$ in the category  $\omcC^{\infty}_{\alpha}$ are equal in $\Cob^{\infty}_{\alpha}$ if for any morphism $g\in \Hom(\varepsilon'^{\ast},\varepsilon^{\ast})$ the evaluations $\alpha(\mathrm{cl}(f_1,g))=\alpha(\mathrm{cl}(f_2,g))$, see Figure~\ref{closed-evaluation}. 

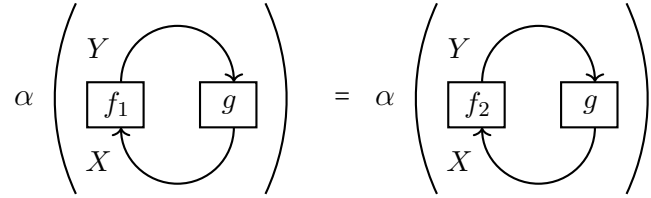
\begin{figure}
    \centering
\begin{tikzpicture}[scale=0.6]
%\draw[thin,gray] (-5,-5) grid (5,5);
%\draw[thick,->] (-5,0) -- (5,0);
%\draw[thick,->] (0,-5) -- (0,5);

\draw[thick] (-3.75,-3.75) rectangle (-2.5,-2.75);
\node at (-3.1,-3.25) {$f_1$};
\draw[thick] (-1.25,-2.75) rectangle (0,-3.75);
\node at (-0.6,-3.25) {$g$};

\draw[thick,<-] (-0.5,-2.75) arc (0:180:1.25);
\node at (-3.5,-2) {$Y$};

\draw[thick,<-] (-3,-3.75) arc (180:360:1.25);
\node at (-3.5,-4.5) {$X$};

\node at (-5.15,-3.1) {\large $\alpha$};
\draw[thick] (-4,-1) arc (155:205:5);

\draw[thick] (0.2,-1) arc (25:-25:5);
\node at (1.9,-3.1) {\large $=$};

\draw[thick] (4.25,-3.75) rectangle (5.5,-2.75);
\node at (4.9,-3.25) {$f_2$};
\draw[thick] (6.75,-2.75) rectangle (8,-3.75);
\node at (7.4,-3.25) {$g$};

\draw[thick,<-] (7.5,-2.75) arc (0:180:1.25);
\node at (4.5,-2) {$Y$};

\draw[thick,<-] (5,-3.75) arc (180:360:1.25);
\node at (4.5,-4.5) {$X$};

\node at (2.85,-3.1) {\large $\alpha$};
\draw[thick] (4,-1) arc (155:205:5);

\draw[thick] (8.2,-1) arc (25:-25:5);
\end{tikzpicture}
    \caption{Evaluation for the equivalence relation: $f_1\sim_{\alpha} f_2$ if 
$\alpha(\cl(f_1,g))=\alpha(\cl(f_2,g))$ for all $g \in \Hom(X^{\ast}, Y^{\ast})$. Here, $\cl$ is the closure of a pair of morphisms, see~\cite{IK-top-automata} for instance.}
    \label{closed-evaluation}
\end{figure}

\vspace{0.07in}

We thus have a chain of categories and functors 
\begin{equation} \Cob^{\infty}_{\Sigma} \lra \omcC^{\infty}_{\alpha}  \lra \Cob^{\infty}_{\alpha}.
\end{equation}  
Hom spaces in the second and third categories are $\Bool$-semilinear, and the last category is obtained via the universal construction for the evaluation $\alpha$. 

Define the state space of a sequence $\varepsilon$: 
\[ A(\varepsilon) =A_{\alpha}(\varepsilon) \ := \ \Hom_{\mcCinfalpha}(\emptyset, \varepsilon).
\] 
It consists of finite $\Bool$-linear sums of decorated 1-manifolds, as above, with boundary $\varepsilon$, modulo the universal construction equivalence relation for the evaluation $\alpha$. 
A morphism $\phi$ in $\Cob^{\infty}_{\Sigma}$ from $\varepsilon$ to $\varepsilon'$ induces a $\Bool$-linear map of state spaces  $\phi_{\ast}:A(\varepsilon)\lra A(\varepsilon')$. 
Varying over all morphisms in $\mcCinfS$ results in a lax monoidal functor 
\[
\mcF_{\alpha} \ : \ \mcCinfS \lra \Bool\dmod 
\]
to the category of $\Bool$-semimodules.

%%%%%%%%%%%%%%%%%%%%%
% Regular and sofic evaluations 
%%%%%%%%%%%%%%%%%%%%%

\subsection{Regular and sofic evaluations and their state spaces}

 Here is the key terminology for this section, see also Section~\ref{subsec_infinite}:  
\begin{itemize}
\item 
Evaluation $\alphaZ:\SigmaZ\lra \Bool$ on infinite words (or the infinite language $L(\alphaZ)$) is called \emph{sofic} if $W^{\perp}_{\f}$ is a regular language. Equivalently, $W=W(\alphaZ)$ is regular. (See Table~\ref{tab:my_label} and discussion there for terminology.) 
\item 
Evaluation $\alphaZ$ on infinite words is called \emph{regular} 
if the state space $A(+)$ is finite. (Note that $\alphaZ^{-1}(1)$ may not be closed in $\SigmaZ$, see Examples~\ref{ex_lang1},~\ref{example_two}.) 
\item 
\emph  {Inficircular} evaluation $\alpha=(\alphaZ,\alphac)$ for infinite words and for finite circular words, as in (\ref{eq_alpha_Z}),  is called \emph{regular} if state spaces $A(\varepsilon)$ are finite $\Bool$-modules for all sign sequences $\varepsilon$. This is equivalent to $A(+)$ and $A(+-)$  being finite. 
\end{itemize}

The state spaces $A(+),A(-)$, as defined above, depend only on the infinite evaluation $\alphaZ$ and not on the circular finite evaluation $\alphac$.

\begin{example}\label{ex_lang1} 
 Consider a non-closed infinite language $L$ on the alphabet $\{a,b\}$ which consists of all words $\omega=\cdots a_i \cdots $ which have only $a$'s in the right tail: $a_i=a$ for all $i>N$ for some $N\in \mathbb{Z}$. Any two words $\omega_1,\omega_2\in \Sigmal$ give the same state $\langle \omega_1 | = \langle \omega_2 | \in A(+)$ since for any $\omega'\in \Sigmar$, words $\omega_1\omega'$ and $\omega_2\omega'$ are either both in $L$ or both not in $L$ (this depends only on $\omega'$). 
 
 Consequently, $A(+)$ is a free rank one $\Bool$-module, $A(+)=\Bool\langle \omega_1|$ for any $\omega_1\in \Sigmal$. Likewise, $A(-)=\Bool |\omega'\rangle$ is a free rank one $\Bool$-module, where $\omega'=aaa\cdots\in \Sigmar$. 
 
 The language $L$ is not a closed subset of $\SigmaZ$ since the complement $\SigmaZ\setminus L$ is not open - it does not contain any basic open subsets of $\SigmaZ$ in the product topology. 
  $L$ is an example of an non-closed language with finite state spaces $A(-),A(+)$. Such examples (also see the example that follows) motivate restricting to closed languages when studying state spaces and related categories $\Cob^{\infty}_{\alpha}$. Likewise, symbolic dynamics restricts to studying closed $\Z$-invariants subspaces of $\SigmaZ$ (the latter are in a bijection with closed languages, according to our definition). 
\end{example}

\begin{example}\label{example_two} (A generalization of Example~\ref{ex_lang1}.) 
Consider an equivalence relation $\sim_{\h}$ on $\Sigma^{\h}$ where two words $\omega_1,\omega_2\in \Sigma^{\h}$ are equivalent if they have the same (infinite) head:
\begin{equation}
    \omega_1\sim_{\h} \omega_2 \ \ \Leftrightarrow \exists u_1,u_2\in \Sigma^{\ast}: \ \ \omega_i=u_i \omega, \ \ \omega\in \Sigma^{\h}.
\end{equation}
There are countably many elements in each equivalence class and the number of equivalence classes $\Sigma^{\h}/\sim_{\h}$ is uncountable. 

Choose a function $f:\Sigma^{\h}/\sim_{\h}\lra \Bool$ which is not identically $0$. To $f$, associate an evaluation 
\begin{equation}
    \alpha_{f} :\Sigma^{\Z}\lra \Bool, \ \ \alpha_{f}(\omega'\omega'') = f(\omega''), \hspace{0.5cm} 
    \omega'\in \Sigma^{\t},
    \hspace{0.5cm}
    \omega''\in \Sigma^{\h}. 
\end{equation}
Note that $\alpha_{f}$ depends only on the head of $\omega'\omega''$, where it is determined by $f$. 

The state space $A(+)$ of $\alpha_f$ is one-dimensional, $A(+)\cong \Bool v'$, with $\langle \omega'|=v'$ for any $\omega'\in \Sigma^{\t}$. The dual state space $A(-)=\Bool v''$ has 
$|\omega''\rangle=v''$ if and only if $f(\omega'')=1$. Otherwise, $|\omega''\rangle=0$. The pairing $A(+)\times A(-)\lra \Bool$ is $(v',v'')=1$.

Note that there are $2^{\aleph_0}$ such evaluations, one for each $f$ as above, since $\Sigma^{\h}/\sim_{\h}$ has cardinality $\aleph_0$. Each of these evaluations produces one-dimensional state spaces $A(+),A(-)\cong \Bool$. 

See~\cite{Sta83} and references therein for essentially this example. 
\end{example}

Pick a regular evaluation $\alphaZ:\SigmaZ\lra \Bool$, so that state spaces $A(+),A(-)$ are finite. Consider free $\Bool$-modules $\Bool^{A(+)}$ and $\Bool^{A(-)}$ with bases of all elements of these two finite sets. Right action of $\Sigma$ on the set $A(+)$ extends to a right action on the free module $\Bool^{A(+)}$. Likewise, left action $\Sigma$ on $A(-)$ extends to a left action on $\Bool^{A(-)}$. There is a natural bilinear pairing 
\begin{equation}
    (\:\:,\:\:) \ : \ \Bool^{A(+)} \times \Bool^{A(-)} \lra \Bool 
\end{equation}
induced by the corresponding pairing $A(+)\otimes A(-)\lra \Bool$.
We can now map words in $\Sigmal$ to $\Bool^{A(+)}$ by lifting the map to $A(+)$, and likewise for $\Sigmar$ and $\Bool^{A(-)}$. 

Let us now reverse the pairing $(\:\:,\:\:)$ on free modules to the pairing 
\begin{equation}
    (\:\:,\:\:)' \ : \ \Bool^{A(+)} \times \Bool^{A(-)} \lra \Bool ,
\end{equation}
so that  on basis vectors $(v,w)'=1$ if and only if $(v,w)=0$ and $(v,w)'=0$ if and only if $(v,w)=1$. Then the new pairing gives evaluation $\alphaZ'$ complementary to $\alpha$: 
\begin{equation}
    \alphaZ'(\omega)=1 \ \ \Leftrightarrow \alphaZ(\omega)=0, \ \  \omega \in \SigmaZ. 
\end{equation}
Consequently, an evaluation $\alpha$ is regular if and only if its complementary evaluation $\alpha'$ is regular.

\begin{theorem}
\label{theorem:regular_sofic}
For a closed evaluation $\alphaZ$ the following are equivalent: 
\begin{enumerate}
    \item\label{item:closedZ-ev-first} 
        The evaluation $\alphaZ$ is regular. 
    \item\label{item:closedZ-ev-second} 
        The state space $A(+)$ is finite. 
    \item\label{item:closedZ-ev-third} 
        The evaluation $\alphaZ$ is \emph{sofic}. 
\end{enumerate}
\end{theorem} 
 
 \begin{proof}
The pairing $A(+)\times A(-)\lra \Bool$ is nondegenerate (i.e., separating). In particular, $A(+)$ is finite if and only if $A(-)$ is finite, in which case they have the same cardinality, $|A(+)|=|A(-)|$. This takes care of the equivalence between the first two descriptions above. 

To show \eqref{item:closedZ-ev-first} implies \eqref{item:closedZ-ev-third}, assume $A(+)$ is finite. Recall the map $\Sigmal\lra A(+)$ taking a semi-infinite word $\omega\in\Sigmal$ to the state $\langle\omega|$.  Let $S = A(+) \sqcup \{ * \}$. Define a transition function $\delta : S \times \Sigma \to S$ to be the semilinear extension of the right concatenation map $\delta(v, a) \mapsto v a$ if $v \in A(+)$, together with 
\begin{equation}
\delta(*, a) = \left(\sum_{v \in A(+)} v\right) a \, \in A(+), \ \  a\in \Sigma. 
\end{equation}
Note that $\sum_{v \in A(+)}v$ is the largest element in the finite semilattice $A(+)$. 
  We claim that the tuple $D = (S, \cdot, *, \{0\})$, where $*$ is the initial state and $0$ is the only accepting state, defines a DFA for $W$ in \eqref{eq_Wprime}. It is clear that the execution of a word $w \in W$ will lead to an accepting state of $D$ since $\langle\omega^-| w = \langle\omega^- w|=0$ for all $\omega^- \in \Sigmal$.  On the other hand, if $w \in \Sigma^* \setminus W$, then there exist semi-infinite words $\omega^- \in \Sigmal$ and $\omega^+ \in \Sigmar$ such that $\alpha(\omega^- w \omega^+) = 1$.   In particular $\langle\omega^- w |\neq 0$, so the execution of $w$ must be rejected by $D$.  Thus the language accepted by $D$ is precisely $W$, so the language $W$ is regular.
 
To show that \eqref{item:closedZ-ev-third} implies \eqref{item:closedZ-ev-first}, assume that the evaluation $\alphaZ$ is sofic.  Then the set of forbidden finite words $W \subset \Sigma^{\ast}$ forms a regular language.   Let $W^{\perp}_{\f}\subset \Sigma^{\ast}$, see \eqref{eq-Wperp}, denote the set of finite words containing no subword in $W$, which is regular since it is the complement of $\Sigma^{\ast} W \Sigma^{\ast}$ in $\Sigma^{\ast}$. %\cite{rozenberg2012handbook}   
Let $F = (S, \cdot, s_{\init}, S_{\t})$ be a minimal DFA for $W$, where $S$ is a finite set of states, $\cdot : \Sigma \times S \to S$ is the transition function, $s_{\init} \in S$ is the initial state, and $S_{\t} \subset S$ is the set of accepting states. Note that $F$ has a unique non-accepting state, so that $S\setminus S_{\t}$ has cardinality one.   The $\Bool$-semimodule $\Bool S_{\t}$ is equipped with right action of $\Sigma$ given by composing the transition function $\cdot$ with the map $S \to \Bool S_{\t}$ sending the non-accepting state to 0, and leaving accepting states unchanged.

 It suffices to construct an embedding $A(+) \lra \Bool S_{\t}$.   For a semi--infinite word 
     \[\omega=\cdots a_{-n} \cdots a_{-2}a_{-1} \in \Sigmal,
     \]
 let $S_{\omega}^{+} \subset S_{\t}$  denote the set of  accepting states that appear infinitely many times in its sequence of partial executions, i.e.
 $$ S_{\omega}^{+} = S_{\t} \cap \bigcap_{N=1}^\infty   
 \{ s_{\init} \cdot  a_{-n} a_{-n+1} \cdots a_{-1} \}_{n = N}^\infty. $$
 
 For $\sigma=\sigma_0\sigma_1\cdots  \in \Sigmar$, let $S_\sigma^{-} \subset S_{\t}$ denote the set of starting states from which all partial executions of $\sigma$ are accepting,  i.e.
$$ S_\sigma^{-} =   \bigcap_{n=0} ^\infty \left\{ s \in S_{\t} \mid s \cdot \sigma_0\sigma_1\cdots \sigma_n \in S_{\t} \right\}.$$  
% Do not think we need this -- \bigcup_{\omega \in \Sigmal} S_\omega^{+} \cap

Let $f^+:\Sigmal \lra \Bool S_{\t}$ be defined by 
\[
f^+(\omega) = \sum_{s \in S_\omega^+} s,
\]
and let $f^-: \Sigmar \lra (\Bool S_{\t})^*$ be defined by  
\[
f^-(\sigma) = \sum_{s \in S_\sigma^-} \delta_s.
\] 
Extend $f^+,f^-$ to $\Bool$-linear maps $f^+:\Bool \Sigmal \lra \Bool S_{\t}$ and $f^-:\Bool \Sigmar \lra (\Bool S_{\t})^*.$

We claim $\alpha(\omega \sigma) = (f^+(\omega), f^-(\sigma))$ for $\omega \in \Sigmal$ and $\sigma \in \Sigmar$, where $(\:\:,\:\:)$ denotes the pairing $\Bool S_{\t} \times (\Bool S_{\t})^{\ast}\lra \Bool$.  Suppose $\alpha(\omega \sigma) = 0$.  Then for all large $n$ the word $a_{-n} a_{-n+1} \cdots a_{-1} \sigma_{0} \cdots \sigma_n$ is rejected by the DFA $F$.  Thus $S_\omega^+ \cap S_\sigma^-= \emptyset$, which implies $(f^+(\omega), f^-(\sigma)) = 0$. On the other hand, if $\alpha(\omega \sigma) = 1$, then every finite substring of $\omega \sigma$ is accepted by the DFA $F$. In particular, the finiteness of $S_{\t}$ implies that $S_\omega^+$ must contain at least one state, which must also lie in $S_\sigma^-$.  This implies $(f^+(\omega),f^-(\sigma)) = 1$. 

The following holds:  
(a) $f^+$ intertwines the action of $\Sigma$ on $A(+)$ and $\Bool S_{\t}$, 
(b) $f^-$ intertwines the action of $\Sigma$ on $A(-)$ and $(\Bool S_{\t})^{\ast}$, 
(c) maps $f^+,f^-$ are  $\Bool$-semilinear.

For (a), let $b \in \Sigma$.  We have

\begin{align*}
f^+(\omega b)  & = \sum  \left\{s : s \in %\left(
 \bigcap_{N=1}^\infty  
\left\{ s_{\init} \cdot  a_{-n} a_{-n+1} \cdots a_{-1} b \right\}_{n = N}^\infty \cap S_{\t}%\right) 
\right 
\}  \\
 & = \sum \left\{s \cdot b : s \in    \bigcap_{N=1}^\infty  
 \left\{ s_{\init} \cdot  a_{-n} a_{-n+1} \cdots a_{-1} 
 \right\}_{n = N}^\infty   \cap  S_{\t}, \quad  s \cdot b \in S_{\t} \right \} \\
  & = f^+(\omega) b
\end{align*}

For (b), 
\begin{align*}
f^-(b \omega)  & = \sum  \left\{ \delta_t(x) :  t \in \bigcap_{n=0} ^\infty \left\{ s \in S_{\t} \mid s \cdot b \sigma_0\sigma_1\cdots \sigma_n \in S_{\t} \right\} \right\} \\
`& = \sum  \left\{ \delta_t(x \cdot b) :  t \in \bigcap_{n=0} ^\infty \left\{ s \in S_{\t} \mid s \cdot \sigma_0\sigma_1\cdots \sigma_n \in S_{\t} \right\} \right\} \\
 & =  \sum  \left\{ b \cdot \delta_{t} :  t \in \bigcap_{n=0} ^\infty \left\{ s \in S_{\t} \mid s \cdot \sigma_0\sigma_1\cdots \sigma_n \in S_{\t} \right\} \right\} \\
  & = b f^-(\omega) 
\end{align*}

Property (c) follows from the definitions of $f^+$ and $f^-$.

\vspace{0.07in}

Note that $f^+$ factors through the quotient map, $f^+: \Bool\Sigmal\lra A(+)\lra \Bool S_{\t}$, and likewise for $f^-$. 
We claim that $f^+$ descends to an embedding of $\Bool$-semimodules $A(+) \to \Bool S_{\t}$, and similarly for $f^-$:

\begin{equation*}
   \begin{tikzcd}
\Bool S_{\t}  \times (\Bool S_{\t})^{\ast}     \arrow{rr} & & \Bool \\
&  & \\ 
A(+)  \times A(-) \arrow[uu, "f^-" left, hookrightarrow, shift right=5]  \arrow[uu, "f^+" left, hookrightarrow,shift left=5]  \arrow{uurr}& &  
\end{tikzcd}  
\end{equation*}

To see that $f^+$ and $f^-$ descend to $A(+)$ and $A(-)$, respectively, suppose $f^+(\omega) \neq f^+(\omega')$ for some $\omega, \omega' \in \Sigmal$.  Then without loss of generality, we have $(f^+(\omega), \delta_s) = 1$ and $(f^+(\omega'), \delta_s) = 0$ for some state $s \in S_{\t}$.  By the minimality of $F$, there exists $\sigma \in \Sigmar$ such that  $s \in S_\sigma^{-}$.  Thus $\alpha(\omega \sigma) = f^+(\omega) f^-(\sigma) = 1$, but $\alpha(\omega' \sigma) = f^+(\omega) f^-(\sigma) = 0$.   Thus, $\omega \neq \omega'$ as elements of $A(+)$.

Similarly, suppose $f^-(\sigma) \neq f^-(\sigma')$ for some $\sigma, \sigma' \in \Sigmar$.  Then without loss of generality, there exists a functional $\delta_s \in S_\sigma^- \setminus S_{\sigma'}^-$.   Pick  any $\omega \in \Sigmar$ such that $s \in S_\omega^+$.  Then $\alpha(\omega \sigma) = 1$, but $\alpha(\omega \sigma') = 0$.  Thus $\sigma \neq \sigma'$ in $A(-)$.

To see that $f^+$ is an embedding, suppose that $\omega \neq \omega'$ for some $\omega, \omega' \in A(+)$.  Then without loss of generality, there exists $\sigma \in \Sigmar$ such that $\alpha(\omega \sigma)  \neq \alpha(\omega' \sigma)$.  Thus, $(f^+(\omega), f^-(\sigma)) \neq (f^+(\omega'), f^-(\sigma))$.  It follows from the nondegeneracy of the pairing $\Bool S_{\t} \times (\Bool S_{\t})^* \to \Bool$ that $f^+(\omega) \neq f^+( \omega')$.  A similar argument applies for $f^-$.
\end{proof}

We can restate the theorem above as follows. 
\begin{corollary}
 A closed evaluation $\alphaZ$ is regular if and only if it is sofic. 
\end{corollary}

\begin{theorem} \label{thm_inficirc} For an inficircular evaluation $\alpha=(\alphaZ,\alphac)$ with $\alphaZ$ closed the category $\Cob^{\infty}_{\alpha}$ has finite hom spaces if and only if $\alphaZ$ is sofic and $\alphac$ is regular. 
\end{theorem} 
\begin{proof}
Suppose $\Cob^\infty_\alpha$ has finite hom spaces.  Then $\alphaZ$ is sofic by the previous theorem, and $\alphac$ is regular by  \cite[Proposition~3.19]{IK-top-automata}.

Conversely, suppose that $\alphaZ$ is sofic and $\alphac$ is regular.  It suffices to show that $A(+)$, $A(-)$, and $A(+-)$ are finite since this implies $A(\varepsilon)$ is finite for an arbitrary sign sequence $\varepsilon$.    By the previous theorem, $A(+)$ and $A(-)$ are finite. 

\begin{figure}
    \centering
\begin{tikzpicture}[scale=0.6]
\begin{scope}[shift={(0,0)},decoration={markings,mark=at position 0.85 with {\arrow{>}}}]
%\draw[thin,yellow] (0,0) grid (4,4);

\node at (1,4.5) {$+$};
\node at (3,4.5) {$-$};

\draw[thick,postaction={decorate}] (1,1) -- (1,4);
\draw[thick,postaction={decorate}] (3,4) -- (3,1);

\node at (0.25,2.5) {$\omega_1$};
\node at (3.75,2.5) {$\omega_2$};

\draw[thick,dashed] (0,4) -- (4,4);

\draw[thick,fill] (1.15,2.5) arc (0:360:1.5mm);
\draw[thick,fill] (3.15,2.5) arc (0:360:1.5mm);
\end{scope}

\begin{scope}[shift={(8,0)},decoration={markings,mark=at position 0.85 with {\arrow{>}}}]
%\draw[thin,yellow] (0,0) grid (4,4);

\node at (1,4.5) {$+$};
\node at (3,4.5) {$-$};

\draw[thick,dashed] (0,4) -- (4,4);

\draw[thick,postaction={decorate}] (3,4) .. controls (2.85,1) and (1.15,1) .. (1,4);

\draw[thick,fill] (2.15,1.7) arc (0:360:1.5mm);
\node at (2,1.15) {$\omega$};

\end{scope}

\end{tikzpicture}
    \caption{Elements $\uparrow\downarrow \!\!(\omega_1,\omega_2)$ and $\circleft(\omega)$ span $A(+-)$.}
    \label{fig_sofic-0007}
\end{figure}

\vspace{0.07in}

$A(+-)$ is spanned by elements $\uparrow\downarrow \!\!(\omega_1,\omega_2)$ and $\circleft(\omega)$ for $\omega_1\in \Sigma^{\t},\omega_2\in\Sigma^{\h}$ and $\omega\in \Sigma^{\ast}$, see Figure~\ref{fig_sofic-0007}. Elements of the first type span the image of $A(+)\otimes A(-)$ in $A(+-)$, which is a quotient of $A(+)\otimes A(-)$, hence a finite subsemilattice of $A(+-)$.  

By inspecting each relevant cobordism for $A(+-)$, one finds that $A(+-)$ must be finite, using the regularity of $\alphac$ and the finiteness of $A(+)$ and $A(-)$. 
\end{proof}

\subsection{Decomposition of the identity}

$\quad$ 

A sofic  evaluation $\alphaZ$ on (doubly)-infinite words gives rise to finite state spaces $A(+)$ and $A(-)\cong A(+)^{\ast}$. Suppose furthermore that $A(+)$ is a projective $\Bool$-module, that is, a retract of the free finite semimodule. This means that for some $n$ there exists semimodule maps  
\begin{equation}
    \Bool^n \stackrel{p}{\lra} A(+)\stackrel{\iota}{\lra} \Bool^n, \hspace{1cm} 
    p\circ \iota = \id_{A(+)}.  
\end{equation}
Then there are natural coevaluation and evaluation maps 
\[
\mathsf{coev} : \Bool \lra A(+)\otimes A(-),  \hspace{1cm} 
\mathsf{ev}: A(+)\otimes A(-)\lra \Bool 
\]
that satisfy isotopy relations and give rise to a Boolean 1D TQFT (without defects and inner endpoints).  Coevaluation and evaluation maps are associated to the \emph{cup} and \emph{cap} cobordisms in Figure~\ref{figure-2}. The evaluation can be defined more generally for any finite $\Bool$-semimodule $A(+)$ but coevaluation requires projectivity, see~\cite[Section 3.2]{IK-top-automata} and~\cite{GIKKL23}. 

\begin{prop}\label{prop_dec_id}
    If $\alphaZ$ is sofic and $A(+)$ is projective, there is a Boolean TQFT functor 
    \begin{equation}
        \mcF_{\alphaZ} \ : \ \Cob_{\Sigma}^{\infty} \lra \Bool\dmod
    \end{equation}
    which evaluates floating intervals via  $\alphaZ$ and a circle with a word $\omega\in \Sigma^{\ast}$ on it to the trace $\tr_{A(+)}(\omega)$ of action of $\omega$ on  $A(+)$. 
\end{prop}

In this situation we say that $\alphaZ$ admits a \emph{decomposition of the identity}. Detailed explanations for this proposition in the case of finitely many defects on intervals can be found in~\cite[Section 4.4]{IK-top-automata}; the accumulating defects case is similar and left to the reader. Note that a finite semilattice $A(+)$ is projective if it is a distributive lattice, see~\cite[Proposition 3.8]{IK-top-automata} and references therein.

In examples in Section~\ref{sec_examples_sofic} we compute TQFT extension of sofic evaluations $\alphaZ$ when $A(+)$ is projective. 

\vspace{0.07in} 

A more general class of TQFT liftings of a sofic shift exists, even when $A(+)$ is not projective. Recall the following result, which is Proposition 3.3 in~\cite{BBE21}. 

\begin{prop} A shift $X$ is sofic if and only if there is a finite automaton $(Q)$ such that $X=L_{(Q)}$. 
\end{prop}
Here $L_{(Q)}$ is the set of bi-infinite paths in $Q$ and automaton $(Q)$ is non-deterministic. Initial and terminal states in $(Q)$ are not specified, since they are not needed to describe bi-infinite paths. 

Given such $(Q)$, we build a TQFT $\mcF_{(Q)}$ extending evaluation $\alphai:\SigmaZ\lra \Bool$  associated with $X$ as follows. To $+$ and $-$ oriented $0$-manifolds $\mcF_{(Q)}$ associates free $\Bool$-modules $\Bool Q, \Bool Q^{\ast}$, respectively, see \eqref{align_1}. To the cup and cap cobordisms in Figure~\ref{figure-2} associate the coevaluation and evaluation maps as in \eqref{align_1}. 

To a left-infinite word $\omega=\cdots a_{-2}a_{-1}\in \Sigmal$ associate the element $\mcF(\omega)\in \Bool Q$ which is the sum over all states $q$ such that $(Q)$ contains a semi-infinite path $\omega$ terminating in $q$. 

To a right-infinite word $\omega'=a_0a_1a_2\cdots \in \Sigmar$ associate the element $\mcF(\omega')\in \Bool Q^{\ast}$ which is the sum over all $q^{\ast}$ such that there is a right-infinite path $\omega'$ in $(Q)$ starting in $q$. 

These assignments are consistent and give rise to a TQFT \[\mcF_{(Q)}: \Cob_{\Sigma}^{\infty}\lra \Bool\mathsf{-fmod}
\]
valued in free $\Bool$-modules. A floating interval carrying a bi-infinite word $\omega$ evaluates to $1$ if and only if $\omega \in X=L_{(Q)}$. A circle carrying a finite word $\omega$ evaluates to $\tr_{\Bool Q}(\omega)$, the trace of operator $
\omega$ on $\Bool Q$. The latter is $1$ if and only if there is a closed path $\omega$ in $(Q)$. Evaluation of decorated circles defines a regular circular language $L_{\circ}$ which depends on the automaton $(Q)$ and not only on $X$.

%%%%%%%%%%%%%%%%%%
%%
% Examples 
%
%%%%%%%%%%%%%%%%%%

\section{Examples of state spaces for sofic systems} \label{sec_examples_sofic}

Denote by $\ast\omega\in \Sigma^{\t}$ a left-infinite word that ends in $\omega$ and does not contain 
any prohibited words. This notation is only used in the case when the element $\langle \ast \omega |\in A(+)$ is well-defined, that is, does not depend on a choice of completion of $\omega$ to such an left-infinite word. Notation $\omega\ast\in \Sigma^{\h}$ is used likewise, with $|\omega \ast\rangle$ to denote an element of $A(-)$.

\begin{example}
\label{ex:sofic-system-even-number-of-b}

Let $W=ab(b^2)^{\ast}a$ consist of odd-length products of $b$ surrounded by $a$. 
The language $W^{\perp}\subset \SigmaZ$ associated with $W$ consists of infinite words $\omega$ that do not contain any subwords of the form  $ab^{2n+1}a, n\ge 0$. 

Consider the following  states in $A(+)$, where 
$b^{-\infty}:= \cdots bbb$:
\[v_1 = \langle \ast a b^{2n+1}|, \hspace{1cm}
v_2 = \langle \ast a b^{2n}|, \hspace{1cm}
v_3 = \langle b^{-\infty} | . 
\] 
Here $v_1$, for instance, is a non-zero state represented by a word with an odd number of $b$'s at its end ($v_1$ does not depend on the choice of $n\ge 0$). These three states, clearly, span $A(-)$.   

Consider the following states in $A(-)$, where $b^{+\infty}:=bbb\cdots$: 
\[v_1' = | b^{2m+1} a\ast \rangle, \hspace{1cm} 
v_2' = | b^{2m} a \ast \rangle,  \hspace{1cm}
v_3' = | b^{+\infty} \rangle. 
\]
The matrix of the bilinear pairing on these spanning vectors is shown in Figure~\ref{sofic-0005}.  

\begin{figure}
    \centering
\begin{tikzpicture}[scale=0.6]
\begin{scope}[shift={(0.5,0)}]
%\draw[thin,yellow] (0,0) grid (4,4);
\node at (0.25,3) {$v_1=\langle \ast ab^{2n+1}|$};
\node at (0,2) {$v_2=\langle \ast ab^{2n}|$};
\node at (-0.25,1) {$v_3=\langle b^{-\infty}|$};

\node at (7.25,3) {$v_1' = |b^{2m+1} a\ast \rangle $};
\node at (7.00,2) {$v_2' = |b^{2m} a \ast \rangle$};
\node at (6.60,1) {$v_3' = | b^{\infty} \rangle $};

\node at (3.375,-1) {$k,m,n\geq 0$};

\end{scope}

\begin{scope}[shift={(14.5,0.5)}]
\draw[thick] (0,0) grid (3,3);
\node at (0.5,2.5) {$1$};
\node at (0.5,1.5) {$0$};
\node at (0.5,0.5) {$1$};
%\node at (0.5,-0.5) {$1$};

\node at (1.5,2.5) {$0$};
\node at (1.5,1.5) {$1$};
\node at (1.5,0.5) {$1$};
%\node at (1.5,-0.5) {$1$};

\node at (2.5,2.5) {$1$};
\node at (2.5,1.5) {$1$};
\node at (2.5,0.5) {$1$};
%\node at (2.5,-0.5) {$1$};

%\node at (3.5,2.5) {$1$};
%\node at (3.5,1.5) {$1$}; 
%\node at (3.5,0.5) {$1$}; 
%\node at (3.5,-0.5) {$1$}; 

\node at (0.5,3.5) {$v_1$};
\node at (1.5,3.5) {$v_2$};
%\node at (2.0,3.5) {$=$};
\node at (2.5,3.5) {$v_3$};
%\node at (3.5,3.5) {$v_4$};

\node at (-0.5,2.5) {$v_1'$};
\node at (-0.5,1.5) {$v_2'$};
%\node at (-0.5,1.0) {\rotatebox[origin=c]{90}{$=$}};
\node at (-0.5,0.5) {$v_3'$};
%\node at (-0.5,-0.5) {$v_4'$};

\end{scope}

\begin{scope}[shift={(-0.45,-5)}]
\node at (0,3) {$(v_1,v_1')=\langle \ast ab^{2(m+n+1)}a \ast \rangle=1$};
\node at (0,2) {$(v_1,v_2')=\langle \ast ab^{2(m+n)+1}a \ast \rangle=0$};
\node at (-0.4,1) {$(v_1,v_3')=\langle \ast ab^{2n+1} b^{\infty}\rangle =1$};

\node at (9.15,3) {$(v_2,v_1')=\langle \ast ab^{2(n+m)+1}a \ast \rangle=0$};
\node at (8.9,2) {$(v_2,v_2')=\langle \ast ab^{2(m+n)}a \ast \rangle=1$};
\node at (8.55,1) {$(v_2,v_3')=\langle \ast ab^{2n}  b^{\infty} \rangle=1$};

\node at (17.95,3) {$(v_3,v_1')=\langle  b^{-\infty} b^{2m+1}a \ast \rangle=1$};
\node at (17.7,2) {$(v_3,v_2')=\langle  b^{-\infty}b^{2m}a \ast \rangle=1$};
\node at (17.35,1) {$(v_3,v_3')=\langle  b^{-\infty}  b^{\infty}  \rangle=1$};

\end{scope}

\end{tikzpicture}
    \caption{Bilinear form for  Example~\ref{ex:sofic-system-even-number-of-b}. Spanning vectors and the matrix of the form in the top row and computation of the entries below that. 
    }
    \label{sofic-0005}
\end{figure}

\begin{figure}
    \centering
\begin{tikzpicture}[>=stealth',shorten >=1pt,auto,node distance=2cm]
\node[thick,state] (v1) {$v_1$}; 
\node[thick] (xblank) [right of=v1] {};
\node[thick,state] (v3) [above of=xblank] {$v_3$};
\node[thick,state] (v2) [right of=xblank] {$v_2$};
\path[thick,->] (v1) edge [bend left=20] node {$b$} (v2);
\path[thick,->] (v2) edge node [xshift= 21pt,yshift=7pt] {$a$} (v3);
\path[thick,->] (v3) edge node [xshift=-18pt,yshift=9pt] {$a$} (v1);
\path[thick,->] (v2) edge [bend left=20] node {$b$} (v1);
\path[thick,->] (v3) edge [loop right] node {$a$} (v3);

\end{tikzpicture}
    \caption{Action of $\Sigma$ on nonzero elements of $A(+)$ in  Example~\ref{ex:sofic-system-even-number-of-b}. Note that $av_1=0$, so there is no $a$-arrow out of $v_1$. }
    \label{sofic-0004}
\end{figure}
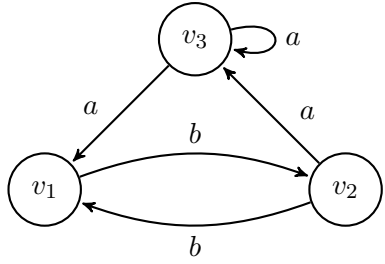

We see that $v_3=v_1+v_2$,  and $A(+)$ is a free $\Bool$-module with a basis $Q=(v_1,v_2)$. Likewise, $v_3'=v_1'+v_2'$, and $A(-)$ is a free $\Bool$-module with a basis $(v_1',v_2')$: 
\[A(+)=\Bool v_1 \oplus \Bool v_2,   \hspace{1cm}
A(-)=\Bool v_1'\oplus \Bool v_2'. 
\] 
Figure~\ref{sofic-0004} shows the action of $\Sigma=\{a,b\}$ on the three nonzero vectors  $v_1,v_2,v_3$ of $A(+)$. 
Figure~\ref{sofic-0006} shows the minimal nondeterministic automaton for $W^{\perp}$, with $v_1,v_2$ corresponding to the states. Letter $a$ acts by $0$ on $v_1$, and this is shown by a dotted arrow to the state $0$. Normally, state $0$ and arrows to it are excluded from an NFA. Here and below in this section we depict NFAs without initial and terminal states, only using them to describe sofic languages of doubly-infinite words as well as circular languages. 

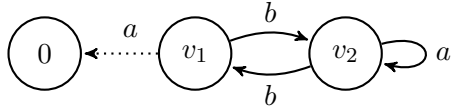
\begin{figure}
    \centering
\begin{tikzpicture}[>=stealth',shorten >=1pt,auto,node distance=2cm]
\node[thick,state] (v1) {$v_1$};
\node[thick,state] (v2) [right of=v1] {$v_2$}; 
\node[thick,state] (0) [left of=v1] {$0$};

\path[thick,->] (v1) edge [bend left=20] node {$b$} (v2);
\path[thick,->] (v2) edge [bend left=20] node {$b$} (v1);
\path[thick,->] (v2) edge [loop right] node {$a$} (v2); 
\path[thick,->,dotted] (v1) edge node [xshift=4pt,yshift=15pt] {$a$} (0);

%\begin{scope}[shift={(8,-1)}]
%\draw[thin,yellow] (0,0) grid (3,3);
%\node at (0,0) {$0$};
%\node at (-0.79,1) {$v_1$};
%\node at (0.78,1) {$v_2$};
%\node at (0,2) {$v_1+v_2$};
%\draw[thick] (0.1,0.1) -- (0.65,0.9);
%\draw[thick] (-0.65,0.9) -- (-0.1,0.1);
%\draw[thick] (0.1,1.9) -- (0.65,1.1);
%\draw[thick] (-0.65,1.1) -- (-0.1,1.9);
%\end{scope}
\end{tikzpicture}
    \caption{Minimal nondeterministic automaton for the subshift in Example~\ref{ex:sofic-system-even-number-of-b}.  $A(-)=\Bool v_1 \oplus \Bool v_2$ and $A(+)=\Bool v_1'\oplus \Bool v_2'$. 
    %Shown on the right is the distributive lattice for $A(-)$.  
    }
    \label{sofic-0006}
\end{figure}

Since the state space $A(+)$ is a free $\Bool$-module, it is also a projective $\Bool$-module, and the decomposition of the identity is given by 
\[\id_+=v_1\otimes v_1'+v_2\otimes v_2'.
\] 
The associated circular language $L_{\circ}\subset \Sigmac$ for this sofic evaluation $\alphaZ$ is described by the evaluation 
\begin{equation} \alpha_{\circ}(\omega) = \alphaZ(v_1\omega v_1')+ 
\alphaZ(v_2\omega v_2') = \alphaZ(\ast ab\omega b a\ast) + \alphaZ (\ast a \omega a \ast), \ \ \omega\in \Sigma^{\ast}. 
\end{equation} 
Language $L_{\circ}$ consists of circular words $\omega$ that do not contain any odd-length products of $b$'s surrounded by $a$'s. In particular,  $b^{2n+1}\notin L_{\circ}, n\ge 0$. 
      
This construction gives a one-dimensional accumulating defect TQFT with 
\[A(+)=\Bool Q, \ \ A(-)=A(+)^{\ast}=\Bool Q^{\ast},
\]
the infinite words language $W^{\perp}$ and circular language $L_{\circ}$ as above.

\end{example}

\begin{example} Let $\alphaZ=\SigmaZ$ for $\Sigma=\{a,b\}$. Then $W=\emptyset$ is the empty language. The state spaces $A(+)\cong A(-)\cong \Bool$ and words  $\omega \in \Sigmal$, $\omega'\in \Sigmar$ give nonzero elements $\langle \omega |\in A(+)$ and $|\omega' \rangle  \in A(-)$, respectively, for any $\omega,\omega'$. 
\end{example}

\begin{example}
\label{ex:sofic-system-000}
Let $\Sigma=\{a,b\}$, $W=\{ab,ba\}$ and $W^{\perp}=\{a^{\Z},b^{\Z} \}$ consists of two words. The state space $A(+)$ is a free rank two $\Bool$-module, with the basis $\{a^{-\infty},b^{-\infty}\}$.  The minimal NFA is shown in  Figure~\ref{sofic-0024}.  
The action of $a$, $b$ on $A(+)$ is given by 
\[  
a = 
\begin{pmatrix}
1 & 0 \\
0 & 0 \\ 
\end{pmatrix}, 
\qquad 
b = 
\begin{pmatrix}
0 & 0 \\
0 & 1 \\ 
\end{pmatrix}, 
\qquad 
\mbox{ with } ab=ba=0. 
\] 
This evaluation $\alphaZ$ extends to a TQFT with the circular language $L_{\circ} = a^{\ast}+b^{\ast}$. 

\begin{figure}
    \centering
\begin{tikzpicture}[scale=0.6]

%\begin{scope}[shift={(0,0)}, decoration={markings,mark=at position 0.5 with {\arrow{>}}}]
%\draw[thin,yellow] (0,0) grid (4,4);
%\node at (0,2) {$0$};
%\draw[thick,postaction={decorate}] (-0.25,2.35) .. controls (-3.00,3.5) and (-3.00,0.5) .. (-0.25,1.65);
%\draw[thick] (0.35,2) arc (0:360:3.5mm);
 
%\node at (3,2) {$1$};
%\draw[thick] (3.35,2.0) arc (0:360:3.5mm);
%\draw[thick,postaction={decorate}] (3.25,2.35) .. controls (6,3.5) and (6,0.5) .. (3.25,1.65);
%\end{scope}

\begin{scope}[shift={(12,0)}, decoration={markings,mark=at position 0.5 with {\arrow{>}}}]
%\draw[thin,yellow] (0,0) grid (4,4);
\node at (-1,3.0) {$a$}; 
\draw[thick,postaction={decorate}] (-0.25,2.35) .. controls (-3.00,3.5) and (-3.00,0.5) .. (-0.25,1.65);

\draw[thick] (0.35,2) arc (0:360:3.5mm);
\draw[thick,<-] (0.0,1.6) -- (0.0,1);

\draw[thick] (3.35,2.0) arc (0:360:3.5mm);
\draw[thick,<-] (3.0,1.6) -- (3.0,1);

\node at (4,3) {$b$};
\draw[thick,postaction={decorate}] (3.25,2.35) .. controls (6,3.5) and (6,0.5) .. (3.25,1.65);
\end{scope}

\end{tikzpicture}
    \caption{Minimal NFA for Example~\ref{ex:sofic-system-000}.}
    \label{sofic-0024}
\end{figure}
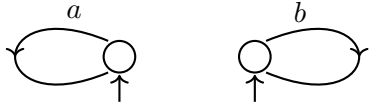

\end{example}

\begin{remark}
Given a sofic shift $X$, suppose that the state space $A(+)$ is a free $\Bool$-module. Then the minimal NFA for 
    $W^{\perp}$ is unique, with its states given by the basis of $A(+)$ and transition function described by the action of $\Sigma^{\ast}$ on $A(+)$. 
\end{remark}

\begin{remark}
 Switching to non-minimal automata describing the shift in 
 Example~\ref{ex:sofic-system-000}, 
 it is straightforward to produce TQFTs with  this $\alphaZ$ with the circular language $L_{\circ}=(a^n)^{\ast}+(b^m)^{\ast}$ and, more generally, 
 with $L_{\circ}$ given by Boolean sums of circular languages $(a^n)^{\ast}$ and $(b^m)^{\ast}$ for finite collections of $n$'s and $m$'s. 
\end{remark}

\begin{example}\label{ex_golden}
Let $\Sigma=\{a,b\}$ and $W=\{bb\}$. Then $X$ consists of infinite words that do not contain two consecutive $b$'s and is called the \emph{golden mean shift}. The state space $A(+)$ is spanned by vectors $\langle \ast a|$ and $\langle \ast ab|$, with the relation
\[
\langle \ast a| + \langle \ast ab| = \langle \ast a| . 
\]
Thus, $A(+)$ is a projective but not a free semilattice, the quotient of the free semilattice on two generators $\{x,y\}$ by the relation $x+y=x$.
The dual module $A(-)$ is spanned by $|a\ast\rangle$ and $|ba\ast \rangle$, with the relation 
\[
|a\ast\rangle + |ba\ast \rangle = |a\ast\rangle. 
\]
Since $A(+)$ is projective, decomposition of the identity exists. It is given by 
\[
\id_+ \ = \ \langle \ast a | \otimes |ba\ast\rangle + \langle \ast ab | \otimes |a \ast \rangle. 
\]
To compute the associated circular language, evaluate a finite word $\omega$ via 
\[
\omega \longmapsto \langle \ast a \omega ba\ast\rangle +  \langle \ast ab \omega a \ast \rangle.
\]
The associated circular language consists of circular words that do not contain two consecutive $b$'s: 
\begin{equation}\label{eq_golden}
L_{\circ}=\Sigma^{\ast}\setminus \left( (a+b)^{\ast}b^2(a+b)^{\ast}+b(a+b)^{\ast}b\right) .
\end{equation}

\vspace{0.07in} 

\begin{figure}
    \centering
\begin{tikzpicture}[>=stealth',shorten >=1pt,auto,node distance=2cm]
    
\node[thick,state] (one) {$1$};
\node[thick,state,right of=one] (two) {$2$};

\path[thick,->] 

(one) edge [bend left=15] node{$b$} (two)
(two) edge [bend left=15] node[xshift=0,yshift=0pt]{$a$} (one) 
(one) edge[loop left] node{$a$} (one);

\end{tikzpicture}
    \caption{A minimal NFA for the golden mean shift.}
    \label{fig_sofic-0008}
\end{figure}
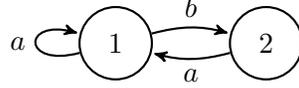

\vspace{0.07in} 

The golden mean shift consists of bi-infinite words that admit a path in the automaton in Figure~\ref{fig_sofic-0008} on the left. This graph gives rise to a TQFT with $A(+)=\Bool v_1\oplus \Bool v_2$ and the circular language given by \eqref{eq_golden}. Another minimal NFA for the same infinite language (and producing the same circular language) can be obtained by adding an $a$-loop to state $2$. Taking a finite cover of this graph or modifying it in other ways while preserving $\alphaZ$ leads to TQFTs with the same infinite evaluation function $\alphaZ$ but different circular languages. For instance, taking the cyclic $n$-cover of Figure~\ref{fig_sofic-0008} automata along the 1-2 loop gives a circular language $L'_{\circ}$ which consists of words in $L_{\circ}$ in \eqref{eq_golden} where the number of $b$'s is $0$ modulo $n$.
\end{example}

\begin{example}\label{example_nondiss}
Let $W=\{a^3,b^2\}$. The language $L_{\Z}=W^{\perp}$ consists of bi-infinite words $\omega$ that contain neither $a^3$ nor $b^2$ as a subword. The state space $A(+)$ is spanned by $\{\langle \ast b|,\langle \ast ba|,\langle\ast ba^2|\}$. The state space $A(-)$ is spanned by 
$\{|b\ast\rangle,|ab\ast\rangle,|a^2b\ast \rangle\}. $ Matrix of the bilinear form is shown in Figure~\ref{nondist_ex_001}, resulting in the following defining relations on the vectors: 
\begin{equation}
\label{eq_relations_ndiss}
\langle \ast b|+\langle \ast ba|=\langle \ast b|+\langle \ast ba^2|,  \hspace{1cm} 
\langle \ast ba|+\langle \ast ba^2| = \langle \ast ba|. 
\end{equation}
The state space $A(+)$ is a non-distributive semilattice (not a projective $\Bool$-module). Likewise $A(-)\cong A(+)^{\ast}$ is not projective. 

\begin{figure}
    \centering
\begin{tikzpicture}[scale=0.6]
\begin{scope}[shift={(0,0)}]
%\draw[thin,yellow] (0,0) grid (4,4);

\draw[thick] (0,3) -- (5.25,3);
\draw[thick] (0,2) -- (5.25,2);
\draw[thick] (0,1) -- (5.25,1);
\draw[thick] (0,0) -- (5.25,0);

\draw[thick] (0,0) -- (0,3);
\draw[thick] (1.75,0) -- (1.75,3);
\draw[thick] (3.5,0) -- (3.5,3);
\draw[thick] (5.25,0) -- (5.25,3);

\node at (0.875,2.5) {$0$};
\node at (2.625,2.5) {$1$};
\node at (4.375,2.5) {$1$};

\node at (0.875,1.5) {$1$};
\node at (2.625,1.5) {$1$};
\node at (4.375,1.5) {$0$};

\node at (0.875,0.5) {$1$};
\node at (2.625,0.5) {$0$};
\node at (4.375,0.5) {$0$};

\node at (-0.85,2.5) {$\langle *b|$};
\node at (-1.00,1.5) {$\langle *ba|$};
\node at (-1.15,0.5) {$\langle *ba^2|$};

\node at (0.875,3.5) {$|b*\rangle$};
\node at (2.625,3.5) {$|ab*\rangle$};
\node at (4.375,3.5) {$|a^2b*\rangle$};

\end{scope}
\end{tikzpicture}
    \caption{ Pairing on the generating vectors in $A(+)$ and $A(-)$ for Example~\ref{example_nondiss}, implying relations \eqref{eq_relations_ndiss}.  }
    \label{nondist_ex_001}
\end{figure}

We cannot immediately extend $A(+)$ with the corresponding action of semi-infinite words on it to a TQFT, since $A(+)$ is not projective. Instead, we can choose an automaton for the language $L_{\Z}$, see Figure~\ref{nondist_ex_002} left. 

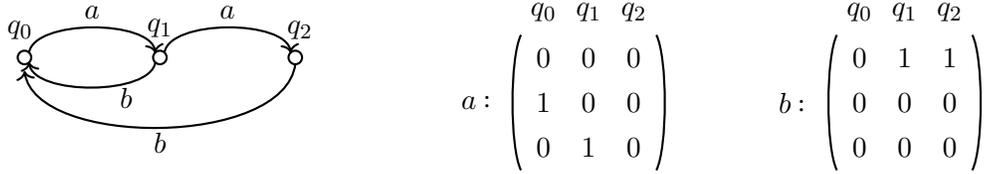
\begin{figure}
    \centering
\begin{tikzpicture}[scale=0.6]
\begin{scope}[shift={(0,0)}]
%\draw[thin,yellow] (0,0) grid (6,4);

\draw[thick] (0.15,2) arc (0:360:1.5mm);
\draw[thick] (3.15,2) arc (0:360:1.5mm);
\draw[thick] (6.15,2) arc (0:360:1.5mm);

\node at (-0.1,2.60) {$q_0$};
\node at (3,2.65) {$q_1$};
\node at (6.1,2.60) {$q_2$};

\draw[thick,->] (0.10,2.10) .. controls (0.20,2.85) and (2.80,2.85) .. (2.90,2.10);
\node at (1.5,3.00) {$a$};

\begin{scope}[shift={(3,0)}]
\draw[thick,->] (0.10,2.10) .. controls (0.20,2.85) and (2.80,2.85) .. (2.90,2.10);
\node at (1.5,3.00) {$a$};
\end{scope}

\draw[thick,<-] (0.10,1.90) .. controls (0.20,1.15) and (2.80,1.15) .. (2.90,1.90);
\node at (2.25,1.10) {$b$};

\draw[thick,<-] (0,1.70) .. controls (0.25,0) and (5.75,0) .. (6,1.85);
\node at (3,0.10) {$b$};

\end{scope}

\begin{scope}[shift={(10,-1)}]
%\draw[thin,yellow] (0,0) grid (6,4);

%\node at (6,5) {left-infinite words};

\node at (0,2) {$a:$};

\draw[thick] (1,0.5) .. controls (0.75,1) and (0.75,3) .. (1,3.5);
\draw[thick] (4,0.5) .. controls (4.25,1) and (4.25,3) .. (4,3.5);

\node at (1.5,4) {$q_0$};
\node at (2.5,4) {$q_1$};
\node at (3.5,4) {$q_2$};

\node at (1.5,3) {$0$};
\node at (1.5,2) {$1$};
\node at (1.5,1) {$0$};

\node at (2.5,3) {$0$};
\node at (2.5,2) {$0$};
\node at (2.5,1) {$1$};

\node at (3.5,3) {$0$};
\node at (3.5,2) {$0$};
\node at (3.5,1) {$0$};

\end{scope}

\begin{scope}[shift={(17,-1)}]
%\draw[thin,yellow] (0,0) grid (6,4);

\node at (0,2) {$b:$};

\draw[thick] (1,0.5) .. controls (0.75,1) and (0.75,3) .. (1,3.5);
\draw[thick] (4,0.5) .. controls (4.25,1) and (4.25,3) .. (4,3.5);

\node at (1.5,4) {$q_0$};
\node at (2.5,4) {$q_1$};
\node at (3.5,4) {$q_2$};

\node at (1.5,3) {$0$};
\node at (1.5,2) {$0$};
\node at (1.5,1) {$0$};

\node at (2.5,3) {$1$};
\node at (2.5,2) {$0$};
\node at (2.5,1) {$0$};

\node at (3.5,3) {$1$};
\node at (3.5,2) {$0$};
\node at (3.5,1) {$0$};
\end{scope}

\end{tikzpicture}
    \caption{Left: an automaton with a  non-distributive language. Right: action of $a$ and $b$ in the basis $\{q_0,q_1,q_2\}$ of $\mcF(+)$.}
    \label{nondist_ex_002}
\end{figure}

The TQFT associated with this automaton has $\mcF(+)=\Bool q_0\oplus \Bool q_1\oplus\Bool q_2$, and the circular language $L_{\circ}$ of this TQFT consists of circular words that do not contain subwords $a^3$ or $b^2$. It can also be written as 
\[
L_{\circ}=\Sigma^{\ast}\setminus
\left\{
\Sigma^{\ast}a^3 \Sigma^{\ast} \cup \Sigma^{\ast}b^2\Sigma^{\ast}\cup b\Sigma^{\ast}b\cup a^2\Sigma^{\ast} a\cup a\Sigma^{\ast}a^2 
\right\}. 
\]
Taking a finite cover of this automaton, for example as in Figure~\ref{nondist_ex_003}, results in a different TQFT with the same bi-infinite language $L_{\Z}$ but a different circular language. For the TQFT  associated to the two-sheeted cover automaton in Figure~\ref{nondist_ex_003}, the circular language $L_{\circ}'$ is a subset of $L_{\circ}$ which consists of circular words with even number of appearances of $a^2$ between consecutive $b$'s (and any number of appearances of $a$ between consecutive $b$'s).

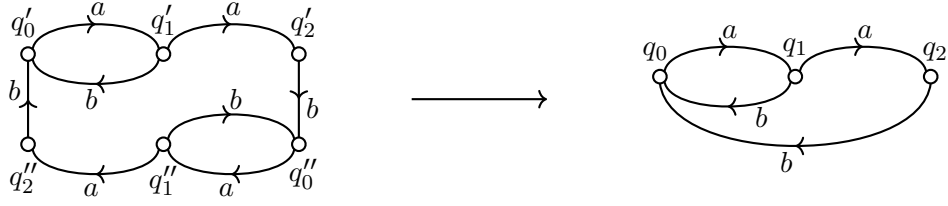
\begin{figure}
    \centering
\begin{tikzpicture}[scale=0.6,decoration={markings,mark=at position 0.5 with {\arrow{>}}}]
\begin{scope}[shift={(14,-0.5)}]
%\draw[thin,yellow] (0,0) grid (6,4);

\draw[thick] (0.15,2) arc (0:360:1.5mm);
\draw[thick] (3.15,2) arc (0:360:1.5mm);
\draw[thick] (6.15,2) arc (0:360:1.5mm);

\node at (-0.1,2.55) {$q_0$};
\node at (3,2.65) {$q_1$};
\node at (6.1,2.55) {$q_2$};

\draw[thick,postaction={decorate}] (0.10,2.10) .. controls (0.20,2.85) and (2.80,2.85) .. (2.90,2.10);
\node at (1.55,3.00) {$a$};

\begin{scope}[shift={(3,0)}]
\draw[thick,postaction={decorate}] (0.10,2.10) .. controls (0.20,2.85) and (2.80,2.85) .. (2.90,2.10);
\node at (1.55,3.00) {$a$};
\end{scope}

\draw[thick,postaction={decorate}] (2.90,1.90) .. controls (2.80,1.15) and (0.20,1.15) .. (0.10,1.90);
\node at (2.25,1.10) {$b$};

\draw[thick,postaction={decorate}] (6,1.85) .. controls (5.75,0) and (0.25,0) .. (0,1.85);
\node at (2.8,0.10) {$b$};
\end{scope}

\begin{scope}[shift={(0,0)}]
% this is the cover 
%\draw[thin,yellow] (0,0) grid (6,3);

\draw[thick] (0.15,2) arc (0:360:1.5mm);
\draw[thick] (3.15,2) arc (0:360:1.5mm);
\draw[thick] (6.15,2) arc (0:360:1.5mm);

\node at (-0.1,2.70) {$q_0'$};
\node at (3,2.75) {$q_1'$};
\node at (6.1,2.70) {$q_2'$};

\draw[thick,postaction={decorate}] (0.10,2.10) .. controls (0.20,2.85) and (2.80,2.85) .. (2.90,2.10);
\node at (1.55,3.00) {$a$};

\begin{scope}[shift={(3,0)}]
\draw[thick,postaction={decorate}] (0.10,2.10) .. controls (0.20,2.85) and (2.80,2.85) .. (2.90,2.10);
\node at (1.55,3.00) {$a$};
\end{scope}

\draw[thick,postaction={decorate}] (2.90,1.90) .. controls (2.80,1.15) and (0.20,1.15) .. (0.10,1.90);
\node at (1.4,1.00) {$b$};

\draw[thick,postaction={decorate}] (6,1.85) -- (6,0.15);
\draw[thick,postaction={decorate}] (0,0.15) -- (0,1.85);

\node at (-0.30,1.2) {$b$};
\node at ( 6.30,0.8) {$b$};

\begin{scope}[shift={(0,-2)}]
\draw[thick] (0.15,2) arc (0:360:1.5mm);
\draw[thick] (3.15,2) arc (0:360:1.5mm);
\draw[thick] (6.15,2) arc (0:360:1.5mm);

\draw[thick,postaction={decorate}] (2.90,1.90) .. controls (2.80,1.15) and (0.20,1.15) .. (0.10,1.90);
\node at (1.4,1.00) {$a$};

\draw[thick,postaction={decorate}] (5.90,1.90) .. controls (5.80,1.15) and (3.20,1.15) .. (3.10,1.90);
\node at (4.4,1.00) {$a$};

\begin{scope}[shift={(3,0)}]
\draw[thick,postaction={decorate}] (0.10,2.10) .. controls (0.20,2.85) and (2.80,2.85) .. (2.90,2.10);
\node at (1.6,3.00) {$b$};
\end{scope}

\node at (-0.1,1.25) {$q_2''$};
\node at (3,1.20) {$q_1''$};
\node at (6.1,1.35) {$q_0''$};
\end{scope}

%\draw[thin,gray] (0,0) grid (6,-4);
\draw[thick,->] (8.5,1) -- (11.5,1);
%\node at () {};  this is the label for the cover map

\end{scope}

\end{tikzpicture}
    \caption{A two-sheeted cover of the automaton in Figure~\ref{nondist_ex_002}, with $q_i',q_i''$ mapped to $q_i$, $i=0,1,2$.}
    \label{nondist_ex_003}
\end{figure}

\end{example}

%%%%%%%%%%%%%%%%%%%%%
%
% One side accumulation and omega-automata
%
%%%%%%%%%%%%%%%%%%%%%

\section{Defects accumulating on one side and \texorpdfstring{$\omega$}{w}-automata}\label{sec_one_side}

%%%%%%%%%%%%%%%%%%%%%%%%%
% Buchi automata TQFTs 
%%%%%%%%%%%%%%%%%%%%%%%%%

\subsection{\texorpdfstring{$\omega$}{omega}-automata and TQFTs}\label{subsec_omega}

 \quad 
 
 \noindent 
{\it B\"uchi automata and TQFTs.}
We now consider a nondeterministic B\"uchi automaton $(Q)$ $=$ $(Q,\delta,Q_{\init},Q_{\ac})$, where $Q$ is a finite set (of states of $(Q)$), $\delta: Q\times \Sigma\lra \mathcal{P}(Q)$ is a transition function, $Q_{\init}\subset Q$ is a subset of initial states and $Q_{\ac}\subset Q$ is an acceptance condition subset. B\"uchi automaton $(Q)$ accepts a right-infinite word $\omega=a_1a_2\cdots $, $\omega\in \Sigma^{\h}$ if there exists an infinite path $\omega$ in the graph of $(Q)$ which starts in a state in $Q_{\init}$ and goes infinitely many times through at least one of the states in $Q_{\ac}$. 

To $(Q)$ we associate a one-dimensional TQFT by considering a version of the decorated one-dimensional cobordism category as follows. 

Define $\Cob^{\h}_{\Sigma}$ to be the category of one-dimensional oriented cobordisms with $\Sigma$-labelled defects and inner endpoints and additionally require that at each inner endpoint oriented \emph{out} of the cobordism there is an infinite countable sequence of accumulating defects, see Figure~\ref{fig_sofic-0009}. 

\vspace{0.1in} 

\begin{figure}
    \centering
\begin{tikzpicture}[scale=0.6]
\begin{scope}[shift={(0,0)},decoration={markings,mark=at position 0.70 with {\arrow{>}}}]
%\draw[thin,yellow] (0,0) grid (13,6);
\draw[thick,dashed] (0,6) -- (13,6);
\draw[thick,dashed] (0,0) -- (13,0);

\node at (2,6.5) {$-$};
\node at (4,6.5) {$+$};
\node at (6,6.5) {$+$};

\node at (2,-0.5) {$+$};
\node at (4,-0.5) {$+$};
\node at (6,-0.5) {$-$};
\node at (8,-0.5) {$+$};

\draw[thick] (0.5,3) .. controls (0.5,3.5) and (2,4) .. (2,5);
\draw[thick,<-] (2,5) -- (2,6);

\draw[thick,fill] (1.0,3.50) arc (0:360:1.5mm);
\draw[thick,fill] (1.85,4.25) arc (0:360:1.5mm);
\draw[thick,fill] (2.10,4.75) arc (0:360:1.5mm);

\node at (2.30,3.85) {$\omega_1$};
\node at (1.50,3.40) {\rotatebox[origin=c]{36}{$\cdots$}};

\draw[thick,postaction={decorate}] (2,0) .. controls (1.5,3) and (4.5,4) .. (4,6);

\draw[thick,fill] (2.35,2) arc (0:360:1.5mm);
\node at (2.9,2) {$\omega_3$};

\draw[thick,postaction={decorate}] (4,0) .. controls (4.2,2) and (5.8,2) .. (6,0);

\draw[thick,fill] (5.15,1.5) arc (0:360:1.5mm);
\node at (5,2) {$\omega_6$};

\draw[thick,->] (8,0) .. controls (8.25,1) and (6.25,2) .. (6.5,3.5);

\draw[thick,fill] (6.68,2.75) arc (0:360:1.5mm);
\draw[thick,fill] (7.00,2.10) arc (0:360:1.5mm);
\draw[thick,fill] (7.35,1.60) arc (0:360:1.5mm);

\draw[thick,fill] (7.75,1.10) arc (0:360:1.5mm);
 
\node at (7.1,2.7) {\rotatebox[origin=c]{-65}{$\cdots$}};
\node at (7.82,1.9) {$\omega_7$};

\draw[thick,postaction={decorate}] (4.5,3) .. controls (4.5,4) and (6,5) .. (6,6);
\draw[thick,fill] (5.05,4.0) arc (0:360:1.5mm);
\node at (5.5,3.75) {$\omega_2$};

\draw[thick,postaction={decorate}] (8,4.0) .. controls (9,5) and (11,4.75) .. (12,5.25);

\draw[thick,fill] (9.5,4.65) arc (0:360:1.5mm);
\draw[thick,fill] (10.15,4.81) arc (0:360:1.5mm);

\draw[thick,fill] (11.5,5.05) arc (0:360:1.5mm);

\node at (10.18,4.23) {$\omega_4$};
\node at (11.15,4.4) {\rotatebox[origin=c]{10}{$\cdots$}};

\begin{scope}[shift={(1.75,-0.5)}]
\draw[thick,<-] (9,2) arc (-90:270:1);
\draw[thick,fill] (10.15,3) arc (0:360:1.5mm);
\node at (10.75,3) {$\omega_5$};
\end{scope}

\end{scope}

\end{tikzpicture}
    \caption{Words $\omega_1,\omega_4,\omega_7\in \Sigma^{\h}$ are right-infinite, with defects accumulating towards out-oriented (sink) inner endpoints. Words $\omega_2,\omega_3,\omega_5,\omega_6\in \Sigma^*$ are finite words, each depicted by a single dot.}
    \label{fig_sofic-0009}
\end{figure}
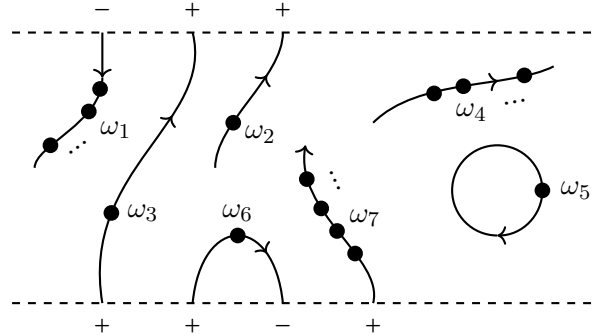

\vspace{0.1in} 

At an inner endpoint oriented \emph{into} the cobordism there are only finitely many defects, so that defects accumulate only towards \emph{out} inner endpoints. Category $\Cob^{\h}_{\Sigma}$ is a symmetric tensor category.  

\begin{prop}\label{prop_buchi}
    A B\"uchi automaton $(Q)$ determines a Boolean TQFT 
    \begin{equation}
        \mcF_{(Q)} \ : \ \Cob^{\h}_{\Sigma} \lra \Bool\mathrm{-fmod}
    \end{equation}
    valued in the category of free $\Bool$-semimodules. 
\end{prop}

\begin{proof}
    The proof consists of the construction of that TQFT, which we denote $\mcF$ here instead of $\mcF_{(Q)}$. Most of the construction matches the one in Section~\ref{subsec_qreview} where to an automaton there is associated a functor \eqref{eq_TQFT_B} from $\Cob_{\Sigma}$ to $\Bool\mathsf{-fmod}$. 
    
    Let 
    $\mcF(+)=\Bool Q$ and $\mcF(-)=\Bool Q^{\ast}$, with the cup and cap maps given as in the automata case, 
    \[
    \mcF(\cup) = \sum_{q\in Q} q\otimes q^{\ast}, \ \ 
    \mcF(\cap) (q_1\otimes q_2^{\ast}) = \delta(q_1,q_2). 
    \]
    Note that $\delta(q_1,q_2)$ stands for the $\Bool$-valued Dirac delta-function, while $\delta_a: Q\lra \mathcal{P}(Q)$, $a\in \Sigma$ denotes the transition function for $(Q)$. 

     To an \emph{in} oriented interval, viewed as a cobordism from the empty $0$-manifold $\emptyset_0$ to $+$, assign homomorphism $\Bool \lra \Bool Q = \mcF(+)$ given by $1\longmapsto Q_{\init}$. To a dot labelled $a\in \Sigma$ on an upward strand assign the linear map $m_a: \Bool Q \lra \Bool Q$ induced by the transition function $\delta_a$. 

     To an infinite sequence of dots $\omega=a_1a_2\cdots $ on a downward strand assign a linear map $\gamma_{\omega}: \Bool Q\lra \Bool$, where $\gamma_{\omega}(q)=1$ if and only if there is an infinite path $\omega$ that starts at $q$ and goes through a state in $Q_{\ac}$ infinitely many times. Otherwise $\gamma_{\omega}(q)=0$. The linear map $\gamma_{\omega}\in \Bool Q^{\ast}$, and functor $\mcF$ takes the sequence $\omega$ on a half-interval pointing down to $f_{\omega}$, see Figure~\ref{fig_sofic-0010}. 

\vspace{0.1in} 

\begin{figure}
    \centering
\begin{tikzpicture}[scale=0.6]
\begin{scope}[shift={(0,0)}]
%\draw[thin,yellow] (0,0) grid (4,4);
\draw[thick,dashed] (0,4) -- (2,4);
\draw[thick,dashed] (0,0) -- (2,0);
\node at (1,-0.5) {$+$};

\draw[thick,->] (1,0) -- (1,3);

\node at (0.25,2.25) {$\vdots$};
\node at (0.25,1.5) {$a_3$};
\node at (0.25,1.0) {$a_2$};
\node at (0.25,0.5) {$a_1$};

\node at (1.75,1.25) {$\omega$};

\draw[thick,fill] (1.15,2.25) arc (0:360:1.5mm);

\draw[thick,fill] (1.15,1.5) arc (0:360:1.5mm);

\draw[thick,fill] (1.15,1.0) arc (0:360:1.5mm);

\draw[thick,fill] (1.15,0.5) arc (0:360:1.5mm);

\begin{scope}[shift={(1,0)}]
\node at (2.50,4) {$\Bool$};
\node at (2.75,0) {$\Bool Q$};

\draw[thick,->] (2.50,0.5) -- (2.50,3.5);
\node at (3.00,2) {$\gamma_{\omega}$};
\end{scope}

\end{scope}

\begin{scope}[shift={(10,0)}]
%\draw[thin,yellow] (0,0) grid (4,4);

\draw[thick,dashed] (0,4) -- (2,4);
\draw[thick,dashed] (0,0) -- (2,0);
\node at (1,4.5) {$+$};

\draw[thick,->] (1,1) -- (1,4);

\node at (0.25,3.5) {$a_n$};
\node at (0.25,2.75) {$\vdots$};
\node at (0.25,2.0) {$a_2$};
\node at (0.25,1.5) {$a_1$};

%\node at (1.75,2.5) {$\omega$};

\draw[thick,fill] (1.15,3.5) arc (0:360:1.5mm);

\draw[thick,fill] (1.15,2.75) arc (0:360:1.5mm);

\draw[thick,fill] (1.15,2.0) arc (0:360:1.5mm);

\draw[thick,fill] (1.15,1.5) arc (0:360:1.5mm);

\begin{scope}[shift={(1,0)}]
\node at (2.75,4) {$\Bool Q$};
\node at (2.50,0) {$\Bool$};

\draw[thick,->] (2.50,0.5) -- (2.50,3.5);
%\node at (3.00,2) {$\gamma_{\omega}$};
\end{scope}

\node at (5,2) {$=$};
\end{scope}

\begin{scope}[shift={(16,0)}]
%\draw[thin,yellow] (0,0) grid (4,4);

\draw[thick,dashed] (0,4) -- (2,4);
\draw[thick,dashed] (0,0) -- (2,0);
\node at (1,4.5) {$+$};

\draw[thick,->] (1,1) -- (1,4);

\node at (1.75,2.50) {$\omega$};

\draw[thick,fill] (1.15,2.50) arc (0:360:1.5mm);

\begin{scope}[shift={(1,0)}]
\node at (2.75,4) {$\Bool Q$};
\node at (2.50,0) {$\Bool$};

\draw[thick,->] (2.50,0.5) -- (2.50,3.5);
%\node at (3.00,2) {$\gamma_{\omega}$};
\end{scope}

\end{scope}

\end{tikzpicture}
    \caption{Left: an infinite word $\omega=a_1a_2\dots $ defines an element of $\mcF(-)\cong \Bool Q^{\ast}$ and a $\Bool$-linear functional $\gamma_\omega:\Bool Q\lra \Bool$.  Right: a finite word $\omega=a_1a_2\cdots a_n$  defines an element $\langle \omega |$ of $\mcF(+)=\Bool Q$.}
    \label{fig_sofic-0010}
\end{figure}

\vspace{0.1in} 

Note that the only way to produce an infinite sequence of defects on a connected component of a cobordism in $\Cob^{\h}_{\Sigma}$ is at an outward-oriented inner endpoint. In particular, a circle  carries only a finite sequence of defects (possibly empty). 

A floating circle with a finite sequence $\omega=a_1\cdots a_n$ evaluates to the trace of the corresponding endomorphism of $\Bool Q$, thus to $1$ if and only if there is a closed path in $(Q)$ for the sequence $\omega$. 

A floating interval necessarily carries a right-infinite sequence $\omega=a_1a_2\cdots\in \Sigma^{\h}$. From our construction, it is clear that it has a well-defined evaluation, independent of how the word (the floating interval) is chopped into several intervals with finitely-many defects and one interval with infinitely-many defects, when presenting the floating interval as the composition of elementary morphisms in $\Cob^{\h}_{\Sigma}$. Namely floating interval $\omega$  evaluates to $1$ if and only if the word $\omega$ is accepted by the B\"uchi automaton $Q$. 
\end{proof}

\begin{remark}  Reversing orientation of cobordisms results in the category $\Cob^{\t}_{\Sigma}$ where defects accumulate towards \emph{in-oriented} inner endpoints, not \emph{out-oriented}  inner endpoints. Monoidal categories $\Cob^{\t}_{\Sigma}$ and $\Cob^{\h}_{\Sigma}$ are isomorphic, via the involutive isomorphism that reverses signs of sequences and orientations of one-manifolds. A B\"uchi automaton $(Q)$ defines a TQFT 
$\mcF_{(Q)}: \Cob^{\t}_{\Sigma}\lra \Bool\mathrm{-fmod}$ as well, by passing to dual semimodules and left-infinite instead of right-infinite words. 
\end{remark}

\noindent 
{\it Types of $\omega$-automata.}
A \emph{run} $\rho$ of an automaton on semi-infinite words is a sequence of states $q_0,q_1,\ldots$ such that $q_0\in Q_{\init}$ and $q_{i+1}\in \delta_{a_i}(q_i)\subset Q$. Denote by $\mInf(\rho)\subset Q$ the set of states which appear in the run $\rho$ infinitely many times. 

Non-deterministic B\"uchi automata constitute one flavour of $\omega$-automata, that is, automata on semi-infinite words. We recall other types of $\omega$-automata, which use the following acceptance conditions:  
\begin{itemize}
    \item  {\it Rabin automaton:} For some set $\Omega$ of pairs $(B_i,G_i)$ of sets of states the automaton accepts a run $\rho$ if and only if there exists a pair $(B_i,G_i)$ in $\Omega$ such that $B_i \cap \mInf(\rho)$ is empty and $G_i \cap \mInf(\rho)$ is not empty.
    \item {\it Streett automaton:}
    For some set $\Omega$ of pairs $(B_i,G_i)$ of sets of states the automaton accepts a run $\rho$ if and only if for all pairs $(B_i,G_i)$ in $\Omega$ 
    such that $B_i \cap \mInf(\rho)$ is empty or $G_i \cap \mInf(\rho)$ is not empty.
    \item {\it A Muller automaton:} For a subset $F$ of $\mathcal{P}(Q)$ the automaton accepts exactly those runs for which $\mInf(\rho)$ is an element of $F$. 
\end{itemize}
B\"uchi, Rabin and Streett automata are special cases of Muller automata. It is an early result from the theory of $\omega$-automata that nondeterministic B\"uchi, Rabin, Street and Muller automata all recognize the same class of languages on semi-infinite words, known as the class of \emph{regular $\omega$-languages} (see~\cite{Wilke21,BC_Toolbox17}, for instance). Regular $\omega$-languages are exactly those recognized by \emph{$\omega$-regular expressions}, which are expressions of the form $r_0\cdot s_0^{\omega}+\ldots +r_{n-1}\cdot s_{n-1}^{\omega}$.  
Here $r_i,s_i$, $0\le i\le n-1$, are regular expressions and $s^{\omega}$, for a regular expression $s$, is the $\omega$-language consisting of infinite words $w_1w_2\cdots$ with each $w_j$, $j\ge 1$, a word in the language of the regular expression $s$. 

\begin{prop}
    A Muller automaton $(Q)$ determines a Boolean TQFT 
    \begin{equation}
        \mcF_{(Q)} \ : \ \Cob^{\h}_{\Sigma} \lra \Bool\mathrm{-fmod}
    \end{equation}
    valued in the category of free $\Bool$-semimodules. 
\end{prop}

Proof and constructions of Proposition~\ref{prop_buchi} extend to Muller automata in a straightforward way. 
$\square$

\vspace{0.07in} 

In the automata -- Boolean TQFT correspondence to the set $Q$ of states of an automaton there is associated the free Boolean module $\Bool Q$, which is a building block for the associated TQFT and the state space $\mcF(+)$ of a positively oriented 0-manifold. The notion of a quasi-automaton is briefly discussed in~\cite[Section 4.3]{GIKKL23}, generalizing the correspondence from free to projective $\Bool$-modules. 

Analogously, define a B\"uchi quasi-automaton $(P)=(P,\{m_a\}_{a\in \Sigma},q_{\init},b)$ to be a finitely-generated projective $\Bool$-module $P$, with an action of $\Sigma^{\ast}$ on $P$ given by $\Bool$-linear endomorphisms $m_a:P\lra P$, initial vector $q_{\init}\in P$ and a covector $b:P\lra \Bool$. To $(P)$ assign a language $L\subset \Sigma^{\h}$ as follows. Any finite word $\omega=a_1\cdots a_n\in\Sigma^{\ast}$ defines a vector 
\[q_{\init}\omega \in P. 
\]
A right-infinite word $\omega=a_1a_2\cdots$ is defined to be in $L$ if and only if $b(q_{\init}a_1\cdots a_n)=1$ for infinitely many $n>0$. Functional $b$ replaces the accepting subset $Q_{\ac}$ of $Q$ in the definition of a B\"uchi automaton in which the path must land infinitely many times.

\begin{prop}
    An $\omega$-language $L$ is accepted by a B\"uchi quasi-automaton if and only if it is accepted by a deterministic B\"uchi automaton.  
\end{prop}
\begin{proof} 
B\"uchi quasi-automaton $(P)$ gives rise to a deterministic B\"uchi automaton $D_{(P)}$ with elements of $\Bool$-module $P$ as states and deterministic transition function. The set of accepting states consists of elements $x\in P$ with $b(x)=1$. Deterministic B\"uchi automaton $D_{(P)}$ defines the same $\omega$-language as $(P)$. Vice versa, a deterministic B\"uchi automaton $(Q)$ gives rise to a B\"uchi quasi-automaton with $\Bool$-module $P=\Bool Q$ the free module on the states of $Q$ and function $b:P\lra \Bool$ given by $b(q)=1$ if and only if $q$ is in the accepting set for $(Q)$. The B\"uchi quasi-automaton $(P)$ defines the same language as $(Q)$. 
\end{proof} 

Note that the $\omega$-language $L\subset \{a,b\}^{\omega}$ which consists of words with finitely many $a$'s is recognized by a non-deterministic B\"uchi automaton but cannot be recognized by a deterministic B\"uchi automaton, see~\cite[Example 6.2]{PePi04} and \cite[Lecture 7]{Kumar_notes} for instance. 

\begin{prop}\label{prop_buchi_quasi}
 A B\"uchi quasi-automaton $(P)$ determines a Boolean TQFT 
    \begin{equation}
        \mcF_{(P)} \ : \ \Cob^{\h}_{\Sigma} \lra \Bool\mathrm{-pmod}
    \end{equation}
    valued in the category of projective $\Bool$-semimodules.  
\end{prop}
Functor $\mcF_{(P)}$ associates $P$ to the positive 0-manifold and $P^{\ast}$ to the negative 0-manifold. The rest of the construction of $\mcF_{(P)}$ runs in parallel to that in the proof of Proposition~\ref{prop_buchi}. $\square$

\vspace{0.07in} 

\begin{remark}
Note that any $n$-dimensional TQFT functor into the category of semimodules over a commutative semiring $R$ takes $(n-1)$-manifolds to finitely-generated projective semimodules, see~\cite{GIKKL23}. A semimodule is finitely-generated projective if and only if it is a retract of the free semimodule $R^k$ for some $k$.
\end{remark}

\vspace{0.1in} 

The definition of B\"uchi automata above includes a set $Q_{\ac}\subset Q$ of accepting states. The language $L_{\h}$ associated to $(Q)$ consists of  semi-infinite word $\omega$  for which there exists a path in $(Q)$ that spells $\omega$, starts in an initial state, and goes infinitely many times through a state in $Q_{\ac}$. 

\begin{prop} \label{prop_closed} If $Q_{\ac}=Q$, i.e., every state is accepting, then $L_{\h}\subset \Sigma^{\h}$ is a closed subset.  
\end{prop}

Here $\Sigma^{\h}$ has the standard Cantor set topology. A proof is straightforward. $\square$  

\vspace{0.07in} 

\begin{remark}\label{remark_not_closed}
Consider the automaton $(Q)$ in Figure~\ref{sofic-0011} with the accepting set $Q_{\ac}=\{q_0\}$. This automaton appears in~\cite[Section 3.2]{GIKKL23}. One can form the circular language $L_{\circ}$ that consists of closed paths in this automaton that go through state $q_0$. This circular language is different from the circular language $L_{\circ}'$ which consists of all closed paths in $(Q)$, and 
\begin{equation}
\label{eq_L_circ}
 L_{\circ}' \ = \ L_{\circ} \cup a^{\ast}.
\end{equation}
Note that $a^n$, $n\ge 1$, is in $L_{\circ}'$ but not in $L_{\circ}$. 

\begin{figure}
    \centering
\begin{tikzpicture}[scale=0.6]
\begin{scope}[shift={(0,0)}]
%\draw[thin,yellow] (0,0) grid (4,4);

\draw[thick] (0.15,2) arc (0:360:1.5mm);
\draw[thick,fill] (4.15,2) arc (0:360:1.5mm);

\draw[thick,<-] (0.25,2.25) .. controls (0.5,3.25) and (3.5,3.25) .. (3.75,2.25);

\draw[thick,->] (0.25,1.75) .. controls (0.5,0.75) and (3.5,0.75) .. (3.75,1.75);

\draw[thick,->] (-0.25,2.40) arc (30:330:0.75);

\node at (2,3.40) {$b$};
\node at (2,0.60) {$b$};

\node at (-2.05,2) {$a$};

\node at (0.1,1.15) {$q_1$};

\node at (4.65,2) {$q_0$};

\draw[thick,->] (0,3.20) -- (0,2.5);
\draw[thick,->] (4,3.20) -- (4,2.5);

%\node at (10,2.5) {$Q'=\{ q_0,q_1\}$};
%\node at (9.5,1.5) {$Q'' = \{ q_0\}$};

\end{scope}

\end{tikzpicture}
    \caption{Two-state automaton $(Q)$ with the accepting set $Q_{\ac}=\{q_0\}$.}
    \label{sofic-0011}
\end{figure}
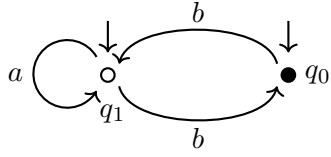

It is easy to see that $L_{\circ}$ can not be the language of circular paths of any FSA~\cite{GIKKL23}, and adding the condition that a circular path goes through a particular subset $Q_{\ac}$ of $Q$ enlarges possible circular languages one can associate with FSAs. 

The same automaton can be used for an example of an $\omega$-language and  a motivation for the B\"uchi condition. Treat the automaton in Figure~\ref{sofic-0011} as a B\"uchi automaton with $Q_{\init}=\{q_0,q_1\}$, $Q_{\ac}=\{q_0\}$. Then the language of semi-infinite paths  that go through $q_0$ infinitely many times is 
\begin{equation}\label{eq_lang_Q}
L_{\h} = (\emptyset + b)(a^{\ast}b^2(b^2)^{\ast})^{\omega},
\end{equation}
while the language of all semi-infinite paths is 
\[
L_{\h}' = (\emptyset + b)(a^{\ast}b^2(b^2)^{\ast}+a^{\ast})^{\omega},
\]
and it contains words with finitely many $b$'s as well. 
Note that $L_{\h}'$ is closed in $\Sigma^{\h}$, for its complement is the union of basis open sets 
$\Sigma^{k}ab(b^{2n})a\Sigma^{\omega}$, over $k,n\ge 0$.  Language $L_{\h}$ is not closed in $\Sigma^{\h}$. In particular, not all $\omega$-regular languages are closed subsets of $\Sigma^{\h}$, see also Proposition~\ref{prop_closed}.  

\end{remark}

%%%%%%%%%%%%%%%%%%
% Semi-infinite evaluations 
%%%%%%%%%%%%%%%%%%

\subsection{Topological theories from semi-infinite evaluations}

Suppose given an $\omega$-evaluation $\alpha_{\h}: \Sigma^{\h}\lra \Bool$, that is, a map from the set of right-infinite words to $\Bool$. Define $\Bool$-modules $A(-), A(+)$ using the concatenation  map $\Sigma^{\ast}\times \Sigma^{\h}\lra \Sigma^{\h}$ as follows. Start with free $\Bool$-modules $\mathsf{Fr}(+)=\Bool^{\Sigma^{\ast}},\mathsf{Fr}(-)=\Bool^{\Sigma^{\h}}$ with a basis $\{\langle \omega| \}_{\omega\in \Sigma^{\ast}}$ of finite words, respectively a basis of right-infinite words $\{|\omega'\rangle \}_{\omega'\in \Sigma^{\h}}$. Concatenation composed with evaluation 
\[
\Fr(+)\times \Fr(-)\lra \Bool, \ \ \langle \omega | \times |\omega'\rangle \lra \alpha_{\h}(\omega\omega') \in \Bool 
\]
is a bilinear pairing $\Fr(+)\otimes \Fr(-)\lra \Bool$. Define $A(+)$ and $A(-)$ to be quotients by the kernels of this pairing. Working over the Boolean semirings, they need to be defined set-theoretically: $A(+):=\Fr(+)/\sim$, where $\sum_i \langle\omega_i| \sim \sum_j \langle\omega_j|$ for $\omega_i,\omega_j\in \Sigma^{\ast}$ if and only if for any $\omega'\in \Sigma^{\h}$ 
\begin{equation}\label{eq_equiv_1}
 \sum_i  \alpha_{\h}(\omega_i \omega') = \sum_j \alpha_{\h}(\omega_j \omega')\in \Bool.
\end{equation}
     
Likewise, $A(-):=\Fr(-)/\sim$, where 
where $\sum_i |\omega'_i\rangle \sim \sum_j |\omega'_j\rangle $ for $\omega'_i,\omega'_j\in \Sigma^{\h}$ if and only if for any $\omega\in \Sigma^{\ast}$ 
\begin{equation}\label{eq_equiv_2}
 \sum_i  \alpha_{\h}(\omega \omega'_i) = \sum_j \alpha_{\h}(\omega \omega'_j).
\end{equation}
 $A(+),A(-)$ are $\Bool$-semimodules. Semimodule $A(+)$ is finite if and only if $A(-)$ is, in which case they are dual semimodules and have the same cardinality. 

Evaluation $\alpha_{\h}$ and the associated language $L_{\h}$ such that $A(+)$ is finite are called an \emph{$\omega$-finite evaluation} and an \emph{$\omega$-finite language}, see~\cite{Sta83} and references there. 

\vspace{0.07in} 

If, in addition, $A(+)$ is a projective semimodule (a retract of $\Bool^n$, for some $n$), then $A(+)$ with the  action of $\Sigma$ on it and evaluation $\alpha_{\h}$ gives rise to a TQFT 
\begin{equation}
        \mcF_{\alpha_{\h}} \ : \ \Cob^{\h}_{\Sigma} \lra \Bool\mathrm{-pmod},
\end{equation}
similar to the constructions explained earlier. 

\vspace{0.1in}

We do not attempt to state and prove the analogue of Theorem~\ref{theorem:regular_sofic} for semi-infinite languages, since the analogue of the condition that $\alphaZ^{-1}(1)$ is closed appears to be rather subtle for $\omega$-regular languages, see~\cite{Sta83,PePi04,Sta97} and references therein.

%%%%%%%%%%%%%%%%%
%
% Examples 
%
%%%%%%%%%%%%%%%%%

\section{Examples of topological theories for infinite automata} \label{sec_examples_automata}

In some examples below we choose a finite or regular language $W$ and define $L=W^{\perp}_{\h}$ to consist of semi-infinite words that do not contain an element of $W$ as a subword. Languages $L$ are $\omega$-regular, but notice that, unlike for bi-infinite words, not every $\omega$-regular language has this form.

%%%%%%%%%%%%%%%%
% Example 0; check it
%%%%%%%%%%%%%%%%

\begin{example}
\label{ex:sofic-system-0000}
Let $\Sigma=\{ a\}$ and $L = \{ a^{+\infty} \}=\Sigma^{\h}$ consists of the unique possible semi-infinite word. Then $A(+)=\Bool \langle \varnothing|$, $A(-)=\Bool | a^{+\infty}\rangle$ are free rank one $\Bool$-modules, with the minimal automaton for this $\omega$-evaluation shown in Figure~\ref{sofic-0025} on the left, with the unique state being accepting. The identity decomposition is $\id_+=|a^{+\infty} \rangle \otimes \langle\emptyset |$.  

\begin{figure}
    \centering
\begin{tikzpicture}[scale=0.6]
\begin{scope}[shift={(0,0)}, decoration={markings,mark=at position 0.5 with {\arrow{>}}}]
%\draw[thin,yellow] (0,0) grid (4,4);
\node at (-1,3.0) {$a$}; 
\draw[thick,postaction={decorate}] (-0.25,2.35) .. controls (-3.00,3.5) and (-3.00,0.5) .. (-0.25,1.65);

\draw[thick] (0.35,2) arc (0:360:3.5mm);
\draw[thick,->] (0.10,0.5) -- (0.10,1.5);
\end{scope}

\begin{scope}[shift={(7,0)}]
%\draw[thin,yellow] (0,0) grid (4,4);
\node at (-1,2) {$A(+)$};
\node at (1.5,3.10) {$+$};
\draw[thick,dashed] (0.5,2.75) -- (2.5,2.75);
\draw[thick,->] (1.5,1.25) -- (1.5,2.75);
\draw[thick,fill] (1.65,2) arc (0:360:1.5mm);
\node at (2.1,2) {$\omega$};
\node at (1.5,0.5) {$\omega \in \Sigma^*$};
\end{scope}

\begin{scope}[shift={(15,0)}]
%\draw[thin,yellow] (0,0) grid (4,4);
\node at (-1,2) {$A(-)$};
\node at (1.5,3.10) {$-$};
\draw[thick,dashed] (0.5,2.75) -- (2.5,2.75);
\draw[thick,<-] (1.5,1.25) -- (1.5,2.75);
\draw[thick,fill] (1.65,2) arc (0:360:1.5mm);
\node at (2.15,2) {$\omega'$};
\node at (1.5,0.5) {$\omega' \in \Sigma^{\h}$};
\node at (3,2) {$=$};
\node at (4.5,3.10) {$-$};
\draw[thick,dashed] (3.5,2.75) -- (5.5,2.75);
\draw[thick,<-] (4.5,1.25) -- (4.5,2.75);
\draw[thick,fill] (4.65,2.25) arc (0:360:1.5mm);
\draw[thick,fill] (4.65,1.75) arc (0:360:1.5mm);
\node at (5,1.75) {$\vdots$};
\end{scope}

\end{tikzpicture}
    \caption{Diagrams for  Example~\ref{ex:sofic-system-0000}. Left: an automaton for $W^{\perp}_{\h}$. Small arrow denotes the initial state. Notice that the states of $A(+)$ are represented by finite words (and their $\Bool$-linear combinations), while those of $A(-)$ are represented by right-infinite words in $\Sigma^{\h}$. In the diagram, letters are accumulating at the \emph{out-oriented} inner  boundary points but not at the \emph{in-oriented} inner points. }
    \label{sofic-0025}
\end{figure}
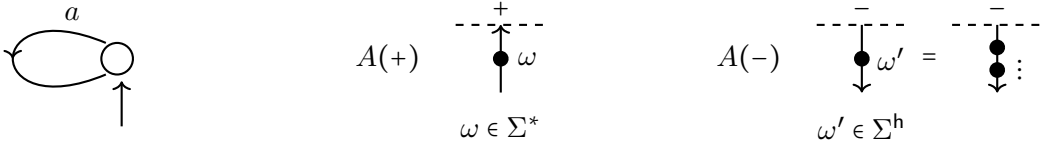

\end{example}

\begin{figure}
    \centering
\begin{tikzpicture}[scale=0.6]

\begin{scope}[shift={(0,0)}]
%\draw[thin,yellow] (0,0) grid (4,4);

\node at (-1.65,1.15) {$A(+)$};
\node at (0.5,2.35) {$+$};
\draw[thick,dashed] (0,2) -- (1,2);
\draw[thick,<-] (0.5,2) -- (0.5,0.5);
\draw[thick,fill] (0.65,1.25) arc (0:360:1.5mm);
\node at (1,1.25) {$a$};

\node at (2,2.35) {$+$};
\draw[thick,dashed] (1.5,2) -- (2.5,2);
\draw[thick,<-] (2,2) -- (2,0.5);
\draw[thick,fill] (2.15,1.25) arc (0:360:1.5mm);
\node at (2.5,1.25) {$b$};
\end{scope}

\begin{scope}[shift={(8,0)}]
%\draw[thin,yellow] (0,0) grid (4,4);
\node at (-1.65,1.15) {$A(-)$};
\node at (0.5,2.35) {$-$};
\draw[thick,dashed] (0,2) -- (1,2);
\draw[thick,->] (0.5,2) -- (0.5,0.5);
\draw[thick,fill] (0.65,1.25) arc (0:360:1.5mm);
\node at (1.25,1.25) {$b^{+\infty}$};

\node at (2.2,2.35) {$-$};
\draw[thick,dashed] (1.7,2) -- (2.7,2);
\draw[thick,->] (2.2,2) -- (2.2,0.5);

\draw[thick,fill] (2.35,1.5) arc (0:360:1.5mm);
\node at (2.65,1.5) {$a$};

\draw[thick,fill] (2.35,0.95) arc (0:360:1.5mm);
\node at (3,0.95) {$b^{+\infty}$};

\end{scope}

\end{tikzpicture}
    \caption{Example~\ref{ex:sofic-system-001}. Words for states in $A(+)$ are finite (elements of $\Sigma^*$). Words representing elements of $A(-)$ are right-infinite (elements of $\Sigma^{\h}$).
    }
    \label{sofic-0008}
\end{figure}
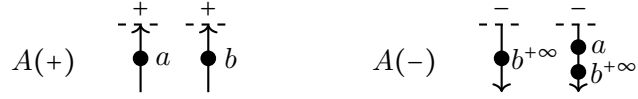

\begin{figure}
    \centering
\begin{tikzpicture}[scale=0.6]
\begin{scope}][shift={(0,0)}]
%\draw[thin,yellow] (0,0) grid (4,4);

\node at (3.8,4.7) {$A(-)$};

\begin{scope}[shift={(0.15,0)}]
\node at (0.4,2.5) {$1$};
\node at (0.4,1.5) {$1$};
\node at (0.4,0.5) {$1$};

\node at (0.4,4.7) {$y_0$};
\node at (0.4,4.2) {$-$};
\draw[thick,dashed] (0.1,4.0) -- (0.7,4.0);
\draw[thick,->] (0.4,4.0) -- (0.4,3.0); 
\draw[thick,fill] (0.52,3.5) arc (0:360:1.15mm);
\node at (1.08,3.55) {$b^{+\infty}$};
\end{scope}

\begin{scope}[shift={(0.45,0)}]
\node at (1.6,2.5) {$1$};
\node at (1.6,1.5) {$0$};
\node at (1.6,0.5) {$0$};

\node at (1.6,4.7) {$y_1$};
\node at (1.6,4.2) {$-$};
\draw[thick,dashed] (1.3,4.0) -- (1.9,4.0);
\draw[thick,->] (1.6,4.0) -- (1.6,3.0);

\draw[thick,fill] (1.72,3.7) arc (0:360:1.15mm);
\node at (2.0,3.7) {$a$};

\draw[thick,fill] (1.72,3.35) arc (0:360:1.15mm);
\node at (2.40,3.25) {$b^{+\infty}$};
\end{scope}

\begin{scope}[shift={(-0.25,0)}]
\node at (-2.85,1.5) {$A(+)$};

\begin{scope}[shift={(0,0.1)}]
\node at (-1.25,2.5) {$x_0$};
\node at (-0.5,2.95) {$+$};
\draw[thick,dashed] (-0.8,2.8) -- (-0.2,2.8);
\draw[thick,<-] (-0.5,2.8) -- (-0.5,2.15);     
\end{scope}

\node at (-1.25,1.5) {$x_1$};
\node at (-0.5,1.95) {$+$};
\draw[thick,dashed] (-0.8,1.8) -- (-0.2,1.8);
\draw[thick,<-] (-0.5,1.8) -- (-0.5,1.15);
\draw[thick,fill] (-0.38,1.42) arc (0:360:1.15mm);
\node at (-0.13,1.42) {$a$};

\begin{scope}[shift={(0,-0.1)}]
\node at (-1.25,0.5) {$x_2$};
\node at (-0.5,0.95) {$+$};
\draw[thick,dashed] (-0.8,0.8) -- (-0.2,0.8);
\draw[thick,<-] (-0.5,0.8) -- (-0.5,0.15); 
\draw[thick,fill] (-0.38,0.42) arc (0:360:1.15mm);
\node at (-0.13,0.42) {$b$};
\end{scope}
 
\end{scope}

\begin{scope}[shift={(-0.1,0)}]
\draw[thick] (0.1,0.1) .. controls (0,0.5) and (0,2.5) .. (0.1,2.9);
\end{scope}

\begin{scope}[shift={(-0.3,0)}]
\draw[thick] (2.9,0.1) .. controls (3,0.5) and (3,2.5) .. (2.9,2.9);
\end{scope}

\begin{scope}[shift={(2,0)}]
\node at (4.0,2.00) {$x_0+x_1=x_0$};
\node at (4.65,1.25) {$x_2=x_1$};

%\node at (6,1.6) {$\Rightarrow$};
%\node at (7.5,1.6) {$x_1=x_2$};

\node at (7.25,1.6) {$\Rightarrow$};
\node at (12.65,1.6) {$A(+) = \langle x_0,x_1: x_0+x_1=x_0\rangle$};

\node at (4.05,0.5) {$y_0+y_1=y_0$};

\node at (7.25,0.5) {$\Rightarrow$};
\node at (12.5,0.5) {$A(-) = \langle y_0,y_1: y_0+y_1=y_0\rangle$};
\end{scope}

\end{scope}

\end{tikzpicture}
    \caption{Generators of $A(+),A(-)$ and the pairing matrix for Example~\ref{ex:sofic-system-001}.}
    \label{sofic-0009}
\end{figure}
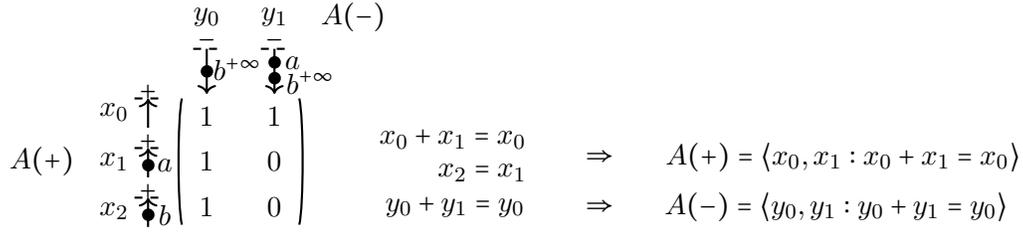

%%%%%%%%%%%%%%%%
% Example 1; state space is projective but 
%  not free B-modules 
%%%%%%%%%%%%%%%%

\begin{example}
\label{ex:sofic-system-001} 
Let $L= \{ab^{+\infty}, b^{+\infty} \}\subset \Sigma^{\h}$. Here $L=W^{\perp}_{\h}$ for $W=\{ba,aa\}$. 
It is clear that $A(+)$ is spanned by finite words $\emptyset, a,b$ and $A(-)$ by infinite words $b^{+\infty},ab^{+\infty}$, see Figure~\ref{sofic-0008} and Figure~\ref{sofic-0009}. The latter figure shows  the bilinear pairing matrix for these spanning sets. Consequently, $A(+)$ is a distributive but not a free $\Bool$-module with 3 elements $0,x_0,x_1$. Likewise, $A(-)$ consists of 3 elements $0,y_0,y_1$. 
The identity decomposition is  
\[ \id_+ \ = \ \langle a | \otimes  | b^{+\infty}\rangle   +  \langle \emptyset | \otimes | a b^{+\infty}\rangle , 
\]
shown in Figure~\ref{sofic-0010}. 
Language $L$ and the corresponding pairing can be upgraded to a TQFT $\mcF$ with $\mcF(+)=A(+)$ and $\mcF(-)=A(-)$, the corresponding action of $a,b$ on state spaces and the identity decomposition as above. 

\begin{figure}
    \centering
\begin{tikzpicture}[scale=0.6]
\begin{scope}[shift={(0.0,0)}, decoration={markings,mark=at position 0.5 with {\arrow{>}}}]
%\draw[thin,yellow] (0,0) grid (4,4);
\draw[thick,dashed] (0,3) -- (3,3);
\node at (0.5,3.35) {$+$};
\node at (2.5,3.35) {$-$};
\draw[thick,postaction={decorate}] (2.5,3) .. controls (2.4,0) and (0.6,0) .. (0.5,3);

\node at (4.5,1.75) {$=$};
\end{scope}

\begin{scope}[shift={(6,0)}, decoration={markings,mark=at position 0.5 with {\arrow{>}}}]
%\draw[thin,yellow] (0,0) grid (4,4);
\draw[thick,dashed] (0,3) -- (3,3);
\node at (0.5,3.35) {$+$};
\node at (2.5,3.35) {$-$};

\draw[thick,->] (0.5,0.5) -- (0.5,3);
\draw[thick,fill] (0.65,1.75) arc (0:360:1.5mm);
\node at (1.05,1.75) {$a$};

\draw[thick,->] (2.5,3) -- (2.5,0.5);
\draw[thick,fill] (2.65,1.75) arc (0:360:1.5mm);
\node at (3.45,1.75) {$b^{+\infty}$};

\node at (5.00,1.75) {$+$};
\end{scope}

\begin{scope}[shift={(12.25,0)}, decoration={markings,mark=at position 0.5 with {\arrow{>}}}]
%\draw[thin,yellow] (0,0) grid (4,4);
\draw[thick,dashed] (0,3) -- (3,3);
\node at (0.5,3.35) {$+$};
\node at (2.5,3.35) {$-$};

\draw[thick,->] (0.5,0.5) -- (0.5,3);
\draw[thick,fill] (2.65,1.33) arc (0:360:1.5mm);
\node at (3.00,2.16) {$a$};
\draw[thick,fill] (2.65,2.16) arc (0:360:1.5mm);
\node at (3.40,1.33) {$b^{+\infty}$};

\draw[thick,->] (2.5,3) -- (2.5,0.5);

\end{scope}

\end{tikzpicture}
    \caption{Identity decomposition for the language in Example~\ref{ex:sofic-system-001}.}
    \label{sofic-0010}
\end{figure}
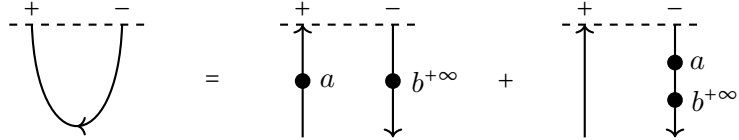

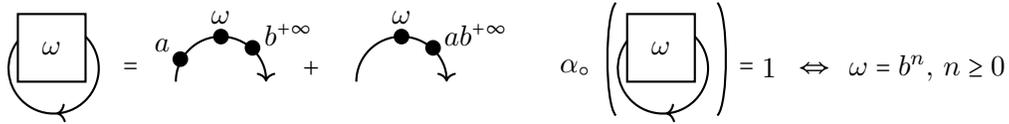
\begin{figure}
    \centering
\begin{tikzpicture}[scale=0.6]
\begin{scope}[shift={(0,0)}, decoration={markings,mark=at position 0.5 with {\arrow{>}}}]
%\draw[thin,yellow] (5,0) grid (10,4);

\draw[thick] (1.5,0) rectangle (3,1.5);
\node at (2.25,0.75) {$\omega$};
\draw[thick,postaction={decorate}] (3,1) arc (45:-218:1);

\node at (4,0.30) {$=$};

\draw[thick,<-] (7,0) arc (0:180:1); 
\draw[thick,fill] (5.25,0.50) arc (0:360:1.5mm);
\node at (4.7,0.80) {$a$};
\draw[thick,fill] (6.15,1.0) arc (0:360:1.5mm);
\node at (6,1.43) {$\omega$};
\draw[thick,fill] (6.85,0.75) arc (0:360:1.5mm);
\node at (7.5,1.05) {$b^{+\infty}$};

\node at (8,0.30) {$+$};

\draw[thick,<-] (11,0) arc (0:180:1); 
%\draw[thick,fill] (9.25,0.50) arc (0:360:1.5mm);
%\node at (8.8,0.80) {$a$};
\draw[thick,fill] (10.15,1.0) arc (0:360:1.5mm);
\node at (10,1.43) {$\omega$};
\draw[thick,fill] (10.85,0.75) arc (0:360:1.5mm);
\node at (11.65,1.05) {$a b^{+\infty}$};

\end{scope}

\begin{scope}[shift={(15,0)}, decoration={markings,mark=at position 0.5 with {\arrow{>}}}]
%\draw[thin,yellow] (0,0) grid (4,4);
\node at (-1.15,0.30) {$\alpha_{\circ}$};
\draw[thick] (-0.25,-0.75) .. controls (-0.50,-0.50) and (-0.50,1.50) .. (-0.25,1.75);

\draw[thick] (2,-0.75) .. controls (2.25,-0.50) and (2.25,1.50) .. (2,1.75);

\draw[thick] (0,0) rectangle (1.5,1.5);
\node at (0.75,0.75) {$\omega$};
\draw[thick,postaction={decorate}] (1.5,1) arc (45:-218:1);

\node at (2.65,0.30) {$=$};
\node at (3.15,0.30) {$1$};
\node at (4.15,0.30) {$\Leftrightarrow$};

\node at (6.65,0.30) {$\omega=b^n$, $n\geq 0$};

\end{scope}

\end{tikzpicture}
    \caption{Computing the associated circular language $L_{\circ}$ for the $\omega$-language in Example~\ref{ex:sofic-system-001}. To evaluate a circle with a finite word $\omega$ apply decomposition of the identity anywhere along the circle. The circular language in this case is $L_{\circ}=b^{\ast}$.   }
    \label{sofic-0012}
\end{figure}

For this boolean TQFT the state spaces are projective but not free $\Bool$-modules. 
The circular language $L_{\circ}$ for this TQFT is $b^{\ast}=\{ b^n | n\ge 0\}$, computed in Figure~\ref{sofic-0012}.

The same $\omega$-evaluation can be described via a TQFT with the state space $A(+)=\Bool\langle v_1,v_2\rangle$ a free rank two $\Bool$-module, with the action of $a$ and $b$ shown in Figure~\ref{sofic-0007} and the initial state $q_{\init}=v_1+v_2$. The decomposition of the identity for this TQFT is $\id_+=v_1\otimes v_1^{\ast}+v_2\otimes v_2^{\ast}$ and the associated circular language is again $b^{\ast}$, which consists of all finite words $\omega$ so that $\tr_{A(+)}(\omega)=1$. 

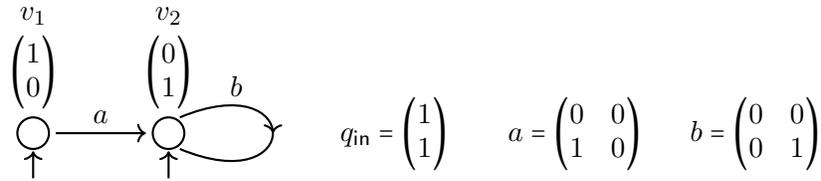
\begin{figure}
    \centering
\begin{tikzpicture}[scale=0.6]

%\begin{scope}[shift={(0.5,0)}, decoration={markings,mark=at position 0.5 with {\arrow{>}}}]
%\draw[thin,yellow] (0,0) grid (4,4);
%\node at (0,2) {$0$};
%\draw[thick] (0.35,2) arc (0:360:3.5mm);
%\draw[thick,->] (0.5,2) -- (2.5,2);
%\node at (3,2) {$1$};
%\draw[thick] (3.35,2.0) arc (0:360:3.5mm);

%\draw[thick,postaction={decorate}] (3.25,2.35) .. controls (6,3.5) and (6,0.5) .. (3.25,1.65);
%\end{scope}

\begin{scope}[shift={(8,0)}, decoration={markings,mark=at position 0.5 with {\arrow{>}}}]
%\draw[thin,yellow] (0,0) grid (4,4);
\node at (1.5,2.35) {$a$};
\draw[thick] (0.35,2) arc (0:360:3.5mm);
\draw[thick,<-] (0.0,1.6) -- (0.0,1);
\draw[thick,->] (0.5,2) -- (2.5,2);
\node at (4.5,3.05) {$b$};
\draw[thick] (3.35,2.0) arc (0:360:3.5mm);
\draw[thick,<-] (3.0,1.6) -- (3.0,1);

\node at (0.0,4.6) {$v_1$};
\node at (0.0,3.35) {$\begin{pmatrix} 1\\0 \end{pmatrix}$};

\node at (3,4.6) {$v_2$};
\node at (3,3.35) {$\begin{pmatrix} 0\\1 \end{pmatrix}$};

\draw[thick,postaction={decorate}] (3.25,2.35) .. controls (6,3.5) and (6,0.5) .. (3.25,1.65);

\node at (8,2) {$q_{\init}=\begin{pmatrix} 1 \\ 1 \\  \end{pmatrix}$};

\node at (12,2) {$a=\begin{pmatrix} 0 & 0\\ 1 & 0\\ \end{pmatrix}$};

\node at (16,2) {$b=\begin{pmatrix} 0 & 0\\ 0 & 1\\ \end{pmatrix}$};

\end{scope}

\end{tikzpicture}
    \caption{Automaton and action of $a$ and $b$ in  Example~\ref{ex:sofic-system-001}.}
    \label{sofic-0007}
\end{figure}

\end{example}

%%%%%%%%%%%%%%%%
% Example 2; state space is projective but 
%  not free B-modules 
%%%%%%%%%%%%%%%%

\begin{example} 
\label{ex:sofic-system-002} 
Let $W=\{ba\}$, and set  
 $L=W^{\perp}_{\h}=\{ a^k b^{+\infty}: k\geq 0\}\sqcup \{ a^{+\infty}\}$. An automaton for this $\omega$-evaluation is shown in  Figure~\ref{sofic-0013}. 
Generating vectors for the state spaces $A(-)$ and $A(+)$ are given in Figure~\ref{sofic-0014}. Note that $|a b^{+\infty}\rangle = |a^{+\infty}\rangle$, hence the vector $|a^{+\infty}\rangle$ is not shown. 

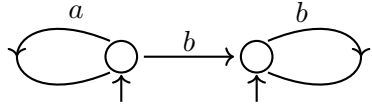
\begin{figure}
    \centering
\begin{tikzpicture}[scale=0.6]

%\begin{scope}[shift={(0,0)}, decoration={markings,mark=at position 0.5 with {\arrow{>}}}]
%\draw[thin,yellow] (0,0) grid (4,4);
%\node at (0,2) {$0$};
%\draw[thick,postaction={decorate}] (-0.25,2.35) .. controls (-3.00,3.5) and (-3.00,0.5) .. (-0.25,1.65);
%\draw[thick] (0.35,2) arc (0:360:3.5mm);
%\draw[thick,->] (0.5,2) -- (2.5,2);
%\node at (3,2) {$1$};
%\draw[thick] (3.35,2.0) arc (0:360:3.5mm);
%\draw[thick,postaction={decorate}] (3.25,2.35) .. controls (6,3.5) and (6,0.5) .. (3.25,1.65);
%\end{scope}

\begin{scope}[shift={(10,0)}, decoration={markings,mark=at position 0.5 with {\arrow{>}}}]
%\draw[thin,yellow] (0,0) grid (4,4);
\node at (-1,3.0) {$a$}; 
\draw[thick,postaction={decorate}] (-0.25,2.35) .. controls (-3.00,3.5) and (-3.00,0.5) .. (-0.25,1.65);

\node at (1.5,2.35) {$b$};
\draw[thick] (0.35,2) arc (0:360:3.5mm);
\draw[thick,<-] (0.0,1.6) -- (0.0,1);
\draw[thick,->] (0.5,2) -- (2.5,2);
\draw[thick] (3.35,2.0) arc (0:360:3.5mm);
\draw[thick,<-] (3.0,1.6) -- (3.0,1);

\node at (4,3) {$b$};
\draw[thick,postaction={decorate}] (3.25,2.35) .. controls (6,3.5) and (6,0.5) .. (3.25,1.65);
\end{scope}

\end{tikzpicture}
    \caption{The automaton for the $\omega$-language in Example~\ref{ex:sofic-system-002}. }
    \label{sofic-0013}
\end{figure}

\begin{figure}
    \centering
\begin{tikzpicture}[scale=0.6]
\begin{scope}[shift={(0,0)}]
%\draw[thin,yellow] (0,0) grid (4,4);

\node at (0.5,2) {$A(+)$}; 

\node at (3.25,3.35) {$+$};
\draw[thick,dashed] (2.5,3) -- (4.0,3);
\draw[thick,<-] (3.25,3) -- (3.25,1);

\node at (4.5,2) {$=$};

\node at (5.75,3.35) {$+$};
\draw[thick,dashed] (5.0,3) -- (6.5,3);
\draw[thick,<-] (5.75,3) -- (5.75,1);
\draw[thick,fill] (5.90,2) arc (0:360:1.5mm);
\node at (6.23,2) {$a$};

\node at (7.0,2) {$=$};

\node at (8.25,3.35) {$+$};
\draw[thick,dashed] (7.5,3) -- (9.0,3);
\draw[thick,<-] (8.25,3) -- (8.25,1);
\draw[thick,fill] (8.40,2) arc (0:360:1.5mm);
\node at (9.65,2.1) {$a^k, k\geq 1$};
\end{scope}

\begin{scope}[shift={(13,0)}]

\node at (2.75,3.35) {$+$};
\draw[thick,dashed] (2,3) -- (3.5,3);
\draw[thick,<-] (2.75,3) -- (2.75,1);
\draw[thick,fill] (2.90,2) arc (0:360:1.5mm);
\node at (3.20,2) {$b$};

\node at (4.0,2) {$=$};

\node at (5.25,3.35) {$+$};
\draw[thick,dashed] (4.5,3) -- (6.0,3);
\draw[thick,<-] (5.25,3) -- (5.25,1);
\draw[thick,fill] (5.40,2.35) arc (0:360:1.5mm);
\node at (5.95,2.35) {$b^m$};
\draw[thick,fill] (5.40,1.5) arc (0:360:1.5mm);
\node at (5.95,1.65) {$a^k$};
\node at (8.5,2.0) {$m\geq 1, k\geq 0$};
\end{scope}

\begin{scope}[shift={(0.5,-4.0)}]
%\draw[thin,yellow] (0,0) grid (4,4);

\node at (0,2) {$A(-)$}; 

\node at (2.75,3.35) {$-$};
\draw[thick,dashed] (2,3) -- (3.5,3);
\draw[thick,->] (2.75,3) -- (2.75,1);
\draw[thick,fill] (2.90,2) arc (0:360:1.5mm);
\node at (3.5,2) {$b^{+\infty}$};

\end{scope}

\begin{scope}[shift={(6.0,-4.0)}]
%\draw[thin,yellow] (0,0) grid (4,4);

\node at (2.75,3.35) {$-$};
\draw[thick,dashed] (2,3) -- (3.5,3);
\draw[thick,->] (2.75,3) -- (2.75,1);
\draw[thick,fill] (2.90,2.35) arc (0:360:1.5mm);
\node at (3.25,2.35) {$a$};

\draw[thick,fill] (2.90,1.65) arc (0:360:1.5mm);
\node at (3.5,1.65) {$b^{+\infty}$};

\node at (4.5,2) {$=$};

\begin{scope}[shift={(0.5,0)}]
\node at (5.0,3.35) {$-$};
\draw[thick,dashed] (4.5,3) -- (5.5,3);
\draw[thick,->] (5.0,3) -- (5.0,1);
\draw[thick,fill] (5.15,2.35) arc (0:360:1.5mm);
\node at (5.75,2.35) {$a^k$,};
\node at (7.45,2.35) {$k\geq 1$};

\draw[thick,fill] (5.15,1.65) arc (0:360:1.5mm);
\node at (5.9,1.65) {$b^{+\infty}$};
\end{scope}

\end{scope}

\end{tikzpicture}
    \caption{Generators of $A(+)$ and $A(-)$ in  Example~\ref{ex:sofic-system-002}. }
    \label{sofic-0014}
\end{figure}
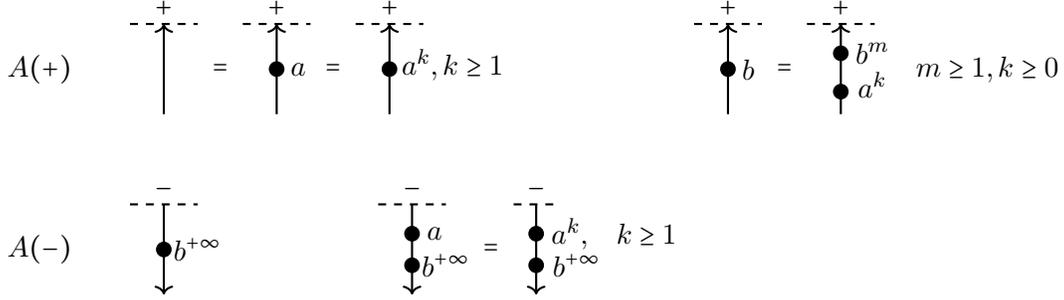

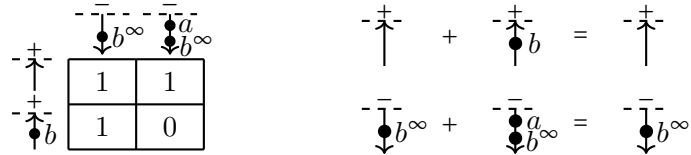
\begin{figure}
    \centering
\begin{tikzpicture}[scale=0.6]

\begin{scope}[shift={(0,0)}]
%\draw[thin,yellow] (0,0) grid (4,4);

\draw[thick] (0,2) -- (3,2);
\draw[thick] (0,1) -- (3,1);
\draw[thick] (0,0) -- (3,0);
\draw[thick] (0,0) -- (0,2);
\draw[thick] (1.5,0) -- (1.5,2);
\draw[thick] (3,0) -- (3,2);

\node at (0.75,1.5) {$1$};
\node at (2.25,1.5) {$1$};
\node at (0.75,0.5) {$1$};
\node at (2.25,0.5) {$0$};

\node at (0.75,3.2) {$-$};
\draw[thick,dashed] (0.15,3) -- (1.35,3);
\draw[thick,->] (0.75,3) -- (0.75,2.1);
\draw[thick,fill] (0.85,2.5) arc (0:360:1.00mm);
\node at (1.35,2.5) {$b^{\infty}$};

\node at (2.25,3.2) {$-$};
\draw[thick,dashed] (1.65,3) -- (2.85,3); 
\draw[thick,->] (2.25,3) -- (2.25,2.1);
\draw[thick,fill] (2.35,2.75) arc (0:360:1.00mm);
\node at (2.6,2.75) {$a$};
\draw[thick,fill] (2.35,2.40) arc (0:360:1.00mm);
\node at (2.85,2.30) {$b^{\infty}$};

\begin{scope}[shift={(-0.25,0)}]
\node at (-0.5,2.2) {$+$};
\draw[thick,dashed] (-1,2) -- (0,2);
\draw[thick,<-] (-0.5,2.0) -- (-0.5,1.3);

\node at (-0.5,1.05) {$+$};
\draw[thick,dashed] (-1,0.8) -- (0,0.8); 
\draw[thick,<-] (-0.5,0.8) -- (-0.5,0);
\draw[thick,fill] (-0.40,0.35) arc (0:360:1.0mm);
\node at (-0.15,0.35) {$b$};
\end{scope}

\end{scope}

\begin{scope}[shift={(6.5,0.85)}]
%\draw[thin,yellow] (0,0) grid (5,4);

\node at (0.5,2.2) {$+$};
\draw[thick,dashed] (0,2) -- (1,2);
\draw[thick,<-] (0.5,2) -- (0.5,1);

\node at (1.95,1.65) {$+$};

\begin{scope}[shift={(0.9,0)}]
\node at (2.5,2.2) {$+$};
\draw[thick,dashed] (2,2) -- (3,2);
\draw[thick,<-] (2.5,2) -- (2.5,1);
\draw[thick,fill] (2.63,1.5) arc (0:360:1.25mm);
\node at (2.92,1.52) {$b$};
\end{scope}

\node at (4.85,1.65) {$=$};

\begin{scope}[shift={(1.8,0)}]
\node at (4.5,2.2) {$+$};
\draw[thick,dashed] (4,2) -- (5,2);
\draw[thick,<-] (4.5,2) -- (4.5,1);
\end{scope}

\end{scope}

\begin{scope}[shift={(6.5,-1.1)}]
%\draw[thin,yellow] (0,0) grid (5,4);

\node at (0.5,2.2) {$-$};
\draw[thick,dashed] (0,2) -- (1,2);
\draw[thick,->] (0.5,2) -- (0.5,1);
\draw[thick,fill] (0.63,1.5) arc (0:360:1.25mm);
\node at (1.12,1.5) {$b^{\infty}$};

\node at (1.95,1.65) {$+$};

\begin{scope}[shift={(0.9,0)}]
\node at (2.5,2.2) {$-$};
\draw[thick,dashed] (2,2) -- (3,2);
\draw[thick,->] (2.5,2) -- (2.5,1.0);
\draw[thick,fill] (2.63,1.73) arc (0:360:1.25mm);
\node at (2.92,1.70) {$a$};
\draw[thick,fill] (2.63,1.35) arc (0:360:1.25mm);
\node at (3.12,1.30) {$b^{\infty}$};
\end{scope}

\node at (4.85,1.65) {$=$};

\begin{scope}[shift={(1.8,0)}]
\node at (4.5,2.2) {$-$};
\draw[thick,dashed] (4,2) -- (5,2);
\draw[thick,->] (4.5,2) -- (4.5,1);
\draw[thick,fill] (4.63,1.5) arc (0:360:1.25mm);
\node at (5.12,1.5) {$b^{\infty}$};
\end{scope}

\end{scope}

\end{tikzpicture}
    \caption{Pairing matrix and skein relations for Example~\ref{ex:sofic-system-002}. }
    \label{sofic-0015}
\end{figure}

Figure~\ref{sofic-0015} shows the inner products of these generators and skein relations. The state spaces, as $\Bool$-modules, are isomorphic to those in the previous example, and they are projective modules. The identity decomposition, shown in Figure~\ref{sofic-0016}, is 
\[ \id_+ \ = \  \langle \emptyset | \otimes | a b^{+\infty}\rangle   + \langle b | \otimes | a b^{+\infty}\rangle . 
\]
The circular language is  $L_{\circ}=a^{\ast}+b^{\ast}$. 

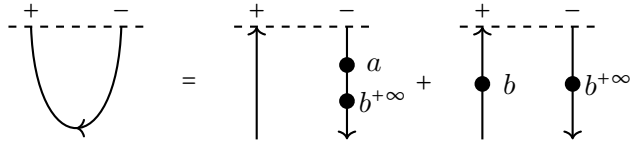
\begin{figure}
    \centering
\begin{tikzpicture}[scale=0.6]
%\draw[thin,yellow] (0,0) grid (4,4);
\begin{scope}[shift={(0.0,0)}, decoration={markings,mark=at position 0.5 with {\arrow{>}}}]
%\draw[thin,yellow] (0,0) grid (4,4);
\draw[thick,dashed] (0,3) -- (3,3);
\node at (0.5,3.35) {$+$};
\node at (2.5,3.35) {$-$};
\draw[thick,postaction={decorate}] (2.5,3) .. controls (2.4,0) and (0.6,0) .. (0.5,3);

\node at (4,1.75) {$=$};
\end{scope}

\begin{scope}[shift={(5,0)}, decoration={markings,mark=at position 0.5 with {\arrow{>}}}]
%\draw[thin,yellow] (0,0) grid (4,4);
\draw[thick,dashed] (0,3) -- (3,3);
\node at (0.5,3.35) {$+$};
\node at (2.5,3.35) {$-$};

\draw[thick,->] (0.5,0.5) -- (0.5,3);
%\draw[thick,fill] (0.65,1.75) arc (0:360:1.5mm);
%\node at (1.10,1.75) {$b^{\infty}$};

\draw[thick,->] (2.5,3) -- (2.5,0.5);
\draw[thick,fill] (2.65,2.15) arc (0:360:1.5mm);
\node at (3.1,2.15) {$a$};

\draw[thick,fill] (2.65,1.35) arc (0:360:1.5mm);
\node at (3.3,1.35) {$b^{+\infty}$};

\node at (4.2,1.75) {$+$};
\end{scope}

\begin{scope}[shift={(10,0)}, decoration={markings,mark=at position 0.5 with {\arrow{>}}}]
\draw[thick,dashed] (0,3) -- (3,3);
\node at (0.5,3.35) {$+$};
\node at (2.5,3.35) {$-$};

\draw[thick,->] (0.5,0.5) -- (0.5,3);
\draw[thick,fill] (0.65,1.73) arc (0:360:1.5mm);
\node at (1.10,1.73) {$b$};
%\draw[thick,fill] (0.65,2.16) arc (0:360:1.5mm);
%\node at (0.9,2.16) {$a$};

\draw[thick,->] (2.5,3) -- (2.5,0.5);
\draw[thick,fill] (2.65,1.73) arc (0:360:1.5mm);
\node at (3.30,1.73) {$b^{+\infty}$};

\end{scope}

\end{tikzpicture}
    \caption{Identity decomposition in  Example~\ref{ex:sofic-system-002}.}
    \label{sofic-0016}
\end{figure}

\end{example}

%%%%%%%%%%%%%%%%%%%%%%%%
%  Example 3: A(-) is free rank 2 B-module,
% golden mean subshift 
%%%%%%%%%%%%%%%%%%%%%%%%

\begin{example}
\label{ex:sofic-system-003}
Consider the one-sided shift, with $\Sigma=\{a,b\}$ and $W=\{bb\}$.  This is the golden mean subshift as in Example~\ref{ex_golden}, but in the one-sided case. 
 A word $\omega$ is in the language $L=W^{\perp}_{\h}$ if it does not contain $bb$ as a subword. Figure~\ref{sofic-0018} shows an automaton for that language, with $Q_{\ac}=Q$. 

\begin{figure}
    \centering
\begin{tikzpicture}[scale=0.6]
%\begin{scope}[shift={(0,0)}, decoration={markings,mark=at position 0.5 with {\arrow{>}}}]
%\draw[thin,yellow] (0,0) grid (4,4);
%\node at (-1,2) {$0$};
%\draw[thick,postaction={decorate}] (-0.25,2.35) .. controls (-3.00,3.5) and (-3.00,0.5) .. (-0.25,1.65);
%\draw[thick] (0.35,2) arc (0:360:3.5mm);
%\draw[thick,postaction={decorate}] (0.4,2.2) .. controls (0.5,2.75) and (2.5,2.75) .. (2.6,2.2);
%\draw[thick,postaction={decorate}] (2.6,1.8) .. controls (2.5,1.25) and (0.5,1.25) .. (0.4,1.8);
%\node at (3,2) {$1$};
%\draw[thick] (3.35,2.0) arc (0:360:3.5mm);
%\end{scope}

%\begin{scope}[shift={(8,0)}, decoration={markings,mark=at position 0.5 with {\arrow{>}}}]
%\draw[thin,yellow] (0,0) grid (4,4);
%\node at (-2.40,2.65) {$0$};
%\draw[thick,postaction={decorate}] (-0.25,2.35) .. controls (-3.00,3.5) and (-3.00,0.5) .. (-0.25,1.65);
%\draw[thick] (0.35,2) arc (0:360:3.5mm);
%\draw[thick,postaction={decorate}] (0.4,2.2) .. controls (0.5,2.75) and (2.5,2.75) .. (2.6,2.2);
%\draw[thick,postaction={decorate}] (2.6,1.8) .. controls (2.5,1.5) and (0.5,1.5) .. (0.4,1.8);
%\node at (1.80,3.00) {$1$};
%\node at (1.20,1.15) {$0$};
%\draw[thick] (3.35,2.0) arc (0:360:3.5mm);
%\draw[thick,->] (-0.5,1) -- (0,1.5);

%\draw[thick,->] (-2.10,2.8) -- (-1.50,3.15);
%\node at (-1.25,3.15) {$a$};

%\draw[thick,->] (2.00,3.10) -- (2.75,3.45);
%\node at (2.95,3.45) {$b$};
%\end{scope}

\begin{scope}[shift={(16,0)}, decoration={markings,mark=at position 0.5 with {\arrow{>}}}]
%\draw[thin,yellow] (0,0) grid (4,4);
    
\node at (-2.40,2.65) {$a$};
\draw[thick,postaction={decorate}] (-0.25,2.35) .. controls (-3.00,3.5) and (-3.00,0.5) .. (-0.25,1.65);
\draw[thick] (0.35,2) arc (0:360:3.5mm);
\node at (0,2) {$0$}; 
\draw[thick,postaction={decorate}] (0.4,2.2) .. controls (0.5,2.75) and (2.5,2.75) .. (2.6,2.2);
\draw[thick,postaction={decorate}] (2.6,1.8) .. controls (2.5,1.5) and (0.5,1.5) .. (0.4,1.8);
\node at (1.80,3.00) {$b$};
\node at (1.20,1.15) {$a$};
\node at (3,2) {$1$};
\draw[thick] (3.35,2.0) arc (0:360:3.5mm);
\draw[thick,->] (0,0.7) -- (0,1.5);

\end{scope}

\end{tikzpicture}
    \caption{An automaton for  Example~\ref{ex:sofic-system-003}. An example of a word in this language is $abaaabaabaa\cdots$. }
    \label{sofic-0018}
\end{figure}
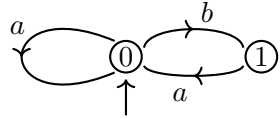
  
The generating sets for $A(-)$ and $A(+)$ are given in Figure~\ref{sofic-0019}, with the matrix of the inner product shown in Figure~\ref{sofic-0020}. 
Identity decomposition in  Figure~\ref{sofic-0021} implies that the circular language $L_{\circ}$ for this boolean TQFT consists of circular words without a consecutive pair $bb$. This language  can be described by a regular expression in Figure~\ref{sofic-0022}. 

\begin{figure}
    \centering
\begin{tikzpicture}[scale=0.6]
\begin{scope}[shift={(0,0)}]
%\draw[thin,yellow] (0,0) grid (4,4);

\node at (0.5,2) {$A(+)$:}; 

\node at (3.25,3.35) {$+$};
\draw[thick,dashed] (2.5,3) -- (4.0,3);
\draw[thick,<-] (3.25,3) -- (3.25,1);

\node at (4.5,2) {$=$};

\node at (5.75,3.35) {$+$};
\draw[thick,dashed] (5.0,3) -- (6.5,3);
\draw[thick,<-] (5.75,3) -- (5.75,1);
\draw[thick,fill] (5.90,2) arc (0:360:1.5mm);
\node at (6.23,2) {$a$};

\node at (7.0,2) {$=$};

\node at (8.25,3.35) {$+$};
\draw[thick,dashed] (7.5,3) -- (9.0,3);
\draw[thick,<-] (8.25,3) -- (8.25,1);
\draw[thick,fill] (8.40,2) arc (0:360:1.5mm);
\node at (9.00,2.1) {$a^k$,};
\node at (10.75,2.1) {$k\geq 1$};
\end{scope}

\begin{scope}[shift={(13,0)}]

\node at (2.75,3.35) {$+$};
\draw[thick,dashed] (2,3) -- (3.5,3);
\draw[thick,<-] (2.75,3) -- (2.75,1);
\draw[thick,fill] (2.90,2) arc (0:360:1.5mm);
\node at (3.20,2) {$b$};

\node at (4.0,2) {$=$};

\node at (5.25,3.35) {$+$};
\draw[thick,dashed] (4.5,3) -- (6.0,3);
\draw[thick,<-] (5.25,3) -- (5.25,1);
\draw[thick,fill] (5.40,2.35) arc (0:360:1.5mm);
\node at (5.75,2.35) {$b$};
\draw[thick,fill] (5.40,1.5) arc (0:360:1.5mm);
\node at (5.95,1.65) {$a^k$};
\node at (8.00,2.0) {$k\geq 0$};
\end{scope}

\begin{scope}[shift={(0.5,-4.0)}]
%\draw[thin,yellow] (0,0) grid (4,4);

\node at (0,2) {$A(-)$:}; 

\node at (2.75,3.35) {$-$};
\draw[thick,dashed] (2,3) -- (3.5,3);
\draw[thick,->] (2.75,3) -- (2.75,1);
\draw[thick,fill] (2.90,2) arc (0:360:1.5mm);
\node at (3.65,2) {$a^{+\infty}$};

\end{scope}

\begin{scope}[shift={(6.0,-4.0)}]
%\draw[thin,yellow] (0,0) grid (4,4);

\node at (2.75,3.35) {$-$};
\draw[thick,dashed] (2,3) -- (3.5,3);
\draw[thick,->] (2.75,3) -- (2.75,1);
\draw[thick,fill] (2.90,2.35) arc (0:360:1.5mm);
\node at (3.25,2.35) {$b$};

\draw[thick,fill] (2.90,1.65) arc (0:360:1.5mm);
\node at (3.65,1.65) {$a^{+\infty}$};

\end{scope}

\end{tikzpicture}
    \caption{Spanning sets for $A(-)$ and $A(+)$ in  Example~\ref{ex:sofic-system-003}.}
    \label{sofic-0019}
\end{figure}
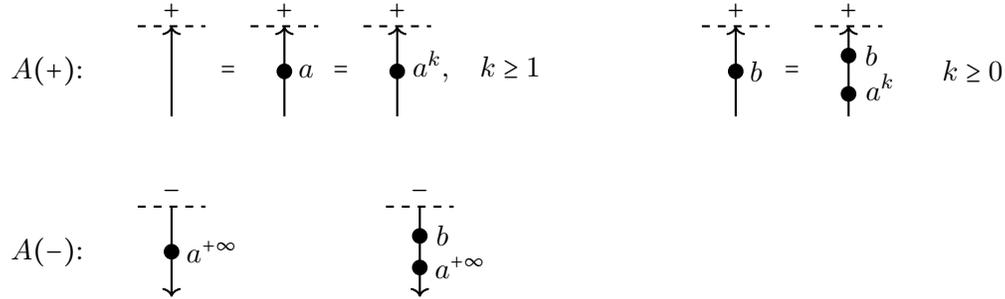

\begin{figure}
    \centering
\begin{tikzpicture}[scale=0.6]
    
\begin{scope}[shift={(0,0)}]
%\draw[thin,yellow] (0,0) grid (4,4);

\draw[thick] (0,2) -- (3.5,2);
\draw[thick] (0,1) -- (3.5,1);
\draw[thick] (0,0) -- (3.5,0);

\draw[thick] (0,0) -- (0,2);
\draw[thick] (1.75,0) -- (1.75,2);
\draw[thick] (3.5,0) -- (3.5,2);

\node at (0.875,1.5) {$1$};
\node at (2.625,1.5) {$1$};
\node at (0.875,0.5) {$1$};
\node at (2.625,0.5) {$0$};

\node at (0.875,3.2) {$-$};
\draw[thick,dashed] (0.2,3.0) -- (1.55,3.0);
\draw[thick,->] (0.875,3.0) -- (0.875,2.1);
\draw[thick,fill] (0.975,2.6) arc (0:360:1.00mm);
\node at (1.60,2.6) {$a^{+\infty}$};

\node at (2.625,3.2) {$-$};
\draw[thick,dashed] (1.95,3.0) -- (3.3,3.0); 
\draw[thick,->] (2.625,3.0) -- (2.625,2.1);
\draw[thick,fill] (2.725,2.75) arc (0:360:1.00mm);
\node at (2.925,2.70) {$b$};
\draw[thick,fill] (2.725,2.40) arc (0:360:1.00mm);
\node at (3.45,2.40) {$a^{+\infty}$};

\node at (-0.5,2.2) {$+$};
\draw[thick,dashed] (-0.8,2.0) -- (-0.2,2.0);
\draw[thick,<-] (-0.5,2.0) -- (-0.5,1.2);

\node at (-0.5,0.97) {$+$};
\draw[thick,dashed] (-0.8,0.8) -- (-0.2,0.8); 
\draw[thick,<-] (-0.5,0.8) -- (-0.5,0.0);
\draw[thick,fill] (-0.40,0.35) arc (0:360:1.0mm);
\node at (-0.23,0.30) {$b$};
\end{scope}

\begin{scope}[shift={(8,0.75)}]
%\draw[thin,yellow] (0,0) grid (4,4);
\node at (0.375,2.2) {$+$};
\draw[thick,dashed] (0,2) -- (0.75,2);
\draw[thick,<-] (0.375,2) -- (0.375,1.2);

\node at (1.85,1.65) {$+$};

\begin{scope}[shift={(1.25,0)}]
\node at (2.125,2.2) {$+$};
\draw[thick,dashed] (1.75,2) -- (2.50,2);
\draw[thick,<-] (2.125,2) -- (2.125,1.2);
\draw[thick,fill] (2.26,1.50) arc (0:360:1.25mm);
\node at (2.5,1.50) {$b$};
\end{scope}

\node at (4.85,1.65) {$=$};

\begin{scope}[shift={(2.25,0)}]
\node at (3.875,2.2) {$+$};
\draw[thick,dashed] (3.5,2) -- (4.25,2);
\draw[thick,<-] (3.875,2) -- (3.875,1.2);
\end{scope}

\end{scope}

\begin{scope}[shift={(8,-1)}]
%\draw[thin,yellow] (0,0) grid (4,4);
\node at (0.375,2.2) {$-$};
\draw[thick,dashed] (0,2) -- (0.75,2);
\draw[thick,->] (0.375,2) -- (0.375,1.2);
\draw[thick,fill] (0.49,1.65) arc (0:360:1.25mm);
\node at (1.15,1.75) {$a^{+\infty}$};

\node at (1.85,1.65) {$+$};

\begin{scope}[shift={(1.25,0)}]
\node at (2.125,2.2) {$-$};
\draw[thick,dashed] (1.75,2) -- (2.50,2);
\draw[thick,->] (2.125,2) -- (2.125,1);
\draw[thick,fill] (2.265,1.70) arc (0:360:1.25mm);
\node at (2.55,1.70) {$b$};
\draw[thick,fill] (2.265,1.35) arc (0:360:1.25mm);
\node at (2.95,1.35) {$a^{+\infty}$};
\end{scope}

\node at (4.85,1.65) {$=$};

\begin{scope}[shift={(2.25,0)}]
\node at (3.875,2.2) {$-$};
\draw[thick,dashed] (3.5,2) -- (4.25,2);
\draw[thick,->] (3.875,2) -- (3.875,1.2);
\draw[thick,fill] (3.975,1.65) arc (0:360:1.25mm);
\node at (4.65,1.65) {$a^{+\infty}$};
\end{scope}
\end{scope}

\end{tikzpicture}
    \caption{For Example~\ref{ex:sofic-system-003}.  }
    \label{sofic-0020}
\end{figure}

\begin{figure}
    \centering
\begin{tikzpicture}[scale=0.6]
%\draw[thin,yellow] (0,0) grid (4,4);
\begin{scope}[shift={(0.0,0)}, decoration={markings,mark=at position 0.5 with {\arrow{>}}}]
%\draw[thin,yellow] (0,0) grid (4,4);
\draw[thick,dashed] (0,3) -- (3,3);
\node at (0.5,3.35) {$+$};
\node at (2.5,3.35) {$-$};
\draw[thick,postaction={decorate}] (2.5,3) .. controls (2.4,0) and (0.6,0) .. (0.5,3);

\node at (4,1.75) {$=$};
\end{scope}

\begin{scope}[shift={(5,0)}, decoration={markings,mark=at position 0.5 with {\arrow{>}}}]
%\draw[thin,yellow] (0,0) grid (4,4);
\draw[thick,dashed] (0,3) -- (3,3);
\node at (0.5,3.35) {$+$};
\node at (2.5,3.35) {$-$};

\draw[thick,->] (0.5,0.5) -- (0.5,3);
\draw[thick,fill] (0.65,1.75) arc (0:360:1.5mm);
\node at (1.05,1.75) {$b$};

\draw[thick,->] (2.5,3) -- (2.5,0.5);
\draw[thick,fill] (2.65,1.75) arc (0:360:1.5mm);
\node at (3.2,1.75) {$a^{\infty}$};

\node at (4,1.75) {$+$};
\end{scope}

\begin{scope}[shift={(10,0)}, decoration={markings,mark=at position 0.5 with {\arrow{>}}}]
\draw[thick,dashed] (0,3) -- (3,3);
\node at (0.5,3.35) {$+$};
\node at (2.5,3.35) {$-$};

\draw[thick,->] (0.5,0.5) -- (0.5,3);

\draw[thick,->] (2.5,3) -- (2.5,0.5);
\draw[thick,fill] (2.7,1.33) arc (0:360:1.5mm);
\node at (3.3,1.33) {$a^{\infty}$};
\draw[thick,fill] (2.7,2.16) arc (0:360:1.5mm);
\node at (3.05,2.16) {$b$};

\end{scope}

\end{tikzpicture}
    \caption{Identity decomposition for  Example~\ref{ex:sofic-system-003}. }
    \label{sofic-0021}
\end{figure}
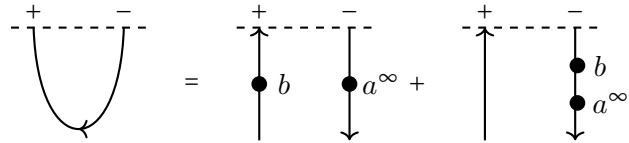 

\begin{figure}
    \centering
\begin{tikzpicture}[scale=0.6]
   \begin{scope}[shift={(0,0)}, decoration={markings,mark=at position 0.5 with {\arrow{>}}}]
%\draw[thin,yellow] (5,0) grid (10,4);

\draw[thick] (1.5,0) rectangle (3,1.5);
\node at (2.25,0.75) {$\omega$};
\draw[thick,postaction={decorate}] (3,1) arc (45:-218:1);

\node at (4,0.30) {$=$};

\begin{scope}[shift={(0.5,0)}]

\draw[thick,<-] (7,0) arc (0:180:1); 
\draw[thick,fill] (5.25,0.50) arc (0:360:1.5mm);
\node at (4.7,0.80) {$b$};
\draw[thick,fill] (6.15,1.0) arc (0:360:1.5mm);
\node at (6,1.43) {$\omega$};
\draw[thick,fill] (6.85,0.75) arc (0:360:1.5mm);
\node at (7.5,1.07) {$a^{+\infty}$};

\node at (8.5,0.30) {$+$};
\end{scope}

\begin{scope}[shift={(1.5,0)}]
\draw[thick,<-] (11,0) arc (0:180:1); 
\draw[thick,fill] (9.25,0.50) arc (0:360:1.5mm);
\node at (8.7,0.80) {$\omega$};
\draw[thick,fill] (10.15,1.0) arc (0:360:1.5mm);
\node at (10,1.6) {$b$};
\draw[thick,fill] (10.85,0.75) arc (0:360:1.5mm);
\node at (11.57,1.07) {$a^{+\infty}$};
\end{scope}

\end{scope}

\begin{scope}[shift={(1.5,-3.5)}, decoration={markings,mark=at position 0.5 with {\arrow{>}}}]
%\draw[thin,yellow] (0,0) grid (4,4);
\node at (-1.0,0.30) {$\alpha$};
\draw[thick] (-0.25,-0.75) .. controls (-0.50,-0.50) and (-0.50,1.50) .. (-0.25,1.75);

\draw[thick] (2,-0.75) .. controls (2.25,-0.50) and (2.25,1.50) .. (2,1.75);

\draw[thick] (0,0) rectangle (1.5,1.5);
\node at (0.75,0.75) {$\omega$};
\draw[thick,postaction={decorate}] (1.5,1) arc (45:-218:1);

\node at (2.65,0.30) {$=$};
\node at (3.20,0.32) {$1$};
\node at (4.15,0.30) {$\Leftrightarrow$};

%\node at (5.7,0.30) {$\omega \in L_{\circ}$};

%\node at (4.15,0.30) {$\Leftrightarrow$};
%\node at (8.8,0.30) {$\omega\in (a^*ab)^*a^*$ or $\omega\in (baa^*)^*$};

\node at (4.15,0.30) {$\Leftrightarrow$};
\node at (8.1,0.30) {$\omega\in (a^*ab)^*a^* + (baa^*)^*$};

\node at (6.9,-1.20) {$L_\circ :=  (a^*ab)^*a^* + (baa^*)^*$};

\end{scope}

\end{tikzpicture}
    \caption{Computing the circular language in Example~\ref{ex:sofic-system-003}. }
    \label{sofic-0022}
\end{figure}

\end{example}

\begin{example}
Consider the B\"uchi automaton in Figure~\ref{sofic-0011} with the accepting set $Q_{\ac}=\{q_0\}$, see also Remark~\ref{remark_not_closed}. The $\omega$-language $L_{\h}$ is given by equation  \eqref{eq_lang_Q}. 
The associated TQFT has $\mcF(+)=\Bool q_0\oplus \Bool q_1$ and the identity decomposition 
\[
\id_+ = q_1\otimes q_1^{\ast}+q_0\otimes q_0^{\ast}=\langle  a | \otimes |(ab^2)^{\omega}\rangle + \langle ab | \otimes |b (ab^2)^{\omega} \rangle. 
\]
The circular language $L_{\circ}'$ of this TQFT is determined by the traces of $\omega\in \Sigma^{\ast}$ on $\mcF(+)$, also see equation \eqref{eq_L_circ}. A finite word $\omega\in L_{\circ}'$ if and only if any full string of consecutive $b$'s in it has even length (i.e., circular word $\omega$ has no subwords in $ab(b^2)^{\ast}a$). 

\end{example}

\begin{remark}
In the category $\mcCinfS$ defined in Section~\ref{subsec_accumulating} defects accumulate towards inner endpoints of both types: inward-oriented and outward-oriented endpoints. In the category $\mcC^{\h}_{\Sigma}$ defined in Section~\ref{subsec_omega} defects accumulate toward out-oriented endpoints only. In the the corresponding category considered in~\cite{IK-top-automata,GIKKL23} there is no defect accumulation at inner endpoints. In general, there are nine possible categories one can introduce, by independently specifying what is allowed at in-oriented and out-oriented inner endpoints: accumulation of defects, no accumulation of defects, or both.   

To build Boolean topological theories  for one of the nine corresponding categories one needs a suitable collection of evaluation functions for circles with finitely many defects and intervals with finitely or infinitely many defects and suitable accumulation conditions. Table~\ref{tab:nine} shows types of finite and infinite words involved in the corresponding evaluations. Likewise, to define a TQFT valued in the category of free $\Bool$-modules, one needs  suitable versions of automata and $\omega$-automata to account for various types of boundary behaviour at inner endpoints of cobordisms.  
\end{remark}

\begin{table}
    \centering
    \begin{tabular}{|c|c|c|c|}
    \hline 
     & right finite & right infinite  &  both  \\ 
    \hline 
      left finite &  $\Sigma^*$  & $\Sigma^\h$  & $\Sigma^* \sqcup \Sigma^\h$ \\
    \hline 
      left infinite  & $\Sigma^\t$  & $\Sigma^{\mathbb{Z}}$ & $\Sigma^\t \sqcup \Sigma^{\mathbb{Z}}$\\
    \hline 
     both   &  $\Sigma^* \sqcup \Sigma^\t$  &  $\Sigma^\h \sqcup \Sigma^{\mathbb{Z}}$  & $\Sigma^* \sqcup \Sigma^\h \sqcup  \Sigma^\t \sqcup \Sigma^{\mathbb{Z}}$ \\
    \hline         
    \end{tabular}
    \caption{Correspondence between pairs of positive and negative cardinality constraints and types of finite and infinite words.}
    \label{tab:nine}
\end{table} 

%Subjects to arXiv: 
%QA: Quantum Algebra,
%CT: Category Theory,  
%FL: Computer Science, Formal Languages,
%MP: Mathematical Physics, 
%DS: Dynamical Systems. 

%%%%%%%%%%%%%%%%%%%%%
%%
%%   REFERENCES 
%%
%%%%%%%%%%%%%%%%%%%%

%\addcontentsline{toc}{section}{References}
%\def\refname{}

\bibliographystyle{amsalpha}

\bibliography{top-automata}

\newcommand{\etalchar}[1]{$^{#1}$}
\providecommand{\bysame}{\leavevmode\hbox to3em{\hrulefill}\thinspace}
\providecommand{\MR}{\relax\ifhmode\unskip\space\fi MR }
\providecommand{\MRhref}[2]{%
  \href{http://www.ams.org/mathscinet-getitem?mr=#1}{#2}
}
\providecommand{\href}[2]{#2}
\begin{thebibliography}{BBEP21}

\bibitem[Adl98]{adler1998symbolic}
Roy Adler, \emph{Symbolic dynamics and {M}arkov partitions}, Bulletin of the American Mathematical Society \textbf{35} (1998), no.~1, 1--56.

\bibitem[BBEP21]{BBE21}
Marie-Pierre B\'{e}al, Jean Berstel, Soren Eilers, and Dominique Perrin, \emph{Symbolic dynamics}, Handbook of automata theory. {V}ol. {II}. {A}utomata in mathematics and selected applications, EMS Press, Berlin, 2021, pp.~987--1030.

\bibitem[BC18]{BC_Toolbox17}
Mikolaj Boja\'nczyk and Wojciech Czerwi\'nski, \emph{An {A}utomata {T}oolbox}, Online Lecture Notes \href{https://www.mimuw.edu.pl/~bojan/upload/automata-toolbox.pdf}{https://www.mimuw.edu.pl/$\sim$bojan/upload/automata-toolbox.pdf} (2018), 1--273.

\bibitem[GH83]{GH83}
John Guckenheimer and Philip Holmes, \emph{Nonlinear oscillations, dynamical systems, and bifurcations of vector fields}, Applied Mathematical Sciences, vol.~42, Springer-Verlag, New York, 1983.

\bibitem[GIK{\etalchar{+}}23]{GIKKL23}
Paul Gustafson, Mee~Seong Im, Remy Kaldawy, Mikhail Khovanov, and Zachary Lihn, \emph{Automata and one-dimensional {TQFT}s with defects}, Lett. Math. Phys. \textbf{113} (2023), no.~93, 1--38.

\bibitem[IK22]{IK-top-automata}
Mee~Seong Im and Mikhail Khovanov, \emph{Topological theories and automata}, arXiv preprint \href{https://arxiv.org/abs/2202.13398}{arXiv:2202.13398} (2022), 1--70.

\bibitem[IK23]{IK_TQFTjourney23}
\bysame, \emph{From finite state automata to tangle cobordisms: a {TQFT} journey from one to four dimensions}, arXiv preprint \href{https://arxiv.org/abs/2309.00708}{arXiv:2309.00708}, to appear in Contemporary Mathematics (2023), 1--40.

\bibitem[IK24]{IK_22_linear}
\bysame, \emph{One-dimensional topological theories with defects: the linear case}, Contemporary Mathematics \textbf{791} (2024), 105--146.

\bibitem[IKO23]{IKV23}
Mee~Seong Im, Mikhail Khovanov, and Victor Ostrik, \emph{Universal construction in monoidal and non-monoidal settings, the {B}rauer envelope, and pseudocharacters}, arXiv preprint \href{https://arxiv.org/abs/2303.02696}{arXiv:2303.02696} (2023), 1--59.

\bibitem[IZ22]{IZ}
Mee~Seong Im and Paul Zimmer, \emph{One-dimensional topological theories with defects and linear generating functions}, Involve \textbf{15} (2022), no.~2, 319--331.

\bibitem[Kit98]{Kit98}
Bruce~P. Kitchens, \emph{Symbolic dynamics}, Universitext, Springer-Verlag, Berlin, 1998, One-sided, two-sided and countable state Markov shifts.

\bibitem[Kum]{Kumar_notes}
K.~Narayan Kumar, \emph{Notes on {A}utomata, {L}ogic, {G}ames and {A}lgebra}, Online lecture notes, \href{https://www.cmi.ac.in/~kumar/research.html}{https://www.cmi.ac.in/$\sim$kumar/research.html}.

\bibitem[LM21]{lind2021introduction}
Douglas Lind and Brian Marcus, \emph{An introduction to symbolic dynamics and coding}, second ed., Cambridge Mathematical Library, Cambridge University Press, Cambridge, 2021.

\bibitem[PP04]{PePi04}
Dominique Perrin and Jean-\'Eric Pin, \emph{Infinite {W}ords: {A}utomata, {S}emigroups, {L}ogic and {G}ames}, vol. 141, Pure and Applied Mathematics, Academic Press, 2004.

\bibitem[Sta83]{Sta83}
Ludwig Staiger, \emph{Finite-state {$\omega $}-languages}, J. Comput. System Sci. \textbf{27} (1983), no.~3, 434--448.

\bibitem[Sta97]{Sta97}
\bysame, \emph{{$\omega$}-languages}, Handbook of formal languages, {V}ol. 3, Springer, Berlin, 1997, pp.~339--387.

\bibitem[Wei73]{weiss1973subshifts}
Benjamin Weiss, \emph{Subshifts of finite type and sofic systems}, Monatshefte f{\"u}r Mathematik \textbf{77} (1973), no.~5, 462--474.

\bibitem[Wil04]{Williams04}
Susan~G. Williams, \emph{Introduction to symbolic dynamics}, Symbolic dynamics and its applications, Proc. Sympos. Appl. Math., vol.~60, Amer. Math. Soc., Providence, RI, 2004, pp.~1--11.

\bibitem[Wil21]{Wilke21}
Thomas Wilke, \emph{{$\omega$}-automata}, Chapter 6 in Handbook of {A}utomata {T}heory. {V}ol. {I}. {T}heoretical foundations, EMS Press, Berlin, 2021, Revised by Sven Schewe, pp.~189--234.

\end{thebibliography}

\end{document}